\documentclass[12pt,reqno]{amsart}
\usepackage[letterpaper,total={6.5in,9in},top=1in, left=1in, right=1in, bottom=1in]{geometry}
\usepackage{graphicx} 

\usepackage[utf8]{inputenc}
\usepackage{color}
\usepackage[headings]{fullpage}
\usepackage{enumerate}
\usepackage{amsmath, amsthm} 
\usepackage{amsfonts}
\usepackage{amssymb}
\usepackage{mathrsfs}
\usepackage{fullpage} 
\usepackage{xcolor}
\usepackage{bbold}
\usepackage{tikz}
\usepackage{stmaryrd}
\usepackage{adjustbox}
\usepackage{lineno}
\usepackage{libertine}
\usepackage[center]{caption}
\usepackage{subcaption} 
\usepackage[allcolors = blue,colorlinks,linktocpage=true]{hyperref}
\usepackage[abs]{overpic}

\usetikzlibrary{cd}

\title{Two non-semisimple quantum invariants of lens spaces}
\author{Rhea Palak Bakshi \and Mitchell Jubeir \and Quinn T. Kolt}
\address{Department of Mathematics, University of California, Santa Barbara, CA 93106, USA}
\email{\rm rheapalak@math.ucsb.edu $\mid$ rheapalakbakshi@gmail.com}
\email{\rm mitchell\_jubeir@ucsb.edu}
\email{\rm quinn@math.ucsb.edu}

\DeclareMathOperator{\Hom}{Hom}
\DeclareMathOperator{\End}{End}

\DeclareMathOperator{\Rep}{Rep}
\DeclareMathOperator{\Id}{Id}
\DeclareMathOperator{\id}{id}
\def\sl{\mathfrak{sl}}
\DeclareMathOperator{\ev}{ev}
\DeclareMathOperator{\ab}{ab}
\DeclareMathOperator{\coev}{coev}

\DeclareMathOperator{\Proj}{Proj}
\DeclareMathOperator{\opp}{op}

\DeclareMathOperator{\colim}{colim}

\DeclareMathOperator{\mcS}{\mathcal{S}}

\newtheorem{theorem}{Theorem}[section]
\newtheorem{prop}[theorem]{Proposition}
\newtheorem{lemma}[theorem]{Lemma}

\theoremstyle{definition}
\newtheorem{definition}[theorem]{Definition}
\newtheorem{remark}[theorem]{Remark}

\renewcommand{\Vec}{\mathrm{Vec}}
\renewcommand{\epsilon}{\varepsilon}

\def\cC{\mathcal{C}}
\def\mcC{\mathcal{C}}

\def\bom{{\boldsymbol\omega}}

\def\cK{\mathcal{K}}
\def\Pen{\mathcal{P}}

\def\cI{\mathcal{I}}
\def\cS{\mathcal{S}}
\def\SM{\cS_{\cI}}

\def\Ku{\mathscr{K}}
\def\bC{\Bbb{C}}
\def\bZ{\mathbb{Z}}
\def\bk{\Bbbk}

\def\one{\Bbb{1}}
\def\incl{\hookrightarrow}
\def\proj{\twoheadrightarrow}

\date{}


\begin{document}

\begin{abstract}
    We calculate the chromatic spherical invariant $\Ku_{\Rep(H)}(L(p,q))$  of lens spaces $L(p, q)$ for arbitrary finite-dimensional unibalanced, unimodular Hopf algebras $H$. In the case of the index-2 Nichols Hopf algebra $\cK_2$ and $q=1$, we find that the invariant of $L(p, 1)$ depends only on the parity of $p$ for $p\geq 1$. We also describe the duals $\SM(M)^*$ of the admissible skein modules of handlebodies $M=H_g$ and of $M=L(p, q)$ for arbitrary tensor ideals $\cI$ of a ribbon tensor category. In particular, $\SM(L(p, q))^*$ corresponds with cyclic traces on $\cI$ satisfying a property reminiscent of the modified trace condition.
\end{abstract}

\maketitle
\tableofcontents

\section{Introduction}

The theory of skein modules was introduced independently by J\'ozef H. Przytycki \cite{skeinprzytycki} and Vladimir Turaev \cite{skeinturaev}. Their work provides a powerful framework for generalizing polynomial link invariants, originally defined for the $3$-sphere, to the broader context of arbitrary $3$-manifolds. For a commutative ring $R$ with unity, a skein module of an orientable $3$-manifold is obtained by quotienting the free $R$-module generated by ambient isotopy classes of links in that manifold by local ``skein relations." These modules capture the topology and geometry of $3$-manifolds and provide a bridge between low-dimensional topology, knot theory, and quantum algebra. The Kauffman bracket skein module (KBSM) is arguably the most extensively studied skein module of all and has deep relations to character varieties \cite{sl2cbullock,ps1}, Teichmüller spaces \cite{bw1,fokchekhov,lequantum,kashaev}, cluster algebras \cite{mullerskein}, the Witten-Reshetikhin-Turaev $3$-manifold invariants, TQFTs \cite{BHMV2, BHMV1, Li4, sikoratqft}, and the representation theory of quantum groups, for example, the KBSM is connected to the representation category of the quantum group $U_q(\mathfrak{sl}_2)$.

At generic values of the quantum parameter $q$, the representation categories of quantum groups, such as $U_q(\mathfrak{sl}_2)$, are semisimple, and the categorical (quantum) trace gives nonzero quantum dimensions of representations, allowing for a meaningful Reshetikhin-Turaev type construction of quantum invariants of $3$-manifolds and links. However, when $q$ is specialized to a root of unity, these categories are no longer semisimple, the categorical trace becomes degenerate, and many of the standard quantum invariants trivialize. To overcome this, Geer, Patureau-Mirand, Kujawa, Turaev, and Virelizier {\cite{gkp,gptrace,gpttrace} introduced a generalization of the categorical trace, called the  modified trace, which is non-zero on projectives even when the categorical trace vanishes. Their work led to the construction of non-semisimple $3$-manifold quantum invariants, known as renormalized quantum invariants (RQIs) \cite{costantino_skein_2023,cgprqi,tvunimodular,gpttrace}, and non-semisimple TQFTs \cite{nsstqft1, nsstqft2}. 

Thus, classical skein modules are inherently tied to semisimple categories and do not capture the richer structures \cite{reutter} arising in non-semisimple quantum topology, in particular those associated with the RQIs of $4$-manifolds, such as the skein lasagna module \cite{mww,renwillis}. Motivated by the need to
extend skein-theoretic constructions beyond the semisimple setting, Costantino, Geer and Patureau-Mirand \cite{admissible} introduced {\it admissible skein modules.} These are defined relative to an ideal $\cI$ of a
pivotal or ribbon $\Bbbk$-category $\cC$, typically the ideal of projective objects, and are built from ribbon graphs that contain an edge colored by the ideal in
every connected component, with skein relations from $\cI$, not all of $\cC$. The resulting modules
remain non-trivial in the non-semisimple setting and are naturally related to modified traces.  

Due to the complexity of renormalized quantum invariants and of admissible skein modules, it is difficult to compare their computational power relative to their classical counterparts and few explicit examples exist in the literature. In our paper, we explicitly compute the chromatic spherical invariant and the admissible skein module for lens spaces. An algorithm for the chromatic spherical invariant exists in the case of spherical Hopf algebras \cite{reina}. We do not assume sphericality.

 Our paper is organized as follows. In Section~\ref{sec: prelim}, we
recall the background needed for our main results. Subsection~\ref{subsec:handles}
reviews handle decompositions, handlebodies, and lens spaces, Subsection~\ref{subsec:categories} collects the necessary notions from category theory and Hopf algebras, including tensor ideals and cyclic and
modified traces, and Subsection~\ref{subsection:icolouredgraphs} introduces
colored and admissible ribbon graphs. In Section~\ref{section:kuperbergtype}, we study the chromatic spherical invariant. After recalling its definition in
terms of Heegaard splittings and bichrome diagrams in
Subsection~\ref{subsec:chromatic-def}, we compute it in
Subsection~\ref{subsec:chromatic-lens} for all lens spaces $L(p,q)$ and representation categories of finite-dimensional unibalanced, unimodular Hopf algebras (see Theorem~\ref{thm:chromatic-lens}). For the index-$2$ Nichols Hopf algebra, we show that the invariant of $L(p,1)$ depends only on the
parity of $p$ (see Theorem~\ref{thm:K2-calculation}). In Section~\ref{section:admissible}}, we turn to admissible skein modules. In Subsection~\ref{subsec:handlebodies}, we describe the admissible skein modules of handlebodies, both as direct limits of endomorphism algebras (see Theorem~\ref{limit-skein-module-iso}) and dually in terms of trace-like families of functionals on $\cI^{\otimes g}$ (see Proposition~\ref{prop:semitrace-correspondence}). In Subsection \ref{subsec:admissible-lens} we compute the dual of the admissible skein module of lens spaces $L(p,q)$ (see Theorem~\ref{thm:admissible-lens}). Finally, in
Section~\ref{section:future}, we discuss some future directions that our work can take.

\subsection{Acknowledgements}  R.P.B. was partially supported by the project PID2024-157173NB-I00 funded by MCIN/AEI/10.13039/501100011033 and by Fondo Europeo de Desarrollo Regional (The European Regional Development Fund), The European Union. Q.T.K. is supported by the National Science Foundation Graduate Research Fellowship Program under Grant No. 2139319.\footnote {Any opinions, findings, and conclusions or recommendations expressed in this paper are those of the authors and do not necessarily reflect the views of the National Science Foundation.} The authors would like to thank Nathan Geer for helpful discussions.

\subsection{Statement on the use of AI} Some of the diagrams in this paper were prepared with the assistance of Claude. All diagrams were checked by the authors, who take full responsibility for their correctness. All mathematical content, including all
results and proofs, and all of the writing are the authors' own.

\section{Preliminaries}\label{sec: prelim}

In this section we discuss the topology and category theory needed for our main results.

\subsection{Handle addition and lens spaces}\label{subsec:handles}

Let $(M,\partial M)$ be an oriented $3$-manifold with boundary $\partial M$ and $W$ be a codimension $1$ submanifold of $M$ such that $W \hookrightarrow \partial M$. A handle of index $k$, or a $k$-handle, is a $3$-ball $D^k \times D^{3-k}$ glued to $W$ along $S^{k-1} \times D^{3-k}$ to obtain a new manifold, say $M'$. Explicitly, a $0$-handle is a $3$-ball glued to $W$ along the empty set, a $1$-handle is a solid tube $D^1 \times D^{2}$ glued along its boundary discs $S^{0} \times D^{2}$ to $W$, a $2$-handle is a ``plate'' $D^2 \times D^1$ glued along its annular boundary to $W$, and a $3$-handle is a $3$-ball glued to $W$ along $S^2 \times D^0$. A handle decomposition of a manifold $M$ is its decomposition into $k$-handles, the union of which is $M$. While constructing  manifolds, we glue handles of lower indices first, and handles of the same index are added simultaneously.  \\

A handlebody $H_g$ is a connected, compact, orientable $3$-manifold with boundary whose handle decomposition consists of one $0$-handle and $g$ $1$-handles. The number $g$ of $1$-handles in the handle decomposition is called the genus of $H_g$. A seminal result of Smale, the proof of which uses Morse theory, is that every closed orientable $3$-manifold admits a handle decomposition. A lens space is a closed, connected,  orientable $3$-manifold whose handle decomposition consists of exactly one handle of each index; the gluing is as follows. A $1$-handle is first attached to a $0$-handle to obtain a solid torus. Then a $2$-handle is glued to this solid torus such that the core of the $2$-handle is glued to a $(p,q)$-curve on the boundary of the solid torus. The resultant manifold has a $2$-sphere boundary, gluing a $3$-handle to which produces the lens space $L(p,q)$. See Figure \ref{fig:lens-space-decomposition} for a depiction of $2$-handle addition along the $(p,q)$-curve on the boundary of the solid torus.
\begin{figure}[ht]
    \centering
  $\vcenter{\hbox{\begin{overpic}[unit=1mm, scale = .5]{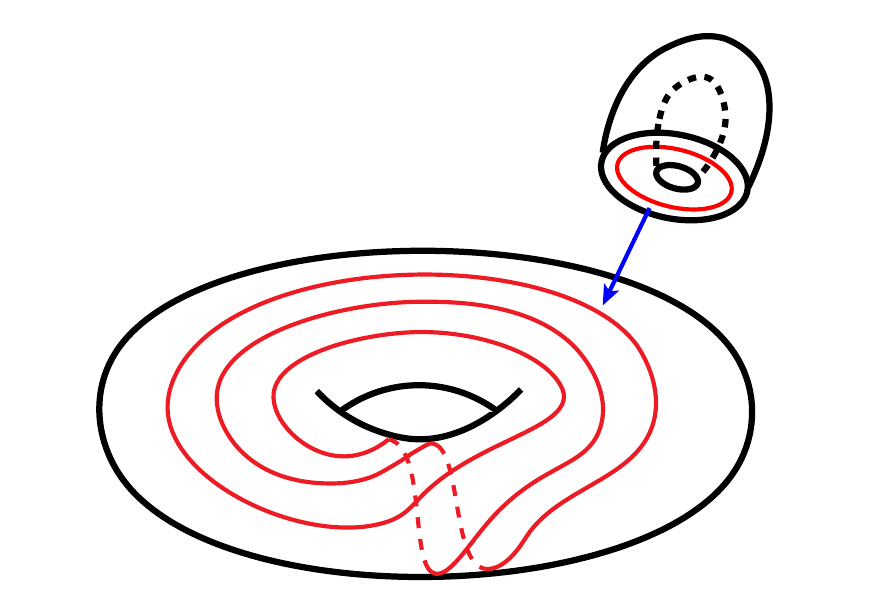}
  \put(54,28){\footnotesize{\text{glue}}}
  \put(66,38){\footnotesize{$2$\text{-handle}}}
  \put(25,2.5){\footnotesize{$T(3,2)$}}
  \end{overpic} }}  $
    \caption{$2$-handle addition along the $(3,2)$-torus curve on the solid torus, followed by gluing a 3-handle, gives us the lens space $L(3,2)$.}
    \label{fig:lens-space-decomposition}
\end{figure} 

\subsection{Category theory and Hopf algebras} \label{subsec:categories} We now define some important terminology in category theory pertinent for our paper.
\begin{definition}[Additional properties and structures on a rigid monoidal category]\label{monoidal-structures}
    We say that a monoidal category $\cC$ is
    \begin{enumerate}
        \item \textit{tensor} if $\mathcal C$ is locally finite, $\Bbbk$-linear, abelian, and rigid with a simple unit;
        \item \textit{braided} if there is a natural \textit{braiding} isomorphism $\beta:\otimes \to \otimes^{\opp}$ such that the following hexagon diagrams commute:
        \begin{center}
            \begin{tikzcd}
                &(Y\otimes X)\otimes Z\arrow[r, "\alpha_{Y,X,Z}"]&Y\otimes (X\otimes Z)\arrow[dr,"\id_Y\otimes \beta_{X,Z}"]\\
                (X\otimes Y)\otimes Z\arrow[ur,"\beta_{X,Y}\otimes\id_Z"]\arrow[dr,"\alpha_{X,Y,Z}"']&&&Y\otimes(Z\otimes X)\\
                &X\otimes (Y\otimes Z)\arrow[r,"\beta_{X,Y\otimes Z}"']&(Y\otimes Z)\otimes X\arrow[ur,"\alpha_{Y,Z,X}"']
            \end{tikzcd}

            \begin{tikzcd}
                &X\otimes (Z\otimes Y)\arrow[r, "\alpha_{X,Z,Y}^{-1}"]&(X\otimes Z)\otimes Y\arrow[dr,"\beta_{X,Z}\otimes \id_Y"]\\
                X\otimes (Y\otimes Z)\arrow[ur,"\id_X\otimes\beta_{Y,Z}"]\arrow[dr,"\alpha_{X,Y,Z}^{-1}"']&&&(Z\otimes X)\otimes Y\\
                &(X\otimes Y)\otimes Z\arrow[r,"\beta_{X\otimes Y, Z}"']&Z\otimes (X\otimes Y)\arrow[ur,"\alpha_{Z,X,Y}^{-1}"']
            \end{tikzcd}
        \end{center}
\begin{figure}[ht]
    \centering
    \begin{subfigure}[b]{0.48\textwidth}
        \centering
  $\vcenter{\hbox{\begin{overpic}[unit=1mm, scale = .3]{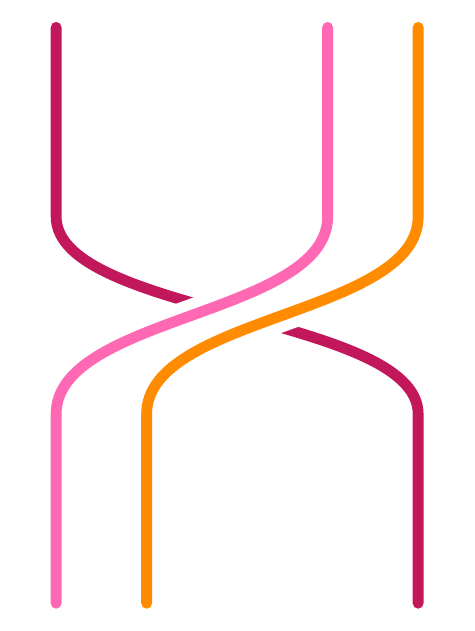}
  \put(-1,3){\footnotesize{$X$}}
    \put(8.5,3){\footnotesize{$Y$}}
        \put(18,3){\footnotesize{$Z$}}
            \put(22,25){\footnotesize{$Y$}}
            \put(13,25){\footnotesize{$X$}}
             \put(-1,25){\footnotesize{$Z$}}
  \end{overpic} }}  = \vcenter{\hbox{\begin{overpic}[unit=1mm, scale = .3]{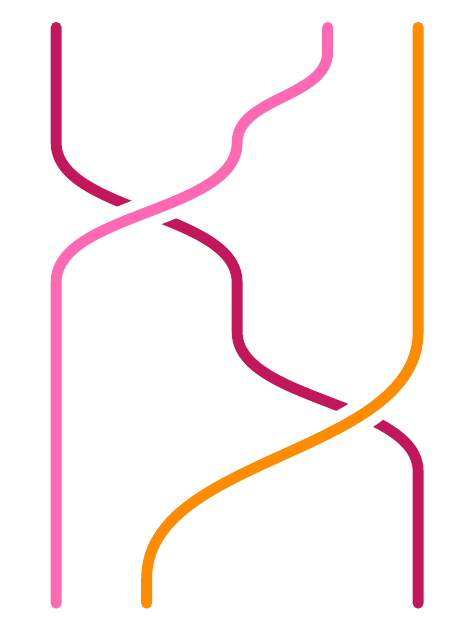}
  \put(-1,3){\footnotesize{$X$}}
    \put(9,3){\footnotesize{$Y$}}
        \put(18,3){\footnotesize{$Z$}}
            \put(22,25){\footnotesize{$Y$}}
            \put(9,25){\footnotesize{$X$}}
             \put(-1,25){\footnotesize{$Z$}}
  \end{overpic} }}$
        \caption{}
        \label{fig:hexagon-a}
    \end{subfigure}
    \hfill
    \begin{subfigure}[b]{0.48\textwidth}
        \centering
  $\vcenter{\hbox{\begin{overpic}[abs, unit=1mm, scale = .3]{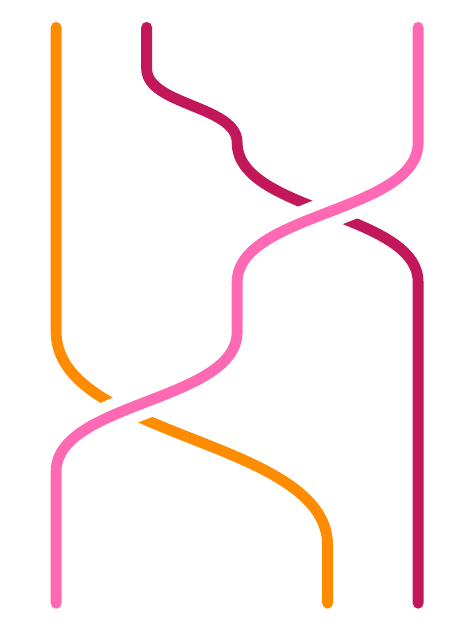}
  \put(-1,3){\footnotesize{$X$}}
    \put(17.6,3){\footnotesize{$Y$}}
        \put(22.3,3){\footnotesize{$Z$}}
            \put(-1,27){\footnotesize{$Y$}}
            \put(3.9,27){\footnotesize{$Z$}}
             \put(22.3,27){\footnotesize{$X$}}
  \end{overpic} }}
  \quad = \quad
\vcenter{\hbox{\begin{overpic}[abs, unit=1mm, scale = .3]{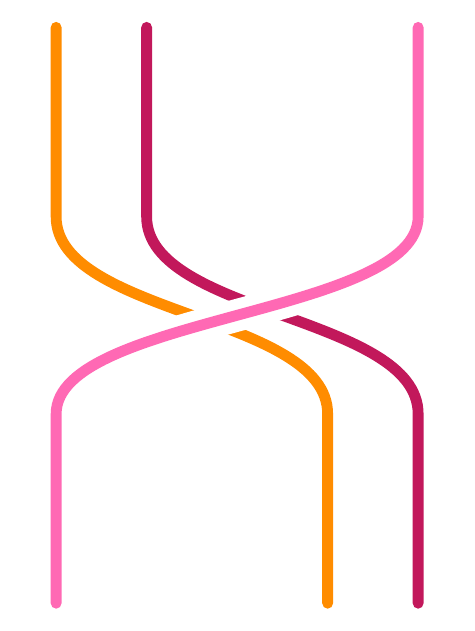}
  \put(-1,3){\footnotesize{$X$}}
    \put(17.6,3){\footnotesize{$Y$}}
        \put(22.3,3){\footnotesize{$Z$}}
            \put(-1,27){\footnotesize{$Y$}}
            \put(3.9,27){\footnotesize{$Z$}}
             \put(22.3,27){\footnotesize{$X$}}
  \end{overpic} }}  $
        \caption{}
        \label{fig:hexagon-b}
    \end{subfigure}
    \caption{Diagrammatic interpretation of the hexagon diagrams.}
    \label{fig:hexagon}
\end{figure}
        
        \item \textit{pivotal} if $\cC$ is rigid and there is a monoidal natural \textit{pivotal structure} isomorphism $\phi:\Id_\cC\Rightarrow (-)^{**}$;
        \item \textit{ribbon} if $\mathcal C$ is braided, rigid, and there is a natural \textit{ribbon twist} isomorphism $\theta:\Id_{\cC}\Rightarrow\Id_{\cC}$ such that 
        \begin{enumerate}
            \item $\theta_\one = \id_\one$;
            \item $\theta_{X^*} = \theta_X^*$ for any $X\in \cC$;
            \item the following diagram commutes for every $X, Y\in\cC$ (the diagrammatic interpretation is given alongside): 
            $$\begin{tikzcd}
                X\otimes Y \arrow[rr, "{\theta_{X\otimes Y}}"] \arrow[d,"{\beta_{X,Y}}"'] && X\otimes Y\\
                Y\otimes X\arrow[rr,"{\beta_{Y,X}}"'] && X\otimes Y\arrow[u,"{\theta_X\otimes\theta_Y}"']
\end{tikzcd}\vcenter{\hbox{\begin{overpic}[abs, unit=1mm, scale = .3]{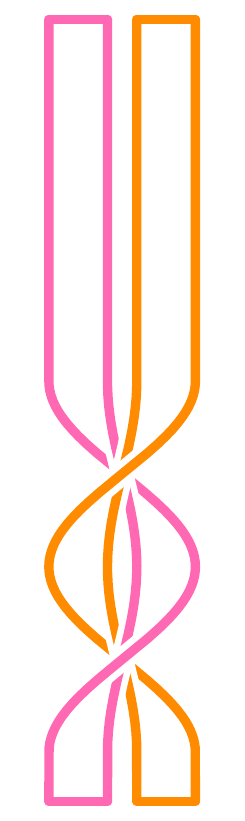}
  \put(-0.7,37.5){\footnotesize{$X$}}
    \put(10.6,37.5){\footnotesize{$Y$}}
  \end{overpic} }}
  \quad = \quad
\vcenter{\hbox{\begin{overpic}[abs, unit=1mm, scale = .3]{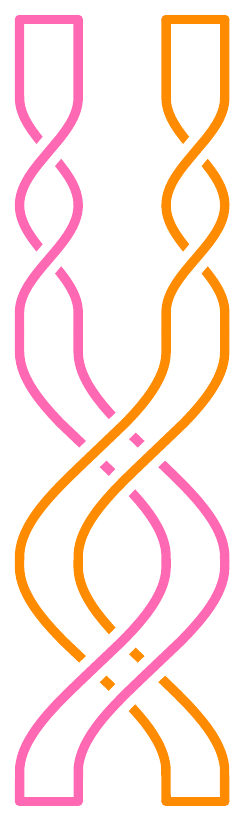}
  \put(-2.2,37.5){\footnotesize{$X$}}
    \put(12.1,37.5){\footnotesize{$Y$}}
  \end{overpic} }}$$
           
           \end{enumerate}
        \item \textit{unimodular} if $\cC$ is a tensor category, and the projective cover $P(\one)$ of the monoidal unit $\one$ is self-dual.
    \end{enumerate}
\end{definition}

\begin{definition}[Additional properties and structures on a Hopf algebra]\label{hopf-structures}
    We say that a Hopf algebra $H$ is
    \begin{enumerate}
        \item \textit{quasitriangular} if there is an invertible \textit{$R$-matrix} $R\in H\otimes H$ such that
        \begin{align*}
            R\Delta(x)R^{-1} &= \Delta^{\mathrm{op}}(x), & (\Delta\otimes\id_H)(R) &= R_{13}R_{23}, & (\id_H\otimes\Delta)(R) &= R_{13}R_{12},
        \end{align*}
        where $R_{12} = R^{(1)}\otimes R^{(2)}\otimes 1$, $R_{13} = R^{(1)}\otimes 1\otimes R^{(2)}$, and $R_{23} = 1\otimes R^{(1)}\otimes R^{(2)}$;
        \item \textit{pivotal} if there is a \textit{pivot} group-like $g\in H$ (i.e., nonzero $g\in H$ with $\Delta(g)=g\otimes g$) such that $S^2(x) = gxg^{-1}$ for all $x\in H$;
        \item \textit{ribbon} if $H$ is quasitriangular and there is an invertible central \textit{ribbon element} ${\mathbf{v}}\in Z(H)$ such that
        \begin{align*}
            {\mathbf{v}}^2 &= uS(u), & \Delta({\mathbf{v}}) &= (R_{21}R)^{-1}({\mathbf{v}}\otimes{\mathbf{v}}), & S({\mathbf{v}}) &= {\mathbf{v}}, & \epsilon({\mathbf{v}}) &= 1,
        \end{align*}
        where $u = S(R^{(2)})R^{(1)}$ is the \textit{Drinfeld element} and $R_{21} = R^{(2)}\otimes R^{(1)}$;
        \item \textit{unimodular} if left (i.e., $x\Lambda = \epsilon(x)\Lambda$) and right (i.e., $\Lambda' x = \epsilon(x)\Lambda'$) integrals coincide.
        \item \textit{counimodular} if left (i.e., $(\lambda' \otimes id)\Delta(x)=\lambda'(x)1$) and right (i.e., $(id \otimes \lambda)\Delta(x)=\lambda(x)1$) cointegrals coincide.
    \end{enumerate}
\end{definition}

    \begin{prop}
        All ribbon Hopf algebras are pivotal with $g = \mathbf{v}^{-1}u$.
    \end{prop}

    Given a Hopf algebra $H$ over a field $\bk$, the category $\Rep(H)$ of finite-dimensional representations of $H$ (that is, $H$-modules) is a tensor category with a fiber functor to $\Vec$, the category of finite-dimensional vector spaces. The additional structures described in Definition \ref{hopf-structures} also correspond with categorical data as described in the following proposition \cite{EGNO}.
    \begin{prop}[Tannaka duality]
        We have the following correspondences relating $H$ and $\Rep(H)$:
        \begin{enumerate}
            \item $R$-matrices on $H$ are in one-to-one correspondence with braidings on $\Rep(H)$, where $\beta_{M, N}:m\otimes n \mapsto R_{21}\cdot (n\otimes m)$;
            \item pivots on $H$ are in one-to-one correspondence with pivotal structures on $\Rep(H)$, where $\phi_M:m\mapsto (\phi\mapsto \phi(g\cdot m))$;
            \item ribbon elements on $H$ given an $R$-matrix $R$ are in one-to-one correspondence with ribbon structures on $\Rep(H)$ given the braiding corresponding to $R$, where $\theta_M:m\mapsto \mathbf{v}^{-1}\cdot m$;
            \item $H$ is unimodular if and only if $\Rep(H)$ is unimodular.
        \end{enumerate}
    \end{prop}

    \begin{definition}[Tensor ideal]
        Let $\cI$ be a full subcategory of a tensor category $\cC$. Then, $\cI$ is a \textit{tensor ideal} of $\cC$ if $\cI$ is closed under tensor product with objects in $\cC$ and retracts.
    \end{definition}

We recall the notion of a modified trace on a tensor ideal of a tensor category, following Geer, Kujawa, Patureau-Mirand, and Virelizier \cite{modifiedtrace1, modifiedtrace2}.

    \begin{definition}[Cyclic and modified traces] Let $\mathcal C$ be a tensor category and $X, Y \in \mathcal{C}$. The \emph{right partial trace} of
$f \in \operatorname{End}_{\mathcal{C}}(X \otimes Y)$ is
\[
  \operatorname{ptr}_r(f)
  = (\mathrm{id}_X \otimes \overrightarrow{\operatorname{ev}}_Y)
    \circ (f \otimes \mathrm{id}_{Y^*})
    \circ (\mathrm{id}_X \otimes \overleftarrow{\operatorname{coev}}_Y) = \vcenter{\hbox{\begin{overpic}[scale = 0.04]{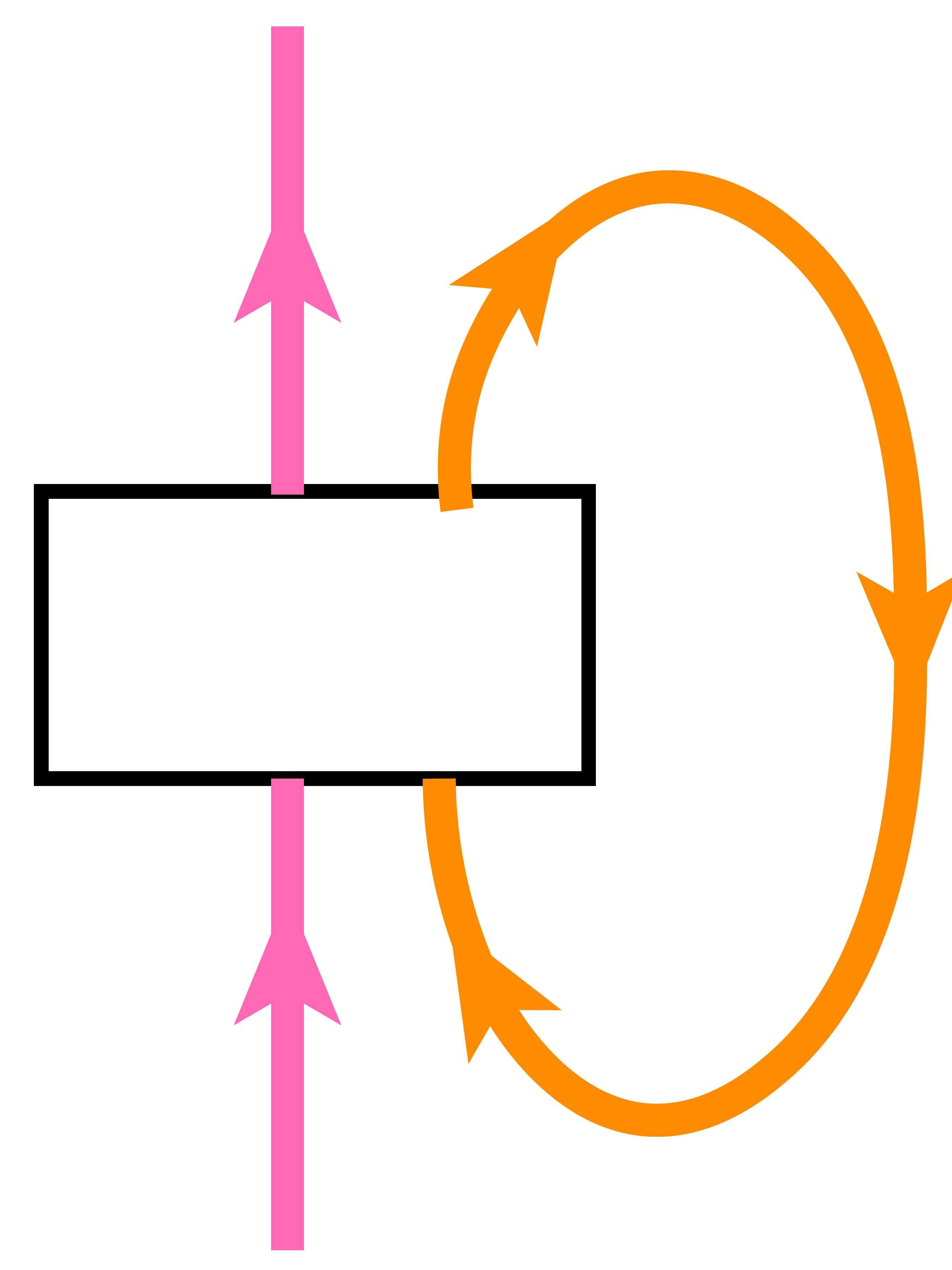}
\put(14,32){\tiny{$f$}}
\put(6,13){\tiny{$X$}}
\put(6,50){\tiny{$X$}}
\put(40,32){\tiny{$Y$}}
\end{overpic} }}
  \in \operatorname{End}_{\mathcal{C}}(X).
\]
The \emph{left partial trace} of
$f \in \operatorname{End}_{\mathcal{C}}(Y \otimes X)$ is
\[
  \operatorname{ptr}_l(f)
  = (\overleftarrow{\operatorname{ev}}_Y \otimes \mathrm{id}_{X})
    \circ (\mathrm{id}_{Y^*} \otimes f)
    \circ (\overrightarrow{\operatorname{coev}}_Y \otimes \mathrm{id}_X)= \vcenter{\hbox{\begin{overpic}[scale = 0.04]{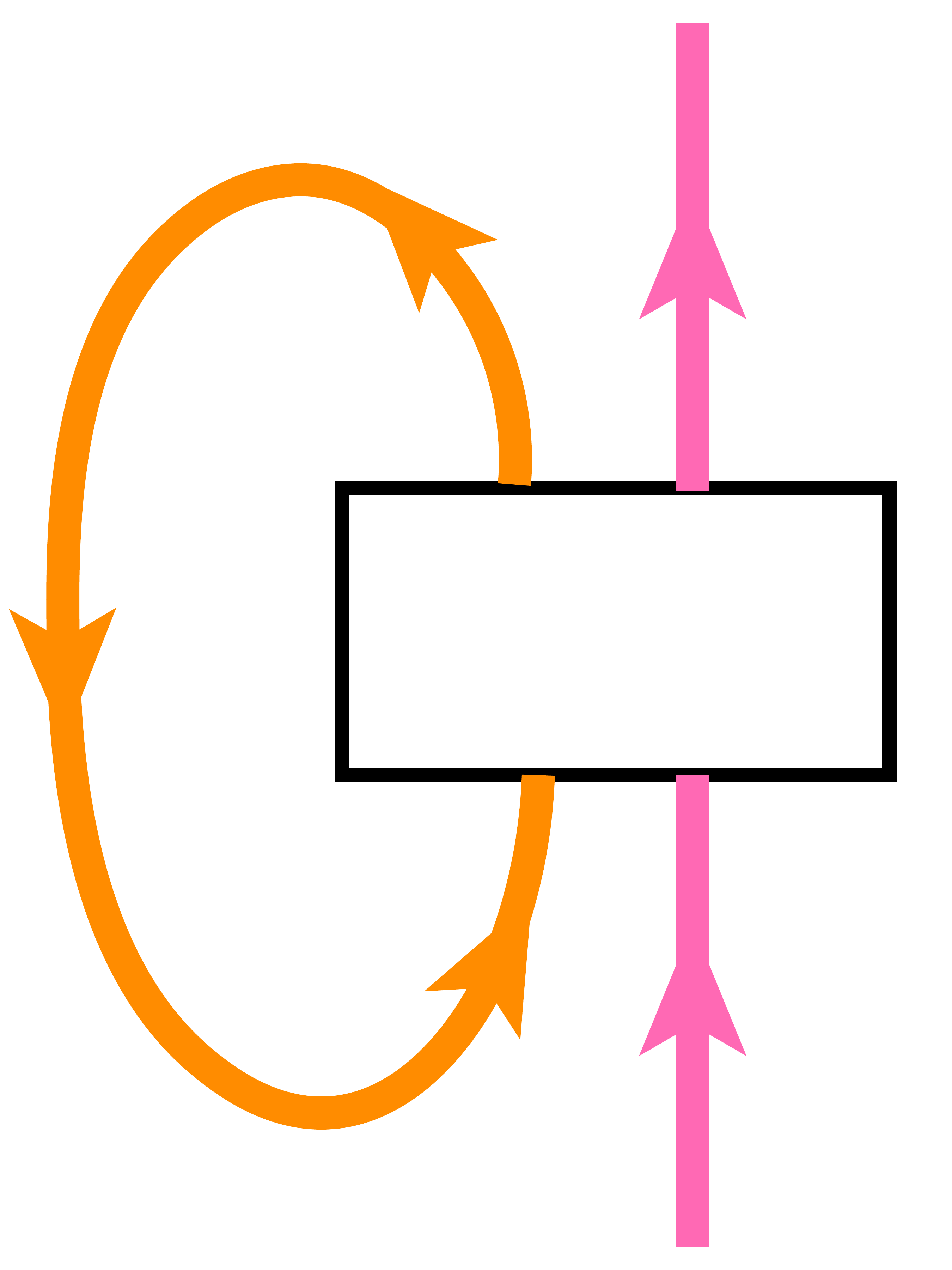}
\put(30,32.5){\tiny{$f$}}
\put(40,13){\tiny{$X$}}
\put(40,50){\tiny{$X$}}
\put(6.5,32){\tiny{$Y$}}
\end{overpic} }}
  \in \operatorname{End}_{\mathcal{C}}(X).
\]

A family of $\Bbbk$-linear maps $\{\, \mathsf{t}_P \colon \operatorname{End}_{\mathcal{C}}(P) \to \Bbbk
  \,\}_{P \in \mathcal{I}}$ on a tensor ideal $\mathcal I$ of $\cC$ is: 

  \begin{enumerate}[(i)]
  
      \item a \textit{cyclic trace} if, for all $P, Q \in \mathcal{I}$ and all morphisms
        $f \colon Q \to P$ and $g \colon P \to Q$, $$t_P\left(\vcenter{\hbox{\begin{overpic}[scale = 0.04]{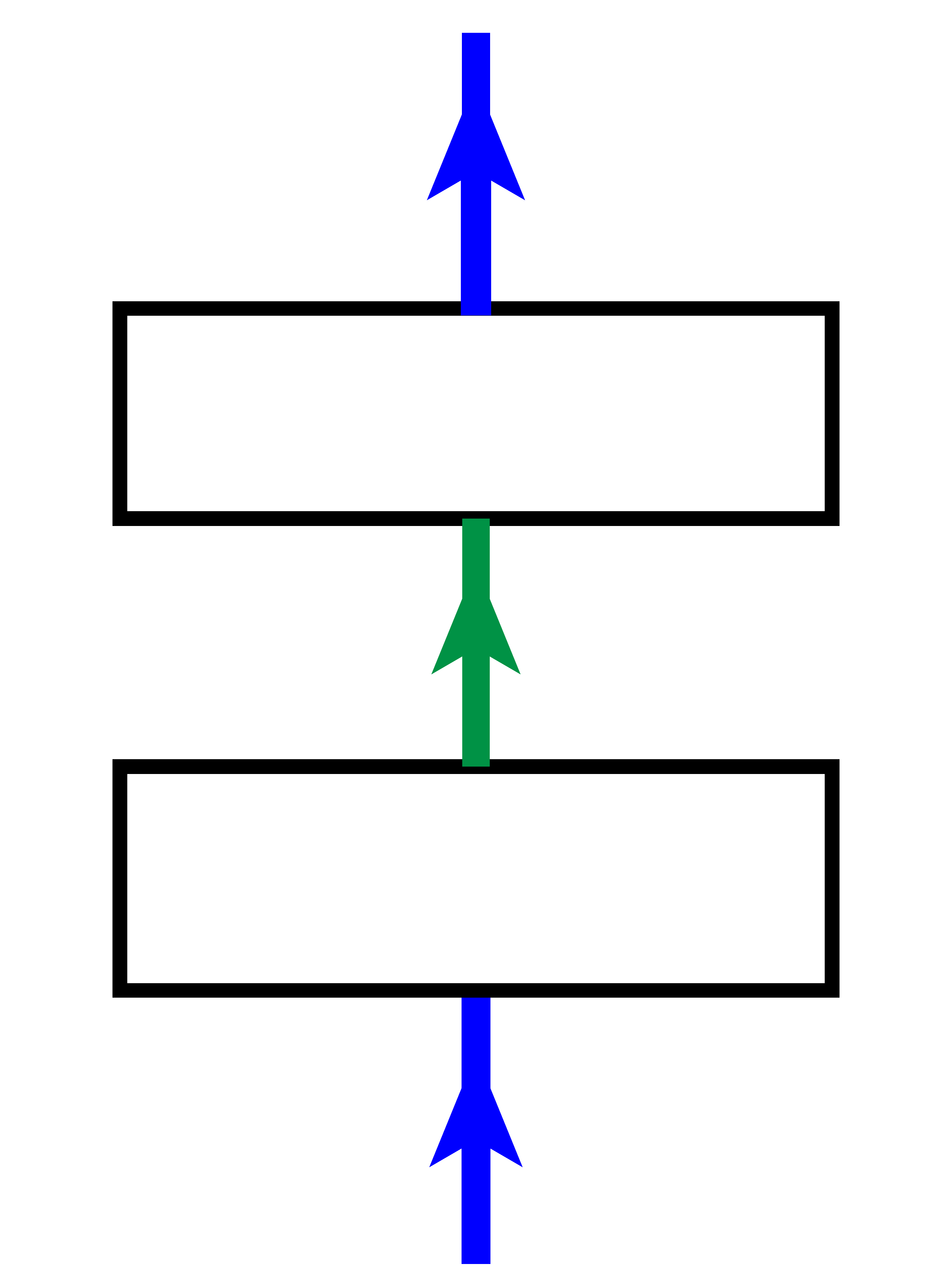}
\put(30,32.5){\tiny{$Q$}}
\put(24,44.5){\tiny{$f$}}
\put(24,20){\tiny{$g$}}
\put(30,56){\tiny{$P$}}
\put(30,6){\tiny{$P$}}
\end{overpic} }}\right)=\mathsf{t}_P(f \circ g) = \mathsf{t}_Q(g \circ f) = t_Q\left(\vcenter{\hbox{\begin{overpic}[scale = 0.04]{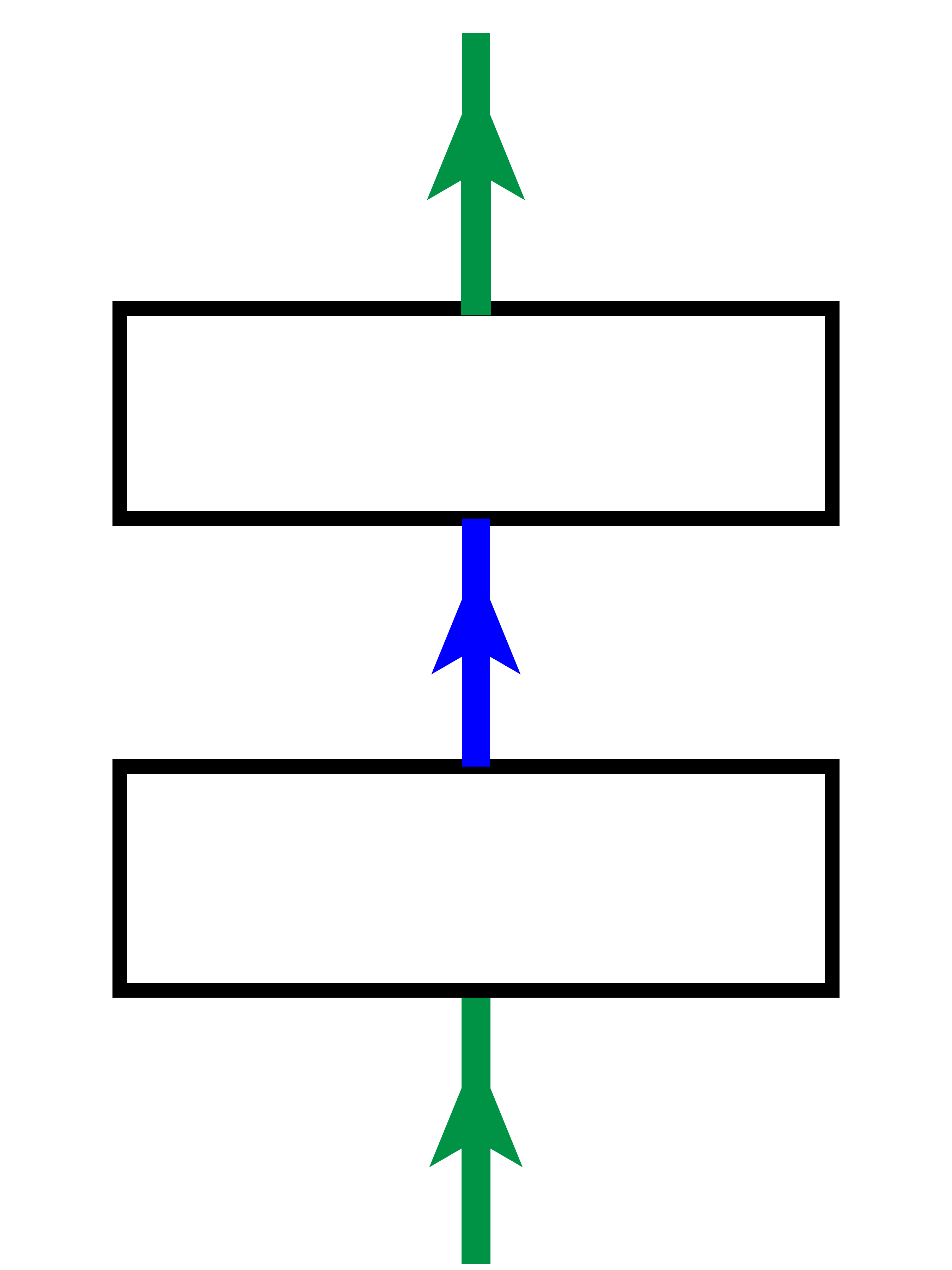}
\put(30,32.5){\tiny{$P$}}
\put(24,45){\tiny{$g$}}
\put(24,19.5){\tiny{$f$}}
\put(30,56){\tiny{$Q$}}
\put(30,6){\tiny{$Q$}}
\end{overpic} }}\right), \ \text{and}$$ 
    \item a {\it right modified trace}, also called a \textit{right m-trace}, if $t$ is a cyclic trace and, for all $P \in \mathcal{I}$,
        $X \in \mathcal{C}$, and
        $f \in \operatorname{End}_{\mathcal{C}}(P \otimes X)$,
        $$t_{P\otimes X}\left(\vcenter{\hbox{\begin{overpic}[scale = 0.04]{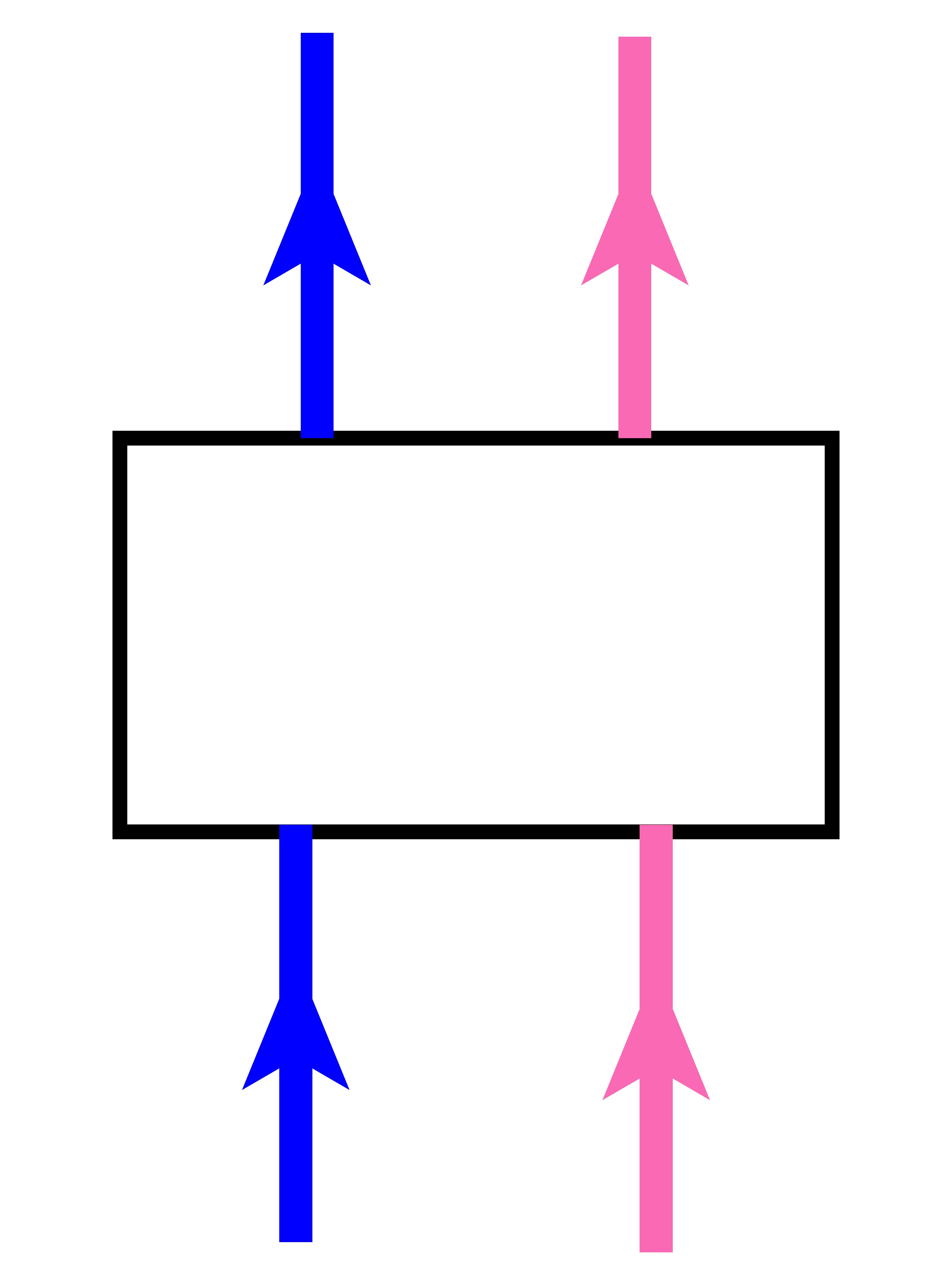}
\put(7,10){\tiny{$P$}}
\put(39,10){\tiny{$X$}}
\put(23,32){\tiny{$f$}}
\put(7,55){\tiny{$P$}}
\put(38,55){\tiny{$X$}}
\end{overpic} }}\right)=\mathsf{t}_{P \otimes X}(f)
                    = \mathsf{t}_P\bigl(\operatorname{ptr}_r(f)\bigr)=t_P\left(\vcenter{\hbox{\begin{overpic}[scale = 0.04]{rightpartialtrace1}
\put(14,32){\tiny{$f$}}
\put(6,13){\tiny{$P$}}
\put(6,50){\tiny{$P$}}
\put(41,32){\tiny{$X$}}

\end{overpic} }}\right).$$
      
  \end{enumerate}

 A {\it left m-trace} may be defined similarly. An {\it m-trace} on $\mathcal I$ is a right m-trace that is also a left m-trace. An m-trace $\{t_P\}_{P\in\cI}$ is \textit{non-degenerate} if for every object $P\in \cI$, the pairing $\Hom(\one\to P)\otimes \Hom(P\to \one)\to \bk$ induced by $t_P$, generated by $(f\otimes g)\mapsto t_P(f\circ g)$, is non-degenerate.

\end{definition}
In a ribbon category, left and right m-traces agree and hence are always m-traces. Modified traces and cyclic traces are essential in the computations for admissible skein modules and chromatic spherical invariant. For example, the dual of the admissible skein module of \(S^2\) is naturally isomorphic to the space of modified traces on \(\mathcal I\) (see \cite{admissible} and Section \ref{section:admissible}).

\subsection{$\cC$- and $\cI$-colored graphs}\label{subsection:icolouredgraphs}

\begin{definition} 
    Let $\cC$ be a pivotal tensor category with a tensor ideal $\cI$. An {\it $\mathcal I$-colored graph} in a 2-manifold $M$ (or 3-manifold) is an oriented multigraph embedded in $M$ (or $\partial M$) up to isotopy, whose edges are non-intersecting and colored by objects in $\mathcal I$ and whose vertices are thickened to coupons that are colored by morphisms in $\mathcal I$ which are compatible with the object labels.
    
    Let $\cC$ be a ribbon tensor category with a tensor ideal $\cI$. An {\it $\mathcal I$-colored ribbon graph} in a 3-manifold $M$ (or $2$-manifold) is an oriented ribbon multigraph embedded in $M$ (or $M\times[0,1]$) up to isotopy, whose ribbons are non-intersecting and colored by objects in $\mathcal I$ and whose coupons are colored by morphisms in $\mathcal I$ which are compatible with the object labels.
\end{definition}
Note that any $\cI$-colored graph/ribbon graph is, in particular, a $\cC$-colored graph/ribbon graph. Given a $\cC$-colored graph $\Gamma$ in a square $[0,1]^2$ with distinguished directions, possibly with endpoints of strings on the top/bottom boundary, the Penrose functor $\Pen_2$ assigns $\Gamma$ to its corresponding morphism in $\cC$ prescribed by the graphical calculus. Let $\pi:[0,1]^3\to[0,1]^2$ be the projection of the cube onto its back face. The projection $\pi$ defines a map $\pi^*$ from $\cI$-colored ribbon graphs in $[0,1]^3$ to $\cI$-colored graphs in $[0,1]^2$: isotope $\Gamma$ so that $\pi(\Gamma)$ has transverse self-intersection, then replace ribbons by edges, crossings by coupons colored by the braiding isomorphism, and ribbon twists by coupons colored by the twist isomorphism. Given a $\cC$-colored ribbon graph $\Gamma$ in a cube $[0,1]^3$, the Reshetikhin-Turaev functor is defined as the composition $\Pen_3 = \Pen_2\circ\pi^*$. An example of evaluating $\Pen_3$ on a ribbon graph embedded in $[0,1]^3$ is given in Figure \ref{fig:rt-example}.
\begin{figure}[ht]
  \centering
  $\vcenter{\hbox{\begin{overpic}[abs, unit=1mm, scale = .24]{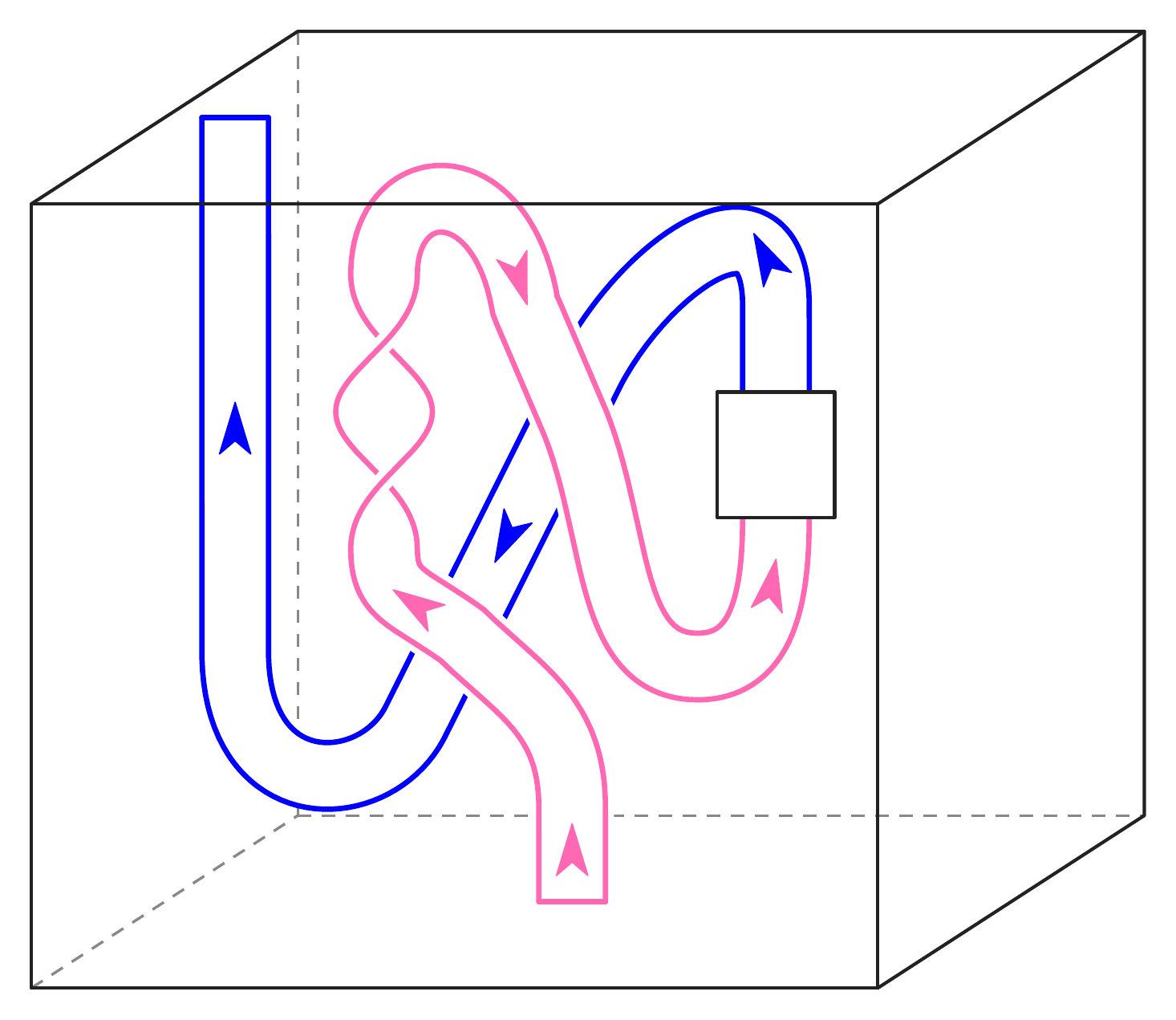}
    \put(8.6,33.8){\makebox(0,0){\footnotesize{$P$}}}
    \put(32.2,8.6){\makebox(0,0){\footnotesize{$X$}}}
    \put(39.3,28.6){\makebox(0,0){\footnotesize{$F$}}}
  \end{overpic} }}
  \quad \overset{\pi^*}{\longmapsto} \quad
  \vcenter{\hbox{\begin{overpic}[abs, unit=1mm, scale = .24]{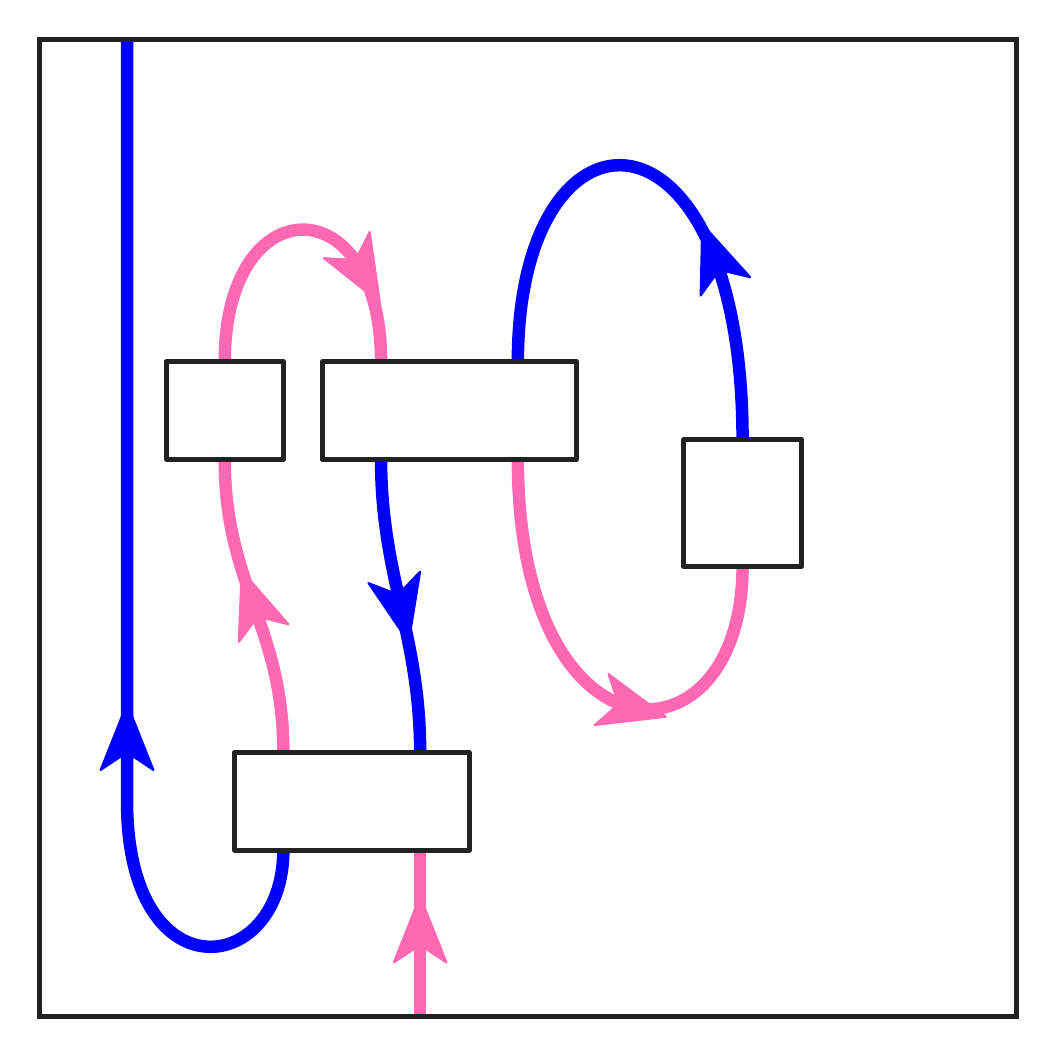}
    \put(9.1,26.2){\makebox(0,0){\footnotesize{$\theta$}}}
    \put(18.2,26.2){\makebox(0,0){\footnotesize{$\beta^{-1}$}}}
    \put(14.3,10.3){\makebox(0,0){\footnotesize{$\beta^{-1}$}}}
    \put(30.2,22.4){\makebox(0,0){\footnotesize{$F$}}}
    \put(19.5,4.0){\makebox(0,0){\footnotesize{$X$}}}
    \put(12.7,19.8){\makebox(0,0){\footnotesize{$X$}}}
    \put(12.3,36.0){\makebox(0,0){\footnotesize{$X$}}}
    \put(32.5,16.3){\makebox(0,0){\footnotesize{$X$}}}
    \put(7.5,17.4){\makebox(0,0){\footnotesize{$P$}}}
    \put(18.9,18.2){\makebox(0,0){\footnotesize{$P$}}}
    \put(32.5,31.8){\makebox(0,0){\footnotesize{$P$}}}
  \end{overpic} }}$
  $$\overset{\Pen_2}{\longmapsto} (\id_P\otimes \overrightarrow{\ev}_X\otimes \overleftarrow{\ev}_P)\circ(\id_P\otimes \theta_X\otimes \beta^{-1}_{X,P}\otimes F)\circ(\id_{P\otimes X\otimes P^*}\otimes \overrightarrow{\coev}_X)\circ (\id_P\otimes \beta_{X,P}^{-1})\circ(\overleftarrow{\coev}_P\otimes \id_X).$$
   \caption{An example computation of the Reshetikhin-Turaev functor $\Pen_3$ of a ribbon graph embedded in a cube.}
    \label{fig:rt-example}
\end{figure}

The chromatic spherical invariant and admissible skein module of a space $M$ are defined in terms of different types of $\mathcal{C}$-colored graphs. The chromatic spherical invariant of $M$ relies on a diagram which decomposes in two pieces: a disjoint union of unoriented circles representing a Heegaard splitting and an $\mathcal{I}$-colored graph on the boundary of $M$.  Meanwhile, an admissible skein module allows colorings from $\cC$ so long as at least one edge/ribbon is colored by an object in $\cI$ in each connected component of $M$. More explicitly, if $M$ is a 2- (or 3-)manifold and $\cC$ is a pivotal (ribbon) category, $\SM(M)$ is defined in terms of $\cC$-colored (ribbon) graphs with at least one edge (ribbon) colored by an object in $\cI$ in each connected component of $M$. Such graphs are called {\it $\mathcal I$-admissible (ribbon) graphs.} When $\cC$ is ribbon and $M$ is a 2-manifold, $\SM(M)\cong \SM(M\times [0,1])$. \\

Let $\cC$ be a pivotal tensor category. By right closure of a $\cC$-colored graph $\Gamma$ in a square $[0,1]^2$ with boundary points labeled by $X_1,\dots, X_n\in\cC$, we mean the $\cC$-colored graph defined by looping the ends of $\Gamma$ as in Figure \ref{fig:closure}. The closure of a $\cC$-colored ribbon graph $\Gamma$ in a cube $[0,1]^3$ is defined similarly. In the semisimple setting, we generally identify the closure with the quantum trace, but again the quantum trace is degenerate in the non-semisimple setting, so we generally do not assign the quantum trace to the closure of $\Gamma$ in this article. The chromatic spherical invariant allows us to assign the modified trace to such a closure, while admissible skein modules do not make any identifications on such a ribbon graph.

\begin{figure}[ht]
  \centering
  \begin{overpic}[abs, unit=1mm, scale = .15]{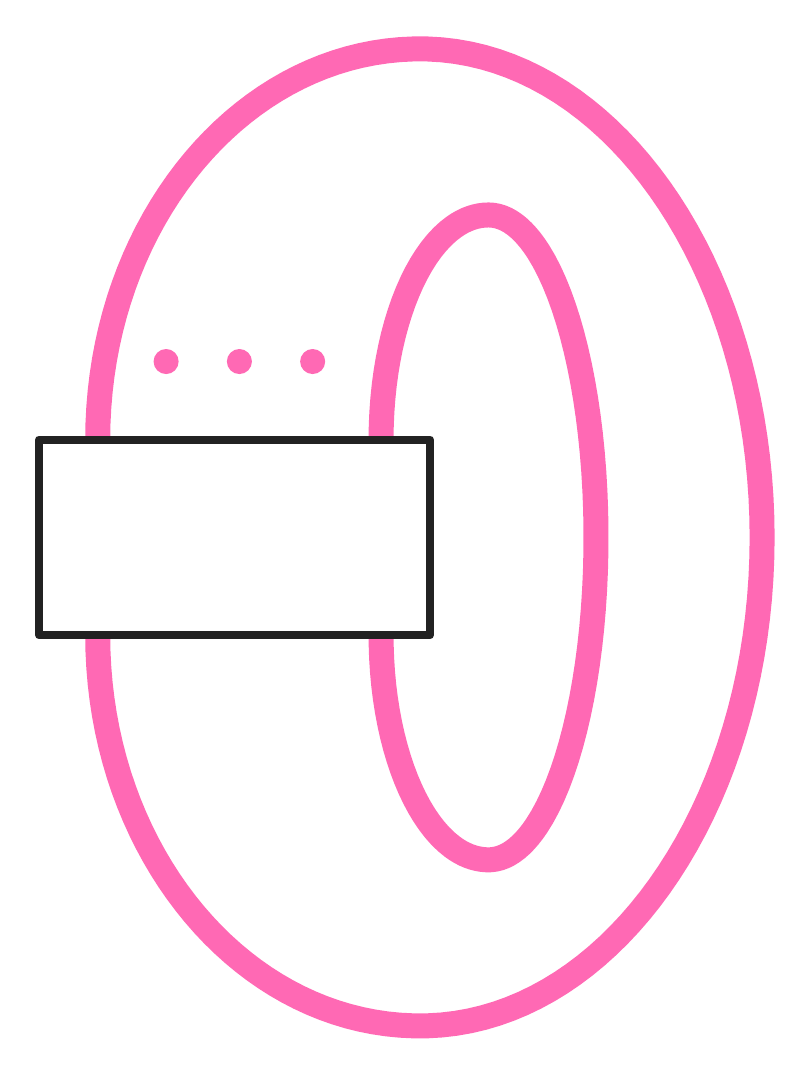}
  \put(5.9,13.7){\makebox(0,0){$\Gamma$}}
  \put(0.2,18.2){\makebox(0,0){\scriptsize{$X_1$}}}
  \put(12.0,18.2){\makebox(0,0){\scriptsize{$X_n$}}}
\end{overpic}
  \caption{The right closure of a $\cC$-colored ribbon graph $\Gamma$.}
  \label{fig:closure}
\end{figure}

\section{The chromatic spherical invariant $\Ku_\cC(M)$}\label{section:kuperbergtype}
To have consistent terminology and notation throughout this article, we use slightly different conventions from \cite{tvunimodular}.

\subsection{Definitions}\label{subsec:chromatic-def}

Let $\cI$ be a tensor ideal of a pivotal tensor category $\cC$ with non-degenerate m-trace $t=\{t_P\}_{P\in\cI}$. 

A \textit{multi-handlebody} is a disjoint union of 3-dimensional handlebodies $H_g$. Let $\Gamma$ be an $\cI$-colored ribbon graph in a multi-handlebody $M$. Let $D^2\subset M$ be a disk with $\partial D^2\subset \partial M$ which intersects $\Gamma$ transversely and nontrivially. We define the local cut operation seen in Figure \ref{fig:cut-operation} whose resulting multi-handlebody and ribbon graph are denoted $\mathrm{cut}_{D^2}(M)$ and $\mathrm{cut}_{D^2}(\Gamma)$ respectively.

\begin{figure}[ht]
  \centering
  $\vcenter{\hbox{\begin{overpic}[abs, unit=1mm, scale = .25]{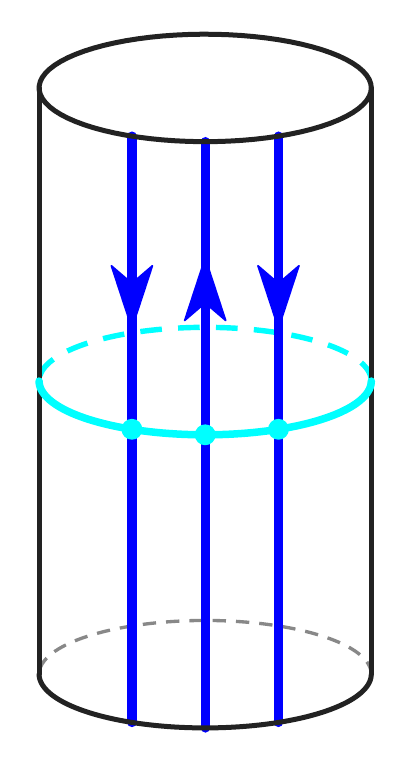}
  \end{overpic} }}
  \;\mapsto\; \sum\limits_i \;
  \vcenter{\hbox{\begin{overpic}[abs, unit=1mm, scale = .25]{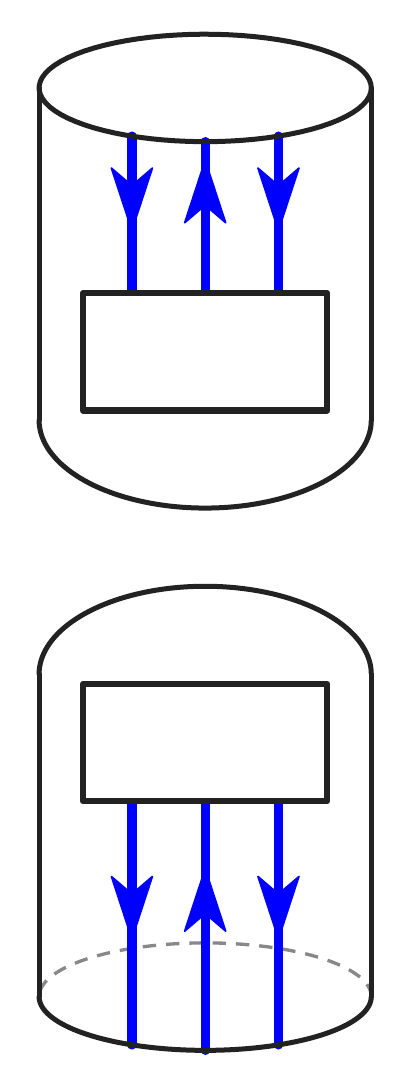}
    \put(8.7,31.0){\makebox(0,0){$x_i$}}
    \put(8.7,14.5){\makebox(0,0){$y_i$}}
  \end{overpic} }}$
  \caption{The local cut operation on $\cI$-colored graphs embedded on the boundary of a multi-handlebody.}
\label{fig:cut-operation}
\end{figure}

\begin{lemma}[\cite{tvunimodular}]
    Let $\cI$ be a tensor ideal of a pivotal tensor category $\cC$ with non-degenerate m-trace $t=\{t_P\}_{P\in\cI}$. Then, there is a unique map $F'$ which assigns a scalar to an $\cI$-colored graph $\Gamma$ in multi-handlebodies $M=\bigsqcup H_g$ with the following properties.
    \begin{enumerate}
        \item $F'(H, \Gamma)$ is invariant under orientation-preserving diffeomorphisms.
        \item In a ball, $F'$ agrees with the modified trace; if $\Gamma$ is an $\cI$-colored graph in a ball $D^3$ and can be written as the closure of an $\cI$-colored graph $\Gamma'$ with nontrivial boundary, then 
        $$F'(D^3, \Gamma) = t(\Pen_2(\Gamma')),$$
        where $\Pen_2$ is the Penrose functor discussed in Subsection \ref{subsection:icolouredgraphs}.
        \item $F'$ turns disjoint unions into products; if $\Gamma_i$ is an $\cI$-colored ribbon graph in $M_i$ for $i=1,2$, then
        $$F'(M_1\sqcup M_2, \Gamma_1\sqcup \Gamma_2) = F'(M_1,\Gamma_1)F'(M_2,\Gamma_2).$$
        \item $F'$ is preserved under local cuts as in Figure \ref{fig:cut-operation}; given a disk $D^2\subset M$ with $\partial D^2\subset \partial M$ which intersects $\Gamma$ transversely and nontrivially
        $$F'(M, \Gamma) = F'(\mathrm{cut}_{D^2}(M),\mathrm{cut}_{D^2}(\Gamma)).$$
    \end{enumerate}
\end{lemma}

Given an object $P\in\cI$, choose bases $x_i^P\in \Hom(\one\to P)$ and $y_i^P\in\Hom(P\to\one)$ which are dual under the pairing induced by $t_P$. Let $\Lambda_P = \sum x_i^{P\otimes P^*}\otimes y_i^{P\otimes P^*}$.

An object $G\in\cI$ is a \textit{generator} if, for any $P\in\cI$, there exist morphisms $f_j:G\to P$ and $g_j:P\to G$, such that 
$$\id_P = \sum_j f_j\circ g_j.$$
Let $G$ be a generator of $\cI$. A \textit{chromatic morphism} is a morphism $\tilde d:G\otimes G\to G\otimes G$ which satisfies the relation described in Figure \ref{fig:chromaticmorph}.

\begin{figure}[ht]
  \centering
  $\vcenter{\hbox{\begin{overpic}[abs, unit=1mm, scale = .3]{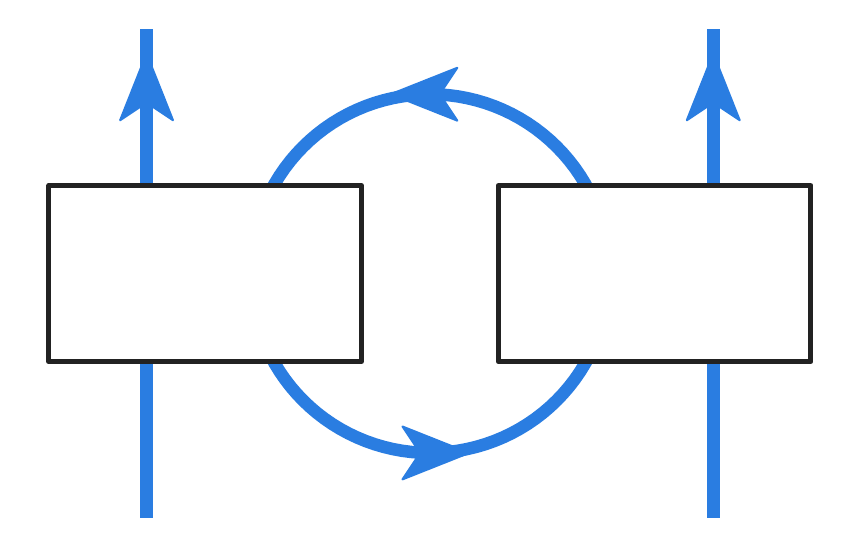}
    \put(10.4,13.9){\makebox(0,0){$\Lambda_G$}}
    \put(33.3,13.9){\makebox(0,0){$\widetilde{d}$}}
    \put(5.2,22.0){\makebox(0,0){\footnotesize{$G$}}}
    \put(5.2,5.0){\makebox(0,0){\footnotesize{$G$}}}
    \put(38.5,22.0){\makebox(0,0){\footnotesize{$G$}}}
    \put(38.5,5.0){\makebox(0,0){\footnotesize{$G$}}}
    \put(21.8,20.0){\makebox(0,0){\footnotesize{$G$}}}
    \put(21.8,7.8){\makebox(0,0){\footnotesize{$G$}}}
  \end{overpic} }}
  \; = \;
  \vcenter{\hbox{\begin{overpic}[abs, unit=1mm, scale = .3]{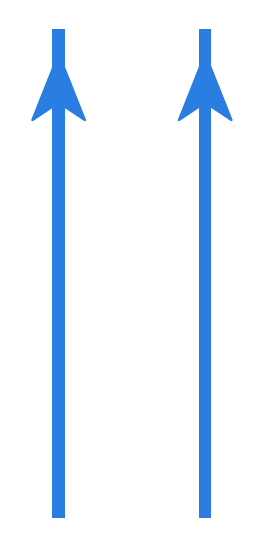}
    \put(5.3,13.9){\makebox(0,0){\footnotesize{$G$}}}
    \put(12.8,13.9){\makebox(0,0){\footnotesize{$G$}}}
  \end{overpic} }}$
  \caption{Defining axiom of the chromatic morphism.}
  \label{fig:chromaticmorph}
\end{figure}

We extend $\tilde d$ to a morphism $\tilde d_P:G\otimes P\to G\otimes P$ on arbitrary objects $P\in\cI$ by 
\[
\widetilde{d}_P \;=\; \sum_j \;
\vcenter{\hbox{\begin{overpic}[abs, unit=1mm, scale = .25]{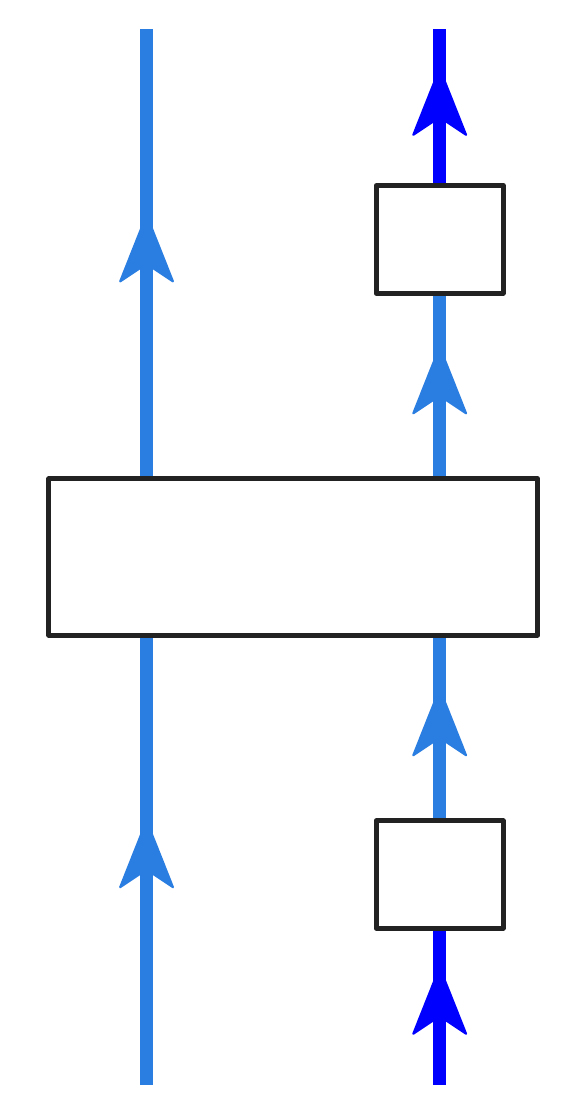}
  \put(12.4,23.6){\makebox(0,0){$\widetilde{d}$}}
  \put(18.6,37.0){\makebox(0,0){\footnotesize{$f_j$}}}
  \put(18.6,10.1){\makebox(0,0){\footnotesize{$g_j$}}}
  \put(8.6,36.4){\makebox(0,0){\footnotesize{$G$}}}
  \put(8.6,10.8){\makebox(0,0){\footnotesize{$G$}}}
  \put(21.0,30.8){\makebox(0,0){\footnotesize{$G$}}}
  \put(21.0,16.3){\makebox(0,0){\footnotesize{$G$}}}
  \put(21.0,42.6){\makebox(0,0){\footnotesize{$P$}}}
  \put(21.0,4.6){\makebox(0,0){\footnotesize{$P$}}}
\end{overpic} }} \;.
\]

\begin{prop}\cite{beliakova2020modifiedtracesymmetrisedintegral} \cite{costantino_skein_2023}\label{prop:hopf-algebra-morphisms}
    Let $H$ be a unimodular, pivotal Hopf algebra. Given $M\in\Rep(H)$, let $\{m_i\}$ be a basis and $\{m^i\}$ be the corresponding dual basis for $M^*$. Then, $H$ is a generator of $\Proj(\Rep(H))$ and we have the following identities:
    \begin{enumerate}[(\alph*)]
        \item the left evaluation $\overleftarrow{\ev}_M:M^*\otimes M\to \bk_\epsilon$ is generated by $f\otimes m\mapsto f(m)$;
        \item the left coevaluation $\overleftarrow{\coev}_M:\bk_\epsilon\to M\otimes M^*$ is given by $1\mapsto \sum m_i\otimes m^i$;
        \item the right evaluation $\overrightarrow{\ev}_M:M\otimes M^*\to \bk_\epsilon$ is generated by $m\otimes f\mapsto f(g\cdot m)$;
        \item the right coevaluation $\overrightarrow{\coev}_M:\bk_\epsilon\to M^*\otimes M$ is given by $1\mapsto \sum m^i\otimes g^{-1}\cdot m_i$;
        \item the right modified trace $t_H:\End_H(H)\to\bk$ is given by $f\mapsto \lambda(gf(1))$;
        \item the chromatic map $\tilde d:H\otimes H\to H\otimes H$ is generated by $x\otimes y\mapsto\lambda(S(y^{(1)})gx) y^{(2)} \otimes y^{(3)}$,
    \end{enumerate}
    where $g$ is the pivot and $\lambda$ is a right cointegral.
\end{prop}
Knowing the right modified trace on a generator $G$ is sufficient to determine its value on arbitrary objects. In particular, for an object $P\in\cI$ we may extend $t_G$ using the $f_j, g_j$ morphisms guaranteed by the definition of a generator: 

\[
t_P \;=\; \sum_j \;
\vcenter{\hbox{\begin{overpic}[abs, unit=1mm, scale = .25]{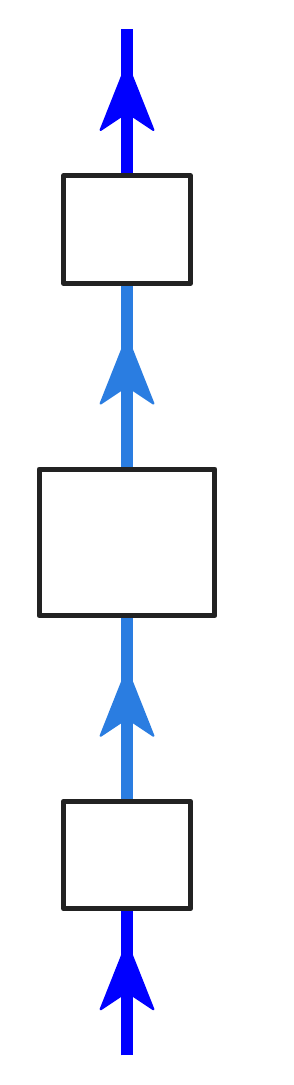}
  \put(5.4,23.0){\makebox(0,0){\footnotesize{$t_G$}}}
  \put(5.4,36.2){\makebox(0,0){\footnotesize{$f_j$}}}
  \put(5.4,9.7){\makebox(0,0){\footnotesize{$g_j$}}}
  \put(7.8,30.0){\makebox(0,0){\footnotesize{$G$}}}
  \put(7.8,15.9){\makebox(0,0){\footnotesize{$G$}}}
  \put(7.8,41.6){\makebox(0,0){\footnotesize{$P$}}}
  \put(7.8,4.3){\makebox(0,0){\footnotesize{$P$}}}
\end{overpic} }} \;.
\]

A pivotal Hopf algebra $H$ is \textit{unibalanced} if the $t$ induced by $t_H$ in Proposition \ref{prop:hopf-algebra-morphisms} is also a left modified trace. Equivalently, a finite-dimensional pivotal Hopf algebra is unibalanced if there are nonzero right and left cointegrals $\lambda,\lambda':H\to\bk$ whose symmetrizations agree: $\lambda(gx) = \lambda'(g^{-1}x)$ for all $x\in H$. In particular, if $H$ is counimodular, this condition becomes $\lambda(g^2x) = \lambda(x)$ for all $x\in H$.

An \textit{admissible bichrome diagram} is a graph on the boundary of a multi-handlebody which is split as a disjoint union of two subgraphs: blue and red with the following properties:
\begin{itemize}
    \item the red subgraph is a disjoint union of unoriented simple closed curves;
    \item the blue subgraph is an $\cI$-colored graph on the multi-handlebody.
\end{itemize}

\begin{lemma}[\cite{tvunimodular}]
    Let $\cC$ be a pivotal tensor category with a non-degenerate m-trace on an ideal $\cI$ and a chromatic morphism $\tilde d$ on a generator $G\in \cI$. Then, $F'$ can be uniquely extended to admissible bichrome handlebody graphs by the following skein relation:
 \[
\vcenter{\hbox{\begin{overpic}[abs, unit=1mm, scale = .25]{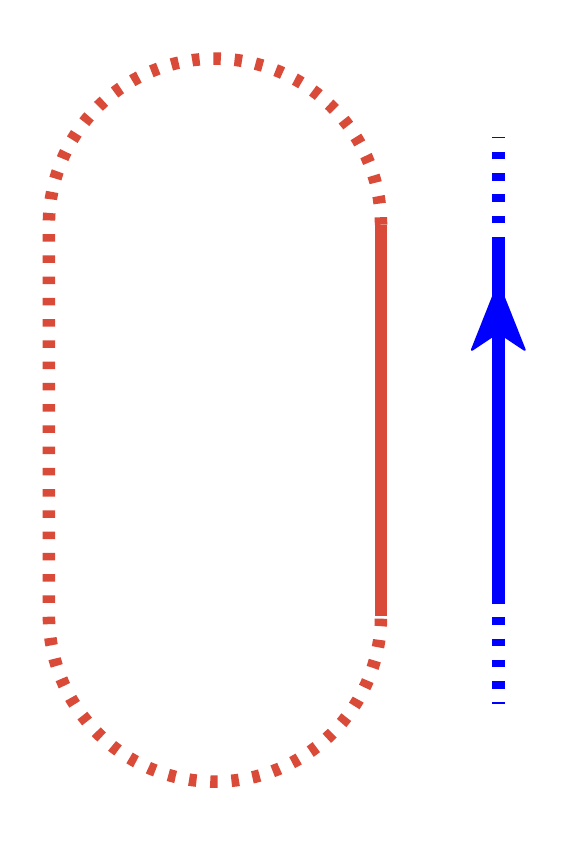}
  \put(23.5,21.9){\makebox(0,0){\footnotesize{$P$}}}
\end{overpic} }}
\;\longrightarrow\;
\vcenter{\hbox{\begin{overpic}[abs, unit=1mm, scale = .25]{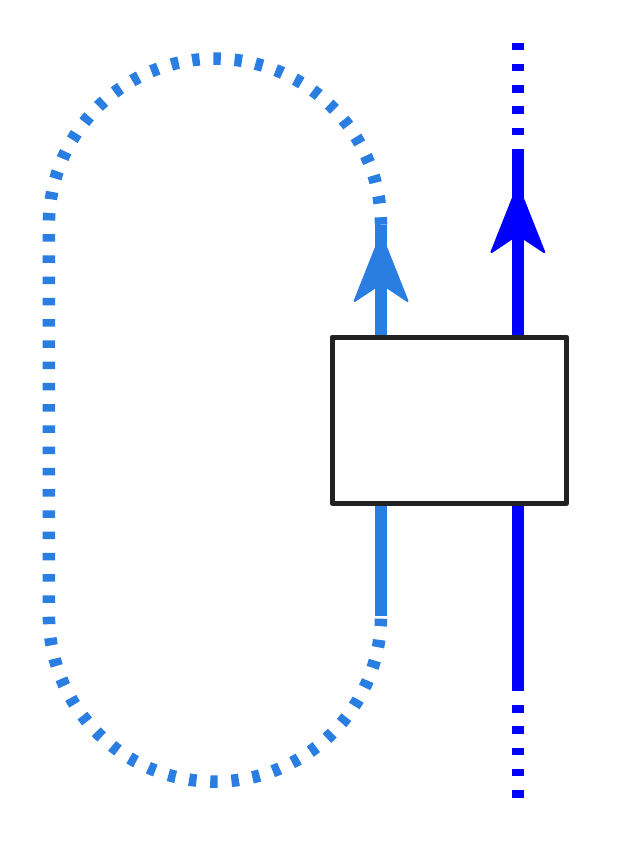}
  \put(19.0,17.8){\makebox(0,0){$\widetilde{d}_P$}}
  \put(13.4,24.0){\makebox(0,0){\footnotesize{$G$}}}
  \put(19.4,26.1){\makebox(0,0){\footnotesize{$P$}}}
\end{overpic} }}
\]
    Again, $F'$ only depends on the orientation-preserving diffeomorphism class.
\end{lemma}

Every closed $3$-manifold $M$ admits a (non-unique) Heegaard splitting
$M \cong M_{\alpha} \cup_{\Sigma} M_{\beta}$, where $M_{\alpha}$ and $M_{\beta}$
are diffeomorphic multi-handlebodies of the same genus glued along their common boundary Heegaard surface
$\Sigma = \partial M_{\alpha} = \partial M_{\beta}$. The gluing is encoded by a
Heegaard diagram $(\Sigma, \{\alpha_i\}, \{\beta_i\})$, in which
$\{\alpha_i\}$ and $\{\beta_i\}$ are minimal reducing sets of unoriented
simple closed curves on $\Sigma$ for $M_{\alpha}$ and $M_{\beta}$, respectively.

Given a tensor ideal $\cI$ in a pivotal tensor category $\cC$ with a non-degenerate modified trace $t$ and generator $G$ with chromatic morphism $\tilde d$, we may define an admissible bichrome diagram $\Gamma(\alpha, \beta)$ from a Heegaard splitting as follows: color the $\{\beta_i\}$ red inside $M_\alpha$, choose a morphism $\eta:G\to G$ such that $t_G(\eta) = 1$, and add a blue-colored right closure of $\eta$ in any empty space on the boundary of each component of $M_\alpha$. This construction allows us to define a chromatic spherical invariant of $M$ as exhibited in the following theorem.

\begin{theorem}[\cite{tvunimodular}]
    Let $\cI$ be a tensor ideal in a pivotal tensor category $\cC$ with a non-degenerate modified trace $t$ and generator $G$ with chromatic morphism $\tilde d$. Let $M$ be a closed 3-manifold with Heegaard splitting $M_{\alpha}\cup_\Sigma M_{\beta}$. Define 
    $$\Ku_\cC(M) = F'(M_{\alpha}, \Gamma(\alpha, \beta)).$$
    Then, $\Ku_\cC(M)$ is independent of the choice of Heegaard splitting and placement of $\eta$.
\end{theorem}
Note that the chromatic spherical invariant $\Ku_\cC$ does depend on $\cI$, but we use the notation of \cite{tvunimodular}.
 In the case where both the chromatic spherical invariant and the Kuperberg invariant are defined, they agree \cite{tvunimodular}. That is, if $H$ is an involutive ($S^2 = \id_H$), unimodular, unibalanced Hopf algebra, then the chromatic spherical invariant for the ideal $\Proj(\Rep(H))$ in $\Rep(H)$ agrees with the Kuperberg invariant on any closed 3-manifold $M$: 
    $$\Ku_{\Rep(H)}(M) = \mathrm{Ku}_H(M).$$
    
\subsection{Computation for lens spaces} \label{subsec:chromatic-lens}
In this section, we compute $\Ku_\cC$ when $\cC=\Rep(H)$ for some unimodular, unibalanced Hopf algebra $H$ and $\cI=\Proj(\Rep(H))$ is the projective ideal. Note that we do not assume that $H$ is involutive, so the Kuperberg invariant is not defined. The lens space $L(p,q)$ has the following genus-one Heegaard splitting: both
$M_{\alpha}$ and $M_{\beta}$ are solid tori, the set $\{\alpha_i\} = \{\alpha_1\}$
consists of a single meridian, and $\{\beta_i\} = \{\beta_1\}$ consists of a
single $(p,q)$-curve on the boundary torus $\Sigma$. The admissible bichrome diagram associated to $L(p,q)$ is shown in Figure \ref{fig:bichrome}.

\begin{figure}[ht]
  \centering
  \begin{overpic}[abs, unit=1mm, scale = .25]{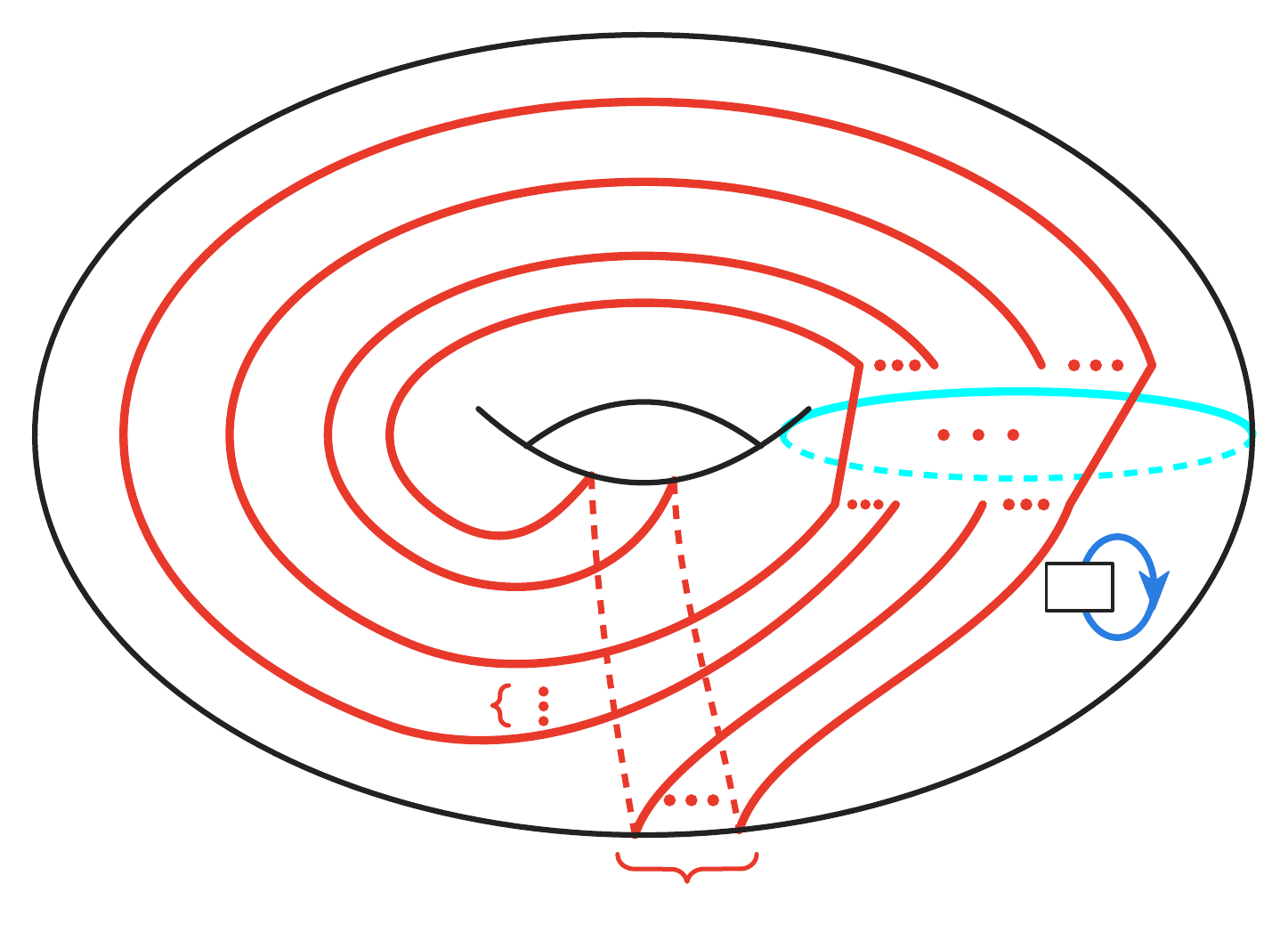}
    \put(19.8,12.2){\makebox(0,0){\scriptsize{$p-q$}}}
    \put(32.7,1.7){\makebox(0,0){\footnotesize{$q$}}}
    \put(51.3,17.2){\makebox(0,0){\scriptsize{$\eta$}}}
    \put(55.5,20.2){\makebox(0,0){\footnotesize{$G$}}}
  \end{overpic}
  \caption{The bichrome diagram for the lens space $L(p,q)$.} 
  \label{fig:bichrome}
\end{figure}

\begin{theorem}\label{thm:chromatic-lens}
    Let $p\geq q\geq 1$ be coprime positive integers, and $H$ be a finite-dimensional unibalanced, unimodular Hopf algebra. Given $n\in\{1,\dots, q-1,q+1,\dots, p-1\}$, let $\bar n = n + p\bZ$, and define the notation $\Lambda^{(\bar n)} = \Lambda^{(n)}$ and $i_{\bar n} = 0$ if $n < q$ and $i_{\bar n} = 1$ if $n > q$. The chromatic spherical invariant of $L(p, q)$ in $\Rep(H)$ is given by
    $$\Ku_{\Rep(H)}(L(p,q)) = \lambda(\Lambda^{(q)}g^{i_{\overline {2q}}}\Lambda^{(\overline{2q})}g^{i_{\overline {3q}}}\Lambda^{(\overline{3q})}g^{i_{\overline {4q}}}
\dots g^{i_{\overline{(p-1)q}}}\Lambda^{(\overline{(p-1)q})}g\Lambda^{(p)}),$$
    where $g$ is the pivot and $(\Lambda, \lambda)$ are a normalized ($\lambda(\Lambda)=1$) integral/cointegral pair.
\end{theorem}

While this expression may seem complicated as written, actually writing out this expression for specific $p, q$ is not difficult. For example, the chromatic spherical invariant of the space $L(p, 1)$ simplifies to
$$\Ku_{\Rep(H)}(L(p,1)) = \lambda(\Lambda^{(1)}g\Lambda^{(2)}g\Lambda^{(3)}\dots
\dots g\Lambda^{(p)}) = \lambda(g^{-1}\Lambda^{(1)}\Lambda^{(2)}\Lambda^{(3)}\dots
\dots \Lambda^{(p)}).$$
The chromatic spherical invariant of the space $L(p, p-1)$ simplifies to
\begin{align*}
    \Ku_{\Rep(H)}(L(p,p-1)) &= \lambda(\Lambda^{(p-1)}\Lambda^{(p-2)}\Lambda^{(p-3)}\dots
\dots \Lambda^{(1)}g\Lambda^{(p)})\\
    &=\lambda(g\Lambda^{(p-1)}\Lambda^{(p-2)}\Lambda^{(p-3)}\dots
\dots \Lambda^{(1)}S^2(\Lambda^{(p)})).
\end{align*}
As a more complex example, the chromatic spherical invariant of the space $L(7, 4)$ is given by 
\begin{align*}
    \Ku_{\Rep(H)}(L(7,4)) &= \lambda(\Lambda^{(4)}\Lambda^{(1)}g\Lambda^{(5)}\Lambda^{(2)}g\Lambda^{(6)}\Lambda^{(3)}g\Lambda^{(7)}) \\
    &= \lambda(g^3\Lambda^{(4)}\Lambda^{(1)}S^2(\Lambda^{(5)})S^2(\Lambda^{(2)})S^4(\Lambda^{(6)})S^4(\Lambda^{(3)})S^{6}(\Lambda^{(7)})).
\end{align*}
\begin{proof}
    First, we manipulate the bichrome diagram to express the chromatic spherical invariant in terms of a generator $G$ of an ideal $\cI$ of a unimodular pivotal finite tensor category $\cC$. Let $\{x_i\}\subset\Hom(\one\to G^{\otimes p})$ and $\{y_i\}\subset\Hom(G^{\otimes p}\to\one)$ be a pair of dual bases under the inner product defined by modified trace map. Let $\tilde d:G\otimes G\to G\otimes G$ denote a chromatic morphism and $\eta:G\to G$ any morphism with modified trace $t_G(\eta) = 1$.  We apply the properties of $F'$ to express the $\Ku_\cC$ in terms of a modified trace.

\[
  \centering
  \newcommand{\lb}[1]{\makebox(0,0){\scalebox{0.72}{\tiny{$#1$}}}}
  \newcommand{\lbs}[1]{\makebox(0,0){\scalebox{0.55}{\tiny{$#1$}}}}
  \begin{aligned}
    \Ku_\cC(L(p,q))&= F'\left( \vcenter{\hbox{\begin{overpic}[abs, unit=1mm, scale = .17]{lens_bichrome_diagram.pdf}
      \put(13.8,8.3){\lb{p-q}}
      \put(22.2,1.0){\lb{q}}
      \put(34.9,11.7){\lb{\eta}}
      \put(38.0,13.9){\lb{G}}
    \end{overpic}}} \right)
    = F'\left( \sum_i \vcenter{\hbox{\begin{overpic}[abs, unit=1mm, scale = .17]{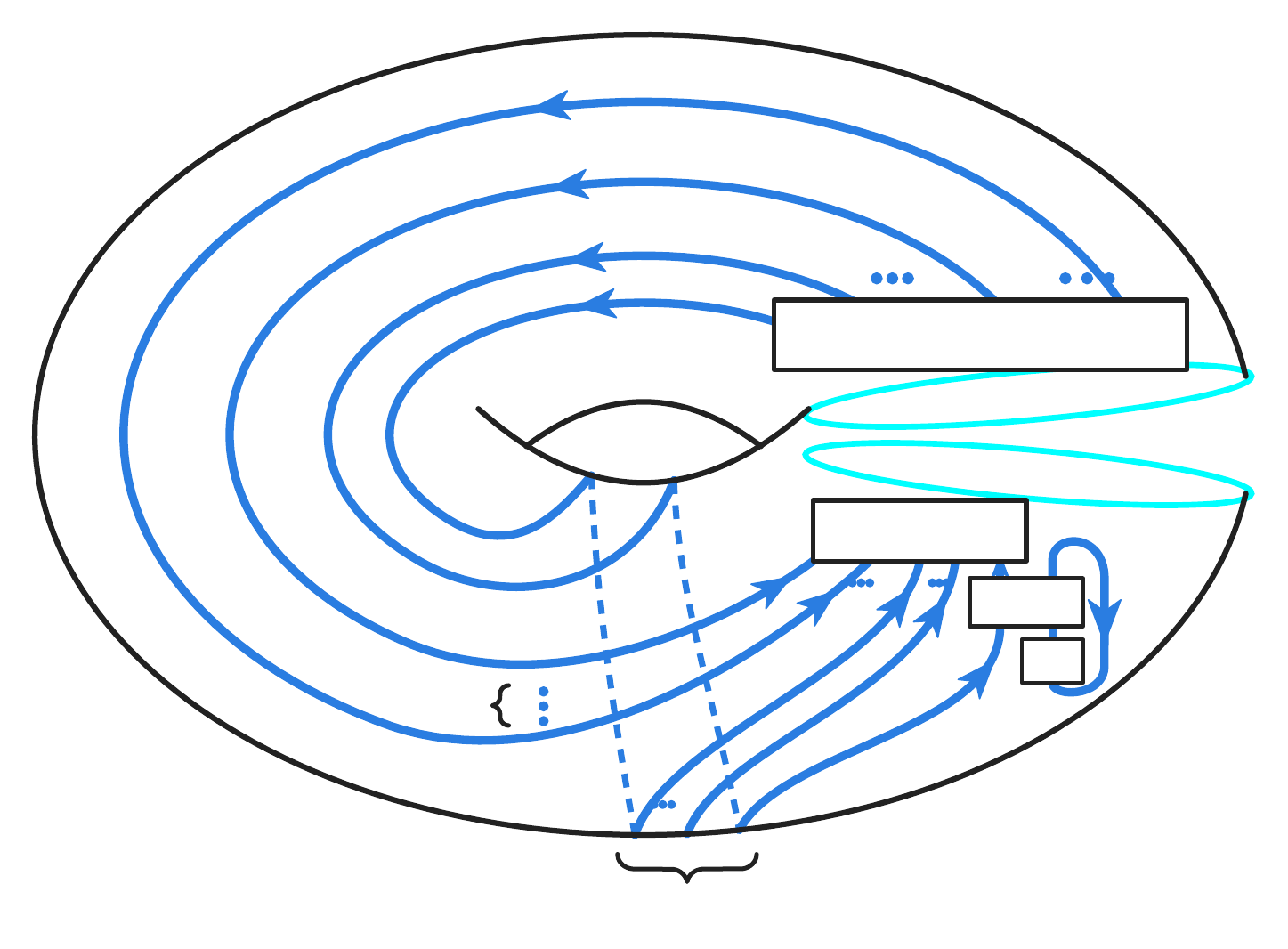}
       \put(31.7,19.9){\lb{x_i}}
      \put(29.7,13.6){\lb{y_i}}
      \put(33.2,11.2){\lbs{\widetilde{d}}}
      \put(34.0,9.3){\lbs{\eta}}
      \put(13.6,8.3){\lb{p-q}}
      \put(22.2,1.0){\lb{q}}
      \put(37.1,10.9){\lb{G}}
      \put(11.5,27.0){\lb{G}}
      \put(30.1,6.3){\lb{G}}
    \end{overpic}}} \right) \\[1ex]
    & = F'\left( \sum_i \vcenter{\hbox{\begin{overpic}[abs, unit=1mm, scale = .14]{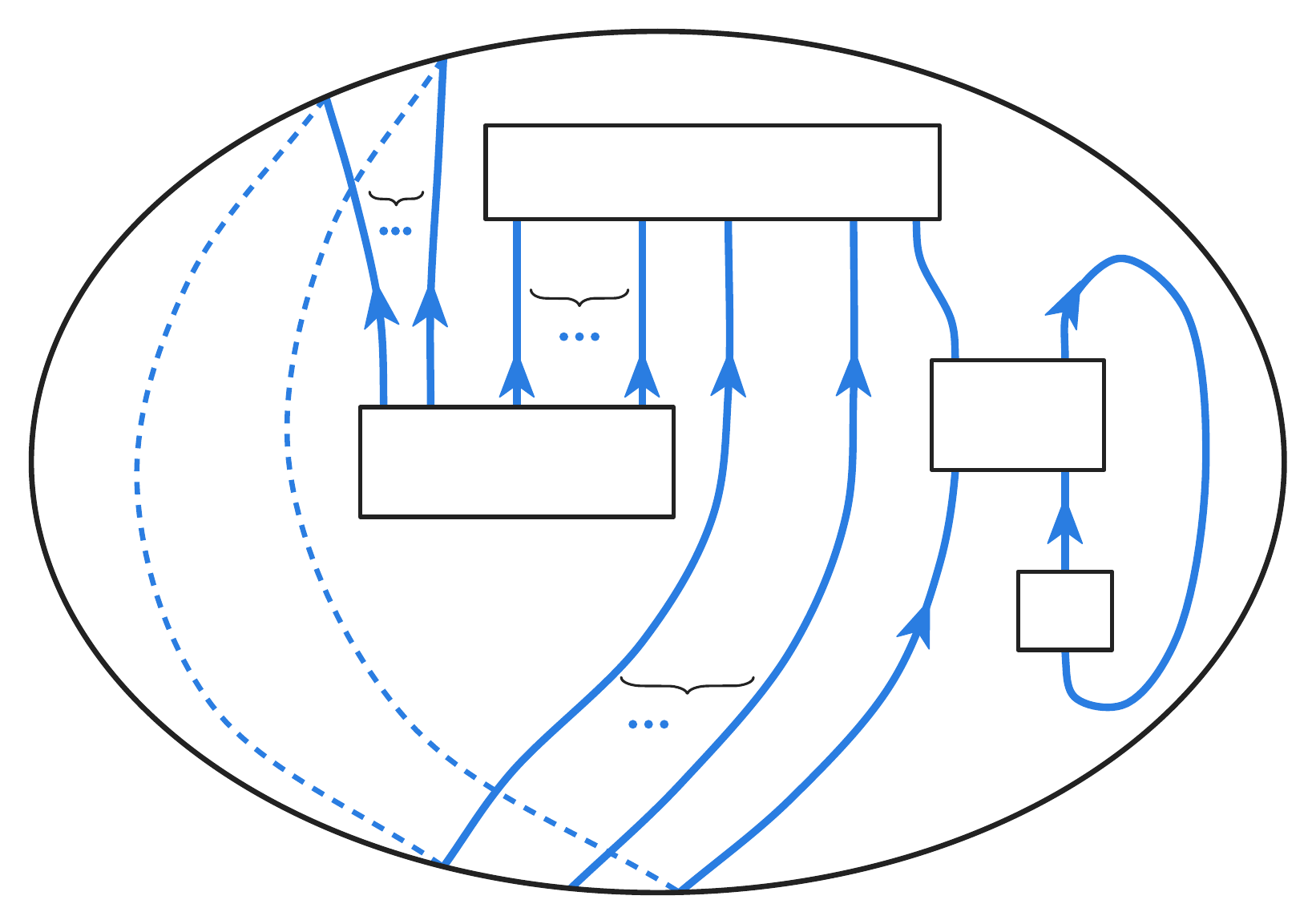}
      \put(15.3,13.6){\lb{x_i}}
      \put(21.1,22.3){\lb{y_i}}
      \put(30.1,15.0){\lb{\widetilde{d}}}
      \put(31.5,9.2){\lb{\eta}}
      \put(17.1,19.9){\lbs{p-q}}
      \put(11.7,22.8){\lb{q}}
      \put(21.2,8.5){\lbs{q-1}}
      \put(29.4,19.2){\lb{G}}
      \put(32.8,11.9){\lb{G}}
      \put(37.0,13.6){\lb{G}}
      \put(28.6,8.1){\lb{G}}
    \end{overpic}}} \right)
    = F'\left( \sum_i \vcenter{\hbox{\begin{overpic}[abs, unit=1mm, scale = .14]{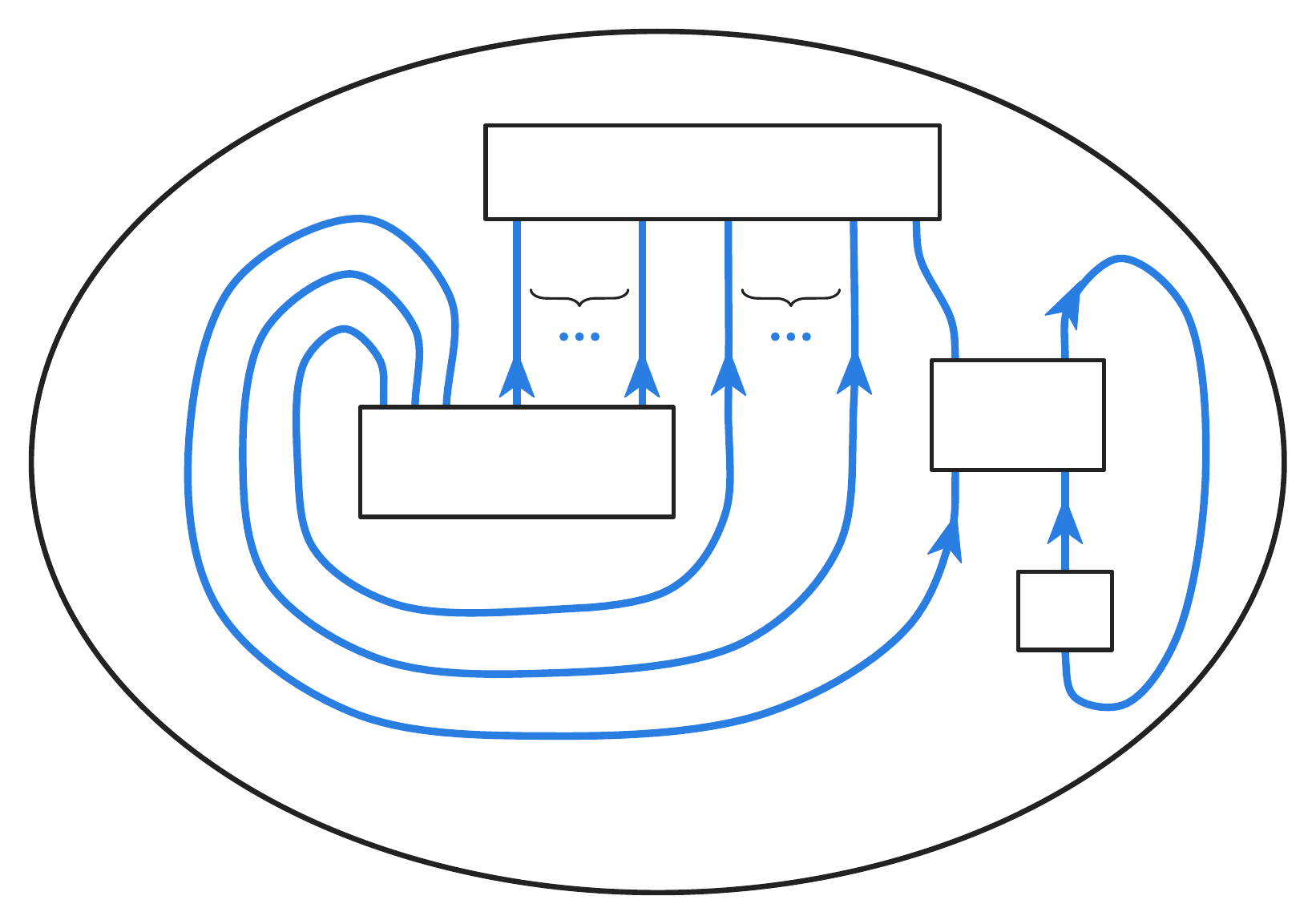}
      \put(15.3,13.6){\lb{x_i}}
      \put(21.1,22.3){\lb{y_i}}
      \put(30.1,15.0){\lb{\widetilde{d}}}
      \put(31.5,9.2){\lb{\eta}}
      \put(17.1,19.9){\lbs{p-q}}
      \put(23.4,19.9){\lbs{q-1}}
      \put(29.4,19.2){\lb{G}}
      \put(32.8,11.9){\lb{G}}
      \put(37.0,13.6){\lb{G}}
      \put(26.5,11.3){\lb{G}}
    \end{overpic}}} \right) \\[1ex]
    & = t_P\left( \sum_i \vcenter{\hbox{\begin{overpic}[abs, unit=1mm, scale = .2]{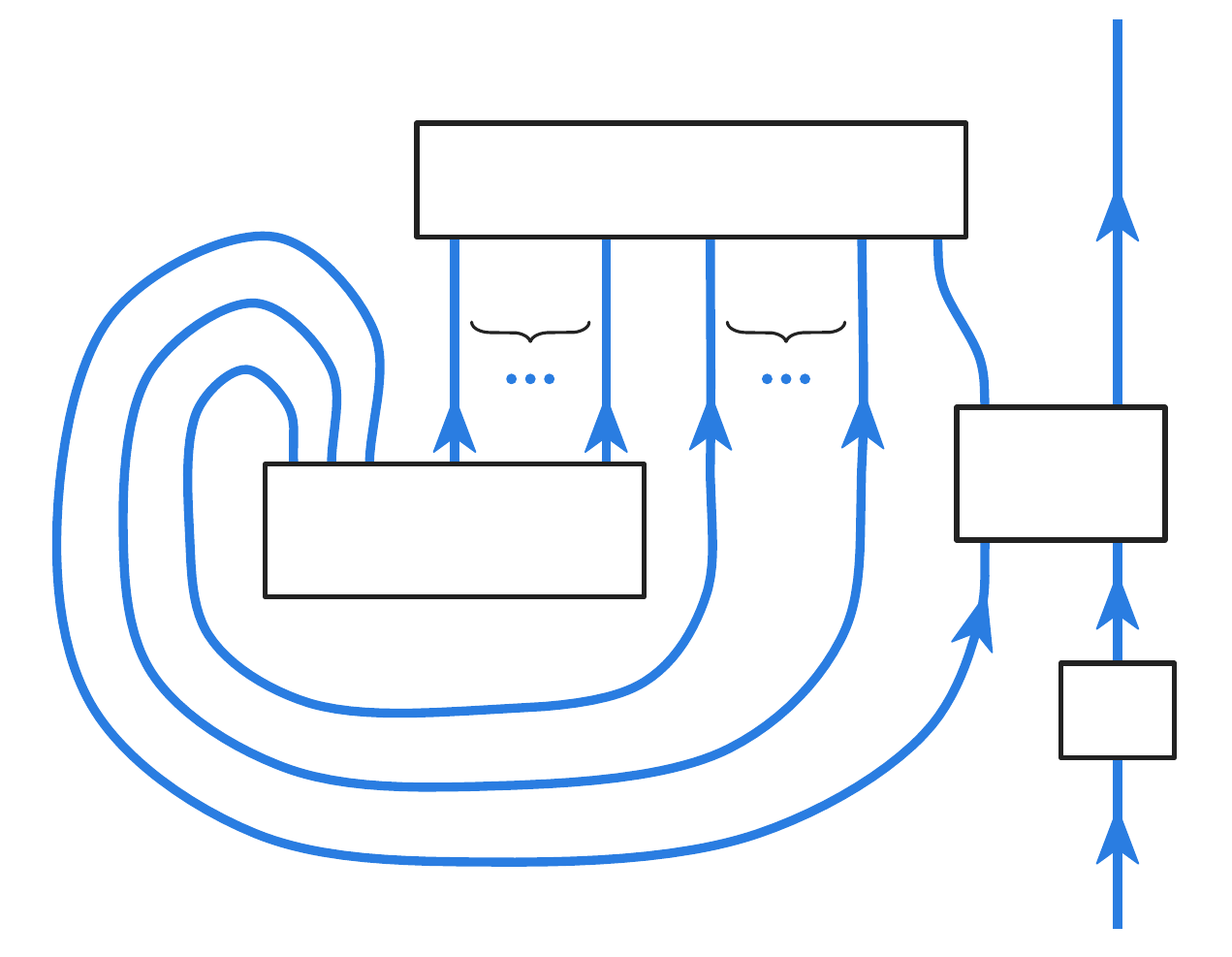}
      \put(15.9,14.9){\lb{x_i}}
      \put(24.2,27.1){\lb{y_i}}
      \put(37.0,16.9){\lb{\widetilde{d}}}
      \put(39.0,8.6){\lb{\eta}}
      \put(18.6,23.7){\lb{p-q}}
      \put(27.4,23.7){\lb{q-1}}
      \put(41.0,29.1){\lb{G}}
      \put(36.0,22.6){\lb{G}}
      \put(41.0,12.4){\lb{G}}
      \put(41.0,4.0){\lb{G}}
      \put(31.8,11.6){\lb{G}}
    \end{overpic}}} \right)
  \end{aligned}
  \label{fig:computation} \]

    Expressed symbolically, the chromatic spherical invariant of $L(p, q)$ is given, in any unimodular pivotal finite tensor category $\cC$ with non-degenerate modified trace $t$ on a tensor ideal $\cI$, by
    \begin{equation}\label{kuperberg-start}
        \Ku_{\cC}(L(p,q)) = \displaystyle\sum_i t_G[(\overleftarrow{\ev}_{G^{\otimes q}}\otimes y_i\otimes \id_G)\circ(\id_{(G^*)^{\otimes q}}\otimes x_i\otimes\id_{G^{\otimes q-1}}\otimes \tilde d)\circ(\overrightarrow{\coev}_{G^{\otimes q}}\otimes\eta)].
    \end{equation}
    We now specialize to $\cC=\Rep(H)$ and take $G=H$ as our projective generator. The modified trace for $H$ is given by $t_H(f:H\to H) = \lambda(gf(1))$. Thus, to evaluate \eqref{kuperberg-start}, we only need to identify the image of $1_H$ under the sequence of compositions given. 
    
    Let $\bk_\epsilon$ denote the trivial representation of $H$, whose underlying vector space is $\bk$ and action is generated by $h\cdot 1_\bk = \epsilon(h)$. Let $B$ be an arbitrary basis for $H^{\otimes q}$. The right coevaluation $\overrightarrow{\coev}_{H^{\otimes q}}:\bk_\epsilon\to (H^{\otimes q})^{*}\otimes H^{\otimes q}$ is given by 
    $$\overrightarrow{\coev}_{H^{\otimes q}}:1_\bk\mapsto\sum_{b\in B} b^*\otimes \Delta^{q-1}(g^{-1})b.$$
    Moreover, $\eta:H\to H$ is chosen to be the map given by $\eta = -\cdot\Lambda$. Thus, after the first map $\overrightarrow{\coev}_{H^{\otimes q}}\otimes\eta$ in the composition in \eqref{kuperberg-start}, we have
    \begin{equation}\label{post-coev}
        1_H\mapsto\sum_{b\in B} b^*\otimes \Delta^{q-1}(g^{-1})b\otimes \Lambda =: s_1.
    \end{equation} 
    
    Let $\Omega$ be a basis for $H$. For a non-negative integer $n$, let $\Omega^{n} = \{\omega_1\otimes\dots\otimes\omega_{n}|\omega_i\in \Omega, i=1,\dots,n\}$ and observe that this is a basis for $H^{\otimes n}$. Note that 
    \begin{align*}
        \Omega^{q-1}\otimes 1_H &= \{\hat\bom\otimes 1 | \bom\in\Omega^{q-1}\},\\
        \Omega^{q-1}\otimes 1_H\otimes\Omega^{p-q} &= \{\bom\otimes 1\otimes\bom' | \bom\in\Omega^{q-1}, \bom'\in \Omega^{p-q}\}
    \end{align*}
    are $H$-bases for $H^{\otimes q}$ and $H^{\otimes p}$ respectively. We set $B = \Omega^{q-1}\otimes 1_H$. The following is a basis for $\Hom(\bk_\epsilon\to H^{\otimes p})$: 
    $$\{x_{\bom,\bom'}:1_\bk\mapsto \Delta^{p-1}(\Lambda)(\bom\otimes 1\otimes\bom') | \bom\in\Omega^{q-1}, \bom'\in \Omega^{p-q}\},$$
    The dual basis for $\Hom(H^{\otimes p}\to \bk_{\epsilon})$ under the inner product defined by the modified trace map is generated by
    $$\{y_{\bom,\bom'}:\hat\bom\otimes 1_H\otimes\hat\bom'\mapsto \delta_{\bom=\hat\bom}\delta_{\bom'=\hat\bom'}1_\bk\}.$$
    In $\Rep(H)$, the chromatic morphism $\tilde d_H:H\otimes H\to H\otimes H$ is given by $\tilde d_H(x\otimes y)=\lambda(S(y^{(1)})gx) y^{(2)} \otimes y^{(3)}$. After the second map $\id_{(H^*)^{\otimes q}}\otimes x_{\bom,\bom'}\otimes\id_{H^{\otimes q-1}}\otimes \tilde d$ of the composition, we have
    \begin{align*}
        s_1 =\sum_{h\in\Omega, \hat\bom\in\Omega^{q-1}} &(\Delta^{q-1}(h)(\hat\bom\otimes 1))^*\otimes \Delta^{q-1}(g^{-1}h)(\hat\bom\otimes 1)\otimes \Lambda\\
        = \sum_{h\in\Omega, \hat\bom\in\Omega^{q-1}} &(\Delta^{q-1}(h)(\hat\bom\otimes 1))^*\otimes \Delta^{q-2}(g^{-1}h^{(1)})\hat\bom\otimes g^{-1}h^{(2)}\otimes \Lambda\\
        \mapsto \sum_{h\in\Omega, \hat\bom\in\Omega^{q-1}} &\bigg[(\Delta^{q-1}(h)(\hat\bom\otimes 1))^*\otimes \Delta^{p-1}(\Lambda)(\bom\otimes 1\otimes\bom')\\
        &\otimes\Delta^{q-2}(g^{-1}h^{(1)})\hat\bom\otimes \lambda(S(\Lambda^{(1)})h^{(2)})\Lambda^{(2)}\otimes \Lambda^{(3)}\bigg]\\
        = \sum_{h\in\Omega, \hat\bom\in\Omega^{q-1}} &\bigg[(\Delta^{q-1}(h)(\hat\bom\otimes 1))^*\otimes \Delta^{q-1}(\Lambda^{(1')})(\bom\otimes 1)\otimes\Delta^{p-q-1}(\Lambda^{(2')})\bom'\\
        &\otimes\Delta^{q-2}(g^{-1}h^{(1)})\hat\bom\otimes \lambda(S(\Lambda^{(1)})h^{(2)})\Lambda^{(2)}\otimes \Lambda^{(3)}\bigg] =: s_2^{\bom, \bom'},
    \end{align*}
    where, if $q=1$, we set $\Delta^{-1} = \epsilon$.
    
For the final step of our composition, we plug into $\overleftarrow\ev_{H^{\otimes q}}\otimes y_{\bom,\bom'}\otimes\id$, and find that evaluation kills all but one term in the sum leaving
\begin{align*}
    s_2^{\bom, \bom'} &\mapsto \lambda(S(\Lambda^{(1)})(\Lambda^{(1')})^{(2)})y_{\bom,\bom'}[\Delta^{p-q-1}(\Lambda^{(2')})\bom'\otimes\Delta^{q-2}(g^{-1}(\Lambda^{(1')})^{(1)})\bom\otimes \Lambda^{(2)}] \Lambda^{(3)}\\
    &= \lambda(S(\Lambda^{(1)})\Lambda^{(2')})y_{\bom,\bom'}[\Delta^{p-q-1}(\Lambda^{(3')})\bom'\otimes\Delta^{q-2}(g^{-1}\Lambda^{(1')})\bom\otimes \Lambda^{(2)}] \Lambda^{(3)} =: s_3^{\bom,\bom'}.
\end{align*}

In any finite-dimensional unimodular Hopf algebra $H$, we may assume the basis $\Omega$ of $H$ consists of $1, \Lambda$ and no other terms with $\lambda(\omega)\neq 0$ nor $\epsilon(\omega)\neq 0$. Therefore, letting $\Delta(\Lambda)=\sum\limits_{\omega,\omega'} d_{\omega,\omega'} \omega\otimes \omega' $, since $(\epsilon\otimes \id)(\Delta \Lambda)=\Lambda$, this implies that $d_{1,\Lambda}=1$ and $d_{1,\omega}=0$ for all other $\omega\in\Omega$. Similarly, $d_{\Lambda,1}=1$ and $d_{\omega,1}=0$ for all other $\omega\in\Omega$. Further, by definition of $\lambda$, $(\id \otimes \lambda)(\Delta \Lambda)=\lambda(\Lambda)1=1$, so, assuming $\dim H > 1$,
\begin{align*}
    1 = (\id\otimes \lambda) \Delta(\Lambda) &= (\id \otimes \lambda)(1\otimes \Lambda + \Lambda \otimes 1 + \sum\limits_{\omega,\omega'\ne 1} d_{\omega,\omega'} \omega\otimes \omega')\\
    &=1+\sum\limits_{\omega\ne 1}d_{\omega,\Lambda} \omega,
\end{align*} 
and thus, $d_{\omega,\Lambda}=0$ for all $\omega\ne 1$. This equality is also true when $\dim H = 1$. We need no longer make this assumption going forward. This implies $\lambda(\Lambda^{(2)})\Lambda^{(1)}\otimes \Lambda^{(3)} = 1\otimes 1$. We begin evaluating the modified trace in Equation \eqref{kuperberg-start} as follows:
\begin{align*}
   t_P(s_3^{\bom, \bom'}) &= \lambda(S(\Lambda^{(1)})\Lambda^{(2')})y_{\bom,\bom'}[\Delta^{p-q}(\Lambda^{(3')})\bom'\otimes\Delta^{q-2}(g^{-1}\Lambda^{(1')})\bom\otimes \Lambda^{(2)}]\lambda(g\Lambda^{(3)}) \\
    &= \lambda(g^{-1}\Lambda^{(2)})y_{\bom,\bom'}[\Delta^{p-q-1}(\Lambda^{(3)})\bom'\otimes\Delta^{q-2}(g^{-1}\Lambda^{(1)})\bom\otimes g^{-1}]\\
    &= y_{\bom,\bom'}[\Delta^{p-q-1}(g)\bom'\otimes\bom\otimes 1].
\end{align*}

One may quickly verify the following formula:
\begin{align*}
    y_{\bom,\bom'}(\Delta^{p-1}(h)(\hat\bom\otimes 1\otimes \hat\bom')) &= (\bom^*\otimes\lambda\otimes(\bom')^*)(\Delta^{p-1}(\Lambda)\Delta^{p-1}(h)(\hat\bom\otimes 1\otimes \hat\bom')).
\end{align*}
Thus, we have an equality $y_{\bom,\bom'} = (\bom^*\otimes\lambda\otimes(\bom')^*)(\Delta^{p-1}(\Lambda)\cdot -)$. Therefore, the chromatic spherical invariant is given by
$$\Ku_{\Rep(H)}(L(p,q)) = \sum_{\bom, \bom'} t_P(s_3^{\bom, \bom'}) = \sum_{\bom, \bom'} (\bom^*\otimes\lambda\otimes(\bom')^*)(\Delta^{p-1}(\Lambda)(\Delta^{p-q-1}(g)\bom'\otimes\bom\otimes 1)).$$
If $q\leq p-q$, then the chromatic spherical invariant is given by
\begin{align}
    = \sum_{\bom, \bom'} &\omega_1^*(\Lambda^{(1)}g\omega_1')\dots\omega_{q-1}^*(\Lambda^{(q-1)}g\omega_{q-1}')\lambda(\Lambda^{(q)}g\omega_q')\nonumber\\\times&(\omega'_{1})^*(\Lambda^{(q+1)}g\omega'_{q+1})\dots(\omega'_{p-2q})^*(\Lambda^{(p-q)}g\omega'_{p-q})\nonumber\\\times&(\omega'_{p-2q+1})^*(\Lambda^{(p-q+1)}\omega_1)\dots(\omega'_{p-q-1})^*(\Lambda^{(p-1)}\omega_{q-1})(\omega'_{p-q})^*(\Lambda^{(p)}).\label{eq:q<p-q}
    \intertext{If $q > p-q$, then the chromatic spherical invariant is given by}
    = \sum_{\bom, \bom'} &\omega_1^*(\Lambda^{(1)}g\omega_1')\dots\omega_{p-q}^*(\Lambda^{(p-q)}g\omega_{p-q}')\nonumber\\\times&\omega_{p-q+1}^*(\Lambda^{(p-q)}\omega_1)\dots\omega_{q-1}^*(\Lambda^{(q-1)}\omega_{2q-p-1})\lambda(\Lambda^{(q)}\omega_{2q-p})\nonumber\\\times&(\omega'_1)^*(\Lambda^{(q+1)}\omega_{2q-p+1})\dots(\omega'_{p-q-1})^*(\Lambda^{(p-1)}\omega_{q-1})(\omega'_{p-q})^*(\Lambda^{(p)}).\label{eq:q>p-q}
\end{align}
For $n=1,\dots,q-1,q+1,\dots p$, let $\bar n = n + p\bZ$, set 
\begin{align*}
    \bar\omega_{\overline{n}} &= \begin{cases}
        \omega_n&n<q,\\
        \omega_{n-q}'&n > q,
    \end{cases}
\end{align*} 
and recall the notations of $\Lambda^{(\bar n)}$ and $i_{\bar n}$ from the theorem statement. Then, both expressions \eqref{eq:q<p-q} and \eqref{eq:q>p-q} may be rearranged as follows:
\begin{align*}
    &\Ku_{\Rep(H)}(L(p,q))\\
    &=\sum_{\bom, \bom'}\bar\omega_1^*(\Lambda^{(1)}g^{i_{\overline{1+q}}}\bar\omega_{\overline{1+q}})\bar\omega_2^*(\Lambda^{(2)}g^{i_{\overline{2+q}}}\bar\omega_{\overline{2+q}})\dots\bar\omega_{\overline{q-1}}^*(\Lambda^{(q-1)}g^{i_{\overline{(q-1)+q}}}\bar\omega_{\overline{(q-1)+q}})\lambda(\Lambda^{(q)}g^{i_{\overline{q+q}}}\bar\omega_{\overline{q+q}})\\
    &\times\bar\omega_{\overline{q+1}}^*(\Lambda^{(q+1)}g^{i_{\overline{(q+1)+q}}}\bar\omega_{\overline{(q+1)+q}})\dots\bar\omega_{\overline{p-1}}^*(\Lambda^{(p-1)}g^{i_{\overline{(p-1)+q}}}\bar\omega_{\overline{(p-1)+q}})\bar\omega_{\overline{p}}^*(\Lambda^{(p)})\\
    &= \sum_{\bom, \bom'}\lambda(\Lambda^{(q)}g^{i_{\overline {2q}}}\bar\omega_{\overline{2q}})
\bar\omega_{\overline{2q}}^*(\Lambda^{(\overline{2q})}g^{i_{\overline {3q}}}
\bar\omega_{\overline{3q}})
\bar\omega_{3q}^*(\Lambda^{(\overline{3q})}g^{i_{\overline {4q}}}
\bar\omega_{\overline{4q}})\dots
\bar\omega_{(p-1)q}^*(\Lambda^{(\overline{(p-1)q})}g^{i_{\overline {pq}}}
\bar\omega_{\overline{pq}})\bar\omega_{\overline{pq}}^*(\Lambda^{(p)}).
\end{align*}
Note that $q+p\bZ$ is an additive generator of $\bZ_p$, so every $\Lambda^{(i)}$ for $i=1,\dots, p$ does occur and there are no repeats, so this expression is actually well-defined. Given a basis $B$ of a space $V$, we have the basic linear algebra identity $\sum_{b\in B}b^*(v)b = v$ for any $v\in V$. This allows us to simplify $\Ku_{\Rep(H)}(L(p,q))$ as in the theorem statement:
$$\Ku_{\Rep(H)}(L(p,q)) = \lambda(\Lambda^{(q)}g^{i_{\overline {2q}}}\Lambda^{(\overline{2q})}g^{i_{\overline {3q}}}\Lambda^{(\overline{3q})}g^{i_{\overline {4q}}}
\dots g^{(\overline{(p-1)q})}\Lambda^{(\overline{(p-1)q})}g\Lambda^{(p)}).$$
\end{proof}

For the lens space $L(0,1)$, following the same method dragging of $x_i$, the invariant evaluates to 
\begin{align*}
        \Ku_{\cC}(L(0,1)) &= \displaystyle\sum_i t_G[( (y_i\otimes id_G) \circ \tilde d\circ(x_i \otimes\eta)].
\end{align*}
Specializing to $\cC=\Rep(H)$ and $G=H$ as our projective generator, noting that the bases $\{x_i\}$ and $\{y_i\}$ become the map $1\mapsto \Lambda$ and $\epsilon$ respectively since this is a normalized pair of bases. For $L(0,1)$, this quantity evaluates to \begin{align*}
    \Ku_{\Rep(H)}(L(0,1)) &= \lambda(S(\Lambda^{(1)}) g \Lambda)\epsilon(\Lambda^{(2)}) \lambda(g\Lambda^{(3)})=\epsilon(S(\Lambda^{(1)}))\epsilon(\Lambda^{(2)}) \lambda(g\Lambda^{(3)})=1.
\end{align*}

As a sample calculation, consider the Nichols Hopf algebra $\cK_2$, which is the $8$-dimensional complex Hopf algebra generated by $\{K,\xi_1,\xi_2\}$ subject to the relations, for $i=1,2$
\begin{equation}\label{Kn-pres}
 K^2 = 1,\hspace{1em} \xi_i^2 = 0,\hspace{1em}K\xi_i =-\xi_i K,\hspace{1em}\xi_1\xi_2=-\xi_2\xi_1. 
\end{equation}
The comultiplication $\Delta:\cK_2 \to \cK_2 \otimes \cK_2$, counit $\epsilon:\cK_2 \to\bC$, and antipode $S:\cK_2 \to \cK_2$ are given as follows: 
\begin{align*}
\Delta(K) &= K\otimes K, & \epsilon(K) &= 1, & S(K) &= K,\\
\Delta(\xi_i) &= K\otimes \xi_i + \xi_i\otimes 1, & \epsilon(\xi_i) &= 0, & S(\xi_i) &= -K\xi_i.
\end{align*}
Using these definitions, $\cK_2$ is a pivotal Hopf algebra with pivot given by $g=K$, a two-sided integral $\Lambda=\xi_1\xi_2(1+K)$ and a two-sided cointegral $\lambda=(\xi_1\xi_2)^*$ where the dual is taken in the basis $B = \{1,K,\xi_1,K\xi_1,\xi_2,K\xi_2,\xi_1\xi_2,K\xi_1\xi_2\}$. Therefore, $\cK_2$ is both unibalanced and unimodular. 
\begin{theorem}\label{thm:K2-calculation}
    The chromatic spherical invariant of the lens spaces $L(p, 1)$ for $\cK_2$ is given, for even $p\geq 2$, by $\Ku_{\Rep(\cK_2)}(L(p, 1))= 0$. For odd $p\geq 1$, they are given by $\Ku_{\Rep(\cK_2)}(L(p, 1))=1$.
\end{theorem}
\begin{proof}
First, note that $$\Delta^{n+1}(\Lambda)= (\Delta^{n}\otimes \id)\big((1\otimes \xi_1\xi_2 +\xi_1\xi_2\otimes 1 + K\xi_2\otimes \xi_1 -K\xi_1\otimes \xi_2)(1\otimes 1 + K\otimes K)\big).$$ Let $m$ denote the multiplication of all terms in a tensor product. Then for any power $n$, note that 
\[m(\Delta^n(\xi_j))=\sum\limits_{i=0}^n K^i \xi_j = \left(\left\lfloor \frac{n}{2}\right\rfloor + 1\right)\xi_j + \left\lfloor \frac{n+1}{2}\right\rfloor K\xi_j.\]
    Let $m(\Delta^n(\xi_1\xi_2))=a_n \xi_1\xi_2 + b_n K \xi_1 \xi_2$ and let $m(\Delta^n(K\xi_1\xi_2))=c_n \xi_1\xi_2 + d_n K \xi_1 \xi_2$. Then, \begin{align*}m(\Delta^{n+1}(\xi_1\xi_2))&=m\big(1\otimes \Delta^{n}(\xi_1\xi_2) + \xi_1\xi_2 \otimes \Delta^{n}(1) + K\xi_2 \otimes \Delta^{n}(\xi_1) -K\xi_1\otimes \Delta^{n}(\xi_2)\big) )\\
    &=\left(a_n + 1 + 2\left\lfloor \frac{n+1}{2}\right\rfloor\right) \xi_1\xi_2 +\left(b_n - 2\left(\left\lfloor \frac{n}{2} \right\rfloor+1\right)\right) K\xi_1\xi_2.
    \end{align*}
    Thus, $b_{n+1}=b_n - 2(\lfloor \frac{n}{2} \rfloor+1)$. Observe that $b_0 = 0$. Thus, $b_{n}=-2n-2\sum\limits_{i=1}^{n-1}\lfloor \frac{i}{2} \rfloor$. Therefore, $b_n=-\lfloor\frac{(n+1)^2}{2}\rfloor$.
    
    Now, we compute $m(\Delta^n(K\xi_1\xi_2))$. We may obtain a recurrence as follows.
    \begin{align*}
        m(\Delta^{n+1}(K\xi_1\xi_2))&=m\big(\Delta^n(K)\otimes K\xi_1\xi_2 + \Delta^n(K\xi_1\xi_2) \otimes K + \Delta^n(\xi_2) \otimes K\xi_1 -\Delta^n(\xi_1) \otimes K\xi_2\big) \\
        &= \left(d_n + \left\lfloor \frac{n+1}{2}\right\rfloor\right) \xi_1\xi_2 + 2\left(c_n + \left\lfloor \frac{n}{2} \right\rfloor+1\right) K\xi_1\xi_2 + K^{n+1}\xi_1\xi_2.
    \end{align*}
    Let 
    $$p(n) = \begin{cases}
        0 & n\equiv 0\pmod{2},\\
        1 & n\equiv 1\pmod{2}.
    \end{cases}$$
    Then, we may express a recurrence on $c_n$ and $d_n$ as follows:
    \begin{align*}
        c_{n+1}&=d_n+p(n+1)+2\left\lfloor \frac{n+1}{2}\right\rfloor,\\
        d_{n+1}&=c_n+p(n)+2\left(\left\lfloor \frac{n}{2}\right\rfloor + 1\right).
    \end{align*} 
    Substituting the $c_n$ relation into the expression for $d_n$, we get that $d_{n+1}=d_{n-1} + 2p(n) + 4\lfloor n/2\rfloor+2$, which simplifies to $d_{n+1}=d_{n-1}+2n+2$. Our base cases are $d_0=1$ and $d_1=2$. Solving the recurrence relation, we get that $d_n=\frac{(n+1)^2+1}{2}$ for even $n$. For odd $n$, $d_n=\frac{(n+1)^2}{2}$. Then, the chromatic spherical invariant for $L(p,1)$ is given by 
    $$\Ku_{\Rep(\cK_2)}(L(p,1)) = (\xi_1\xi_2)^*(Km(\Delta^{p-1}((1+K)\xi_1\xi_2))) = b_{p-1}+d_{p-1},$$
    in the basis $B = \{1,K,\xi_1,K\xi_1,\xi_2,K\xi_2,\xi_1\xi_2,K\xi_1\xi_2\}$. For odd $p\geq 1$, this evaluates to $1$. For even $p\geq 2$, this evaluates to $0$. 
\end{proof}

\section{Admissible skein modules $\SM(M)$}\label{section:admissible}

In this section we compute the admissible skein module of handlebodies and lens spaces.

\subsection{Definitions} We first define the admissible skein module and discuss some previously known computations for the $2$- and $3$-dimensional discs and spheres. 

\begin{definition}[$\mathcal I$-skein relations]
    Let $M$ be an $n$-dimensional manifold, where $n= 2,3$. Recall the definition of $\cI$-admissible graphs and ribbon graphs from Subsection \ref{subsection:icolouredgraphs}. Suppose two $\cI$-admissible graphs or ribbon graphs $\Gamma_1$ and $\Gamma_2$ in $M$ are identical outside of an embedded box $[0,1]^n\subseteq M$ whose top and bottom boundary each intersect an $\cI$-colored edge/ribbon transversely and nontrivially and otherwise there is no intersection of $\partial [0,1]^n$ with $\Gamma_1$ or $\Gamma_2$. Then, $\Gamma_1$ and $\Gamma_2$ are related by an \textit{$\cI$-skein relation} if $\mathcal{P}_n(\Gamma_1\cap [0,1]^n) = \mathcal{P}_n(\Gamma_2\cap [0,1]^n)$, where $\Pen_n$ is the Penrose or Reshetikhin-Turaev functor for $n=2$ or $n=3$ respectively.
\end{definition}

\begin{definition}[Admissible skein modules] 
Let $M$ be an oriented surface (or $3$-manifold), $\mathcal C$ be a pivotal (resp., ribbon) tensor category, and $\mathcal I$ be its tensor ideal. The {\it admissible skein module} of $M$, defined in \cite{admissible} and denoted by $\mathcal S_{\mathcal I}(M)$, is defined to be the $\Bbbk$-linear span of all $\mathcal I$-admissible graphs (resp., ribbon graphs)  in $M$ modulo the $\mathcal I$-skein relations.

\end{definition}

\begin{prop}[\cite{admissible}] The admissible skein module of $M$ is generated by $\mathcal I$-colored (ribbon) graphs. 
    
\end{prop}

\begin{theorem}[\cite{admissible}]
For a tensor ideal $\cI$ in a pivotal (for 2-manifolds) or ribbon (for 3-manifolds) tensor category $\cC$, we have the following identifications:
\begin{enumerate}[(i)]
    \item $\mathcal S_{\mathcal I}(D^2\setminus\{x\})^* \cong \mathcal S_{\mathcal I}(H_1)^*\cong \{$cyclic traces on  $\mathcal I\}$.
    \item $\mathcal S_{\mathcal I}(D^2)^* \cong \{$right m-traces on $\mathcal I\} \cong \{$left m-traces on $\mathcal I\}$.
    \item $ \mathcal S_{\mathcal I}(S^2)^* \cong \mathcal S_{\mathcal I}(D^3)^* \cong \mathcal S_{\mathcal I}(S^3)^*\cong \{$m-traces on $\mathcal I\}$.
\end{enumerate}
    
\end{theorem}

\subsection{Computations for handles and handlebodies}\label{subsec:handlebodies}
Suppose $\cI$ is a tensor ideal of $\cC$ and is closed under direct sums (e.g., the projective ideal). Fix a set $R$ of representatives of the isomorphism classes of objects in $\cI$. We may partially order the objects of $R$ by $X\leq Y$ if $X\in R$ is a retract of $Y\in R$. Call an object $X\in R$ \textit{square-free} if $X$ contains each of its indecomposable direct summands at most once up to isomorphism. Then, we may express the skein modules $\SM(H_g)$ and $\SM(L(p,q))$ as limits over all isomorphism classes of objects in the category or just the square-free ones. 

\begin{theorem}\label{limit-skein-module-iso} 
    Suppose $\cI$ is a tensor ideal of $\cC$ and is closed under direct sums. We may express the skein module $\SM(H_g)$ of a genus $g$ handlebody $H_g$ as the following direct limits:
        \begin{align*}
            \SM(H_g) &\cong \lim_{\substack{\longrightarrow\\P\in R}} \End(P^{\otimes g}) / [\End(P^{\otimes g}), \End(P)^{\otimes g}]\\
            &\cong \lim_{\substack{\longrightarrow\\\text{square-free }\\ P\in R}} \End(P^{\otimes g}) / [\End(P^{\otimes g}), \End(P)^{\otimes g}].
        \end{align*}
    \end{theorem}
    In particular, if $G$ is a minimal generator (equivalently, maximal square-free object)  of $\cI$, then 
    $$\SM(H_g) = \End(G^{\otimes g}) / [\End(G^{\otimes g}), \End(G)^{\otimes g}].$$
\begin{proof}
    The following proof has three components: (1) precisely defining the directed systems which we take the limit over, (2) showing the first isomorphism in the theorem statement, and (3) showing the second isomorphism in the theorem statement. The first and third steps are purely algebraic and fairly short. The bulk of the work, of course, lies in the topological construction used in part (2).
    
    (1) For $P, Q\in R$ with $P\leq Q$, we must define a directed system of linear maps 
    $$E^g_{P, Q}: \End(P^{\otimes g}) / [\End(P^{\otimes g}), \End(P)^{\otimes g}]\to \End(Q^{\otimes g}) / [\End(Q^{\otimes g}), \End(Q)^{\otimes g}].$$ 
    Let $\iota_{P, Q}:P\incl Q$ and $\pi_{P, Q}:Q\twoheadrightarrow P$ be inclusions and projections respectively. Because $\otimes$ is exact in any monoidal category, $\iota_{P, Q}^{\otimes g}$ and $\pi_{P, Q}^{\otimes g}$ are also inclusions and projections. We may define $\tilde E_{P, Q}^{g}:\End(P^{\otimes g})\to \End(Q^{\otimes g})$ by
    $$\tilde E_{P, Q}^{g}:F\mapsto \iota_{P, Q}^{\otimes g}\circ F\circ \pi_{P, Q}^{\otimes g}.$$
    Note the following properties about $\tilde E_{P, Q}^{g}$: 
    \begin{itemize}
        \item $\tilde E_{P, Q}^{g}$ is linear
        \item $\tilde E_{P, Q}^{g}(F\circ G) = \tilde E_{Q, T}^{g}(F)\circ \tilde E_{P, T}^{g}(G)$,
        \item $\tilde E_{P, Q}^{g_1}\otimes \tilde E_{P, Q}^{g_2} = \tilde E_{P, Q}^{g_1+g_2}|_{\End(P^{\otimes g_1})\otimes \End(Q^{\otimes g_2})} $.
    \end{itemize}
    It follows from these properties that 
    $$\tilde E_{P, Q}^{g}([F, f_1\otimes\dots\otimes f_g]) = [\tilde E_{P, Q}^{g}(F), \tilde E_{P, Q}^{1}(f_1)\otimes\dots\otimes \tilde E_{P, Q}^{1}(f_g)].$$ 
    Therefore, $\tilde E_{P, Q}^{g}$ restricts to a map $[\End(P^{\otimes g}), \End(P)^{\otimes g}]\to [\End(Q^{\otimes g}), \End(Q)^{\otimes g}]$. It follows that there is a map $E^g_{P, Q}$ with the source and target described such that the following diagram commutes:
    $$\begin{tikzcd}
        \End(P^{\otimes g})\arrow[r,"\tilde E^g_{P, Q}"]\arrow[dd,two heads,"\mathrm{quot.}"'] & \End(Q^{\otimes g})\arrow[dd,two heads,"\mathrm{quot.}"]\\\\
        \End(P^{\otimes g}) / [\End(P^{\otimes g}), \End(P)^{\otimes g}]\arrow[r, dashed, "E^g_{P, Q}"'] & \End(Q^{\otimes g}) / [\End(Q^{\otimes g}), \End(Q)^{\otimes g}].
    \end{tikzcd}$$
    Observe that regardless of the choice of inclusion/projection pair $\iota_{P, Q}:P\incl Q$ and $\pi_{P, Q}:Q\twoheadrightarrow P$, the map $E_{P, Q}$ is the same. Inside $\End(P^{\otimes g}) / [\End(P^{\otimes g}), \End(P)^{\otimes g}]$, we have the following string of equalities:
    $$[\iota_{P, Q}^{\otimes g}\circ F\circ \pi_{P, Q}^{\otimes g}] = [\iota_{P, Q}^{\otimes g}\circ \hat\pi_{P, Q}^{\otimes g}\circ \hat\iota_{P, Q}^{\otimes g}\circ F\circ \pi_{P, Q}^{\otimes g}] = [\hat\iota_{P, Q}^{\otimes g}\circ F\circ \pi_{P, Q}^{\otimes g}\circ \iota_{P, Q}^{\otimes g}\circ \hat\pi_{P, Q}^{\otimes g}] = [\hat\iota_{P, Q}^{\otimes g}\circ F\circ \hat\pi_{P, Q}^{\otimes g}],$$
    so the limit is well-defined. Our goal now is to show that $\SM(H_g)$ is the limit of the $E^g_{P, Q}$ over both $R$ and square-free objects in $R$.

 \begin{figure}[ht]
    \centering
  $\vcenter{\hbox{\begin{overpic}[unit=1mm, scale = .07]{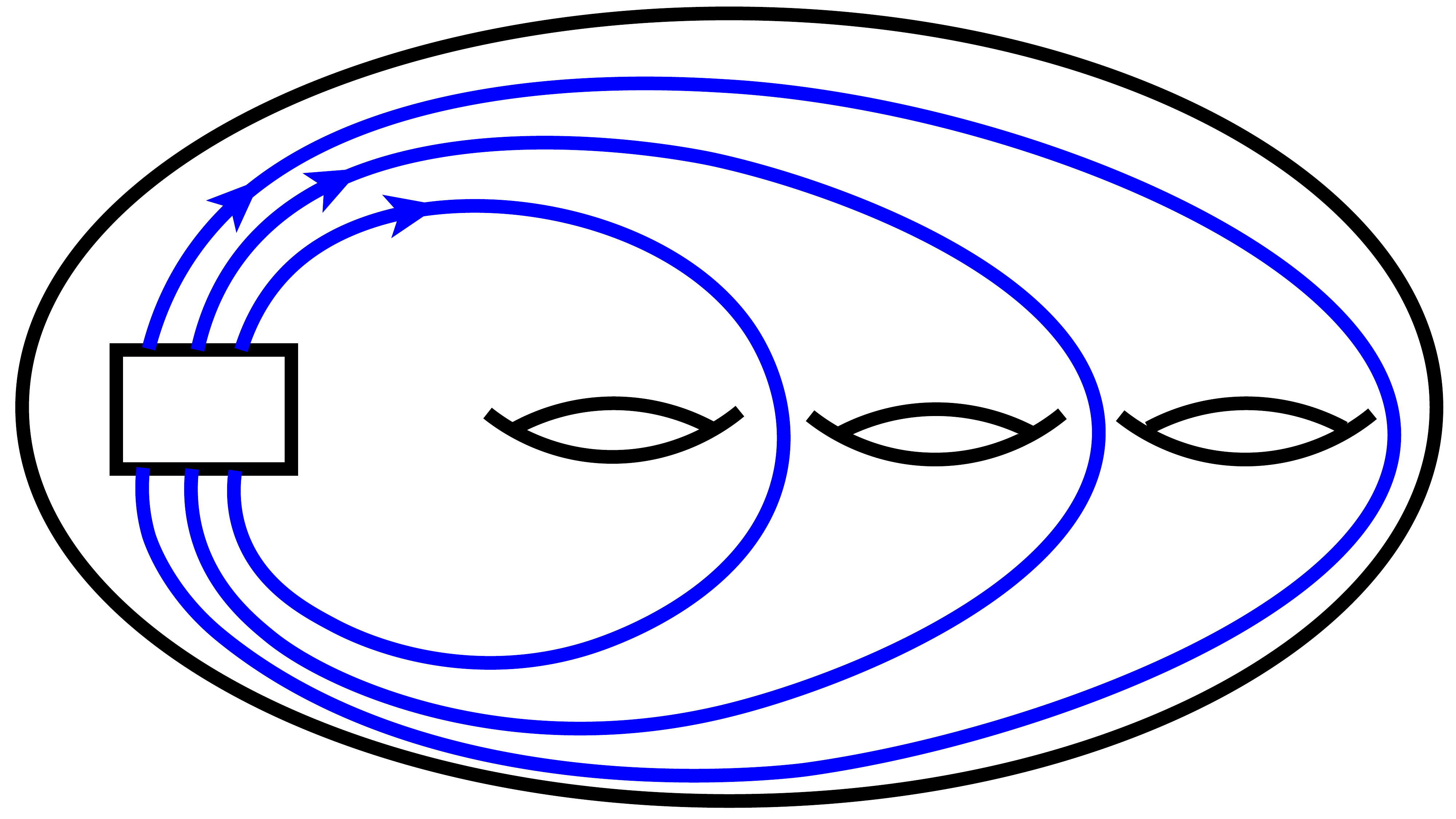}
  \put(5.3,12.1){\footnotesize{$F$}}
  \put(32,14.1){\tiny{$P$}}
\put(21.5,14.1){\tiny{$P$}}
\put(41,14.1){\tiny{$P$}}
  \end{overpic} }}  $
    \caption{Given a morphism $F\in \End(P^{\otimes g})$, we may obtain an element of $\SM(H_g)$ by tracing around each hole. The example of $g=3$ is illustrated in this figure.} 
    \label{fig:end-skein-projection}
\end{figure} 

\begin{figure}[ht]
    \centering
$\vcenter{\hbox{\begin{overpic}[scale=.35]{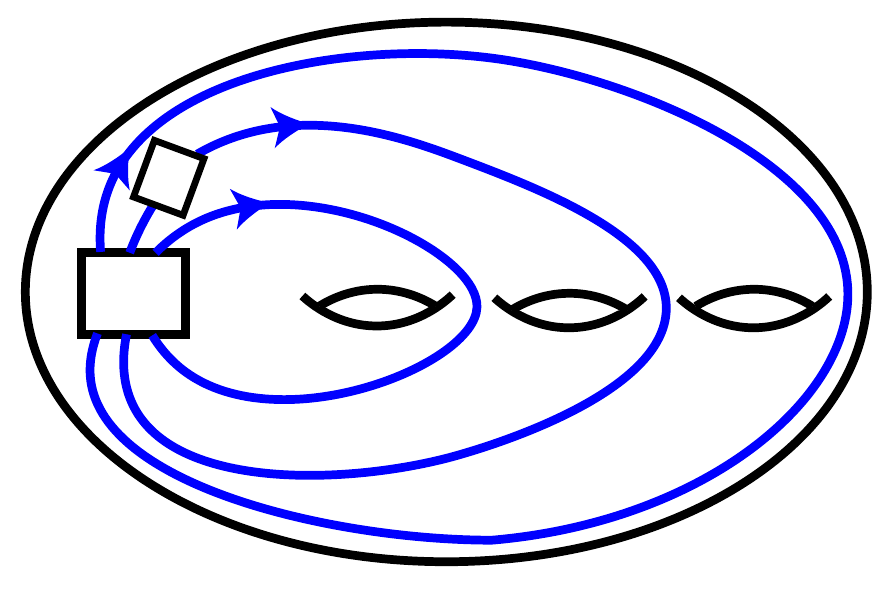}
\put(20,47){\tiny{$F$}}
\put(26,67.5){\tiny{$f$}}
\put(55,83){\tiny{$P$}}
\put(55,70){\tiny{$P$}}
\put(45,57){\tiny{$P$}}
\end{overpic}}} =
\vcenter{\hbox{\begin{overpic}[scale=.35]{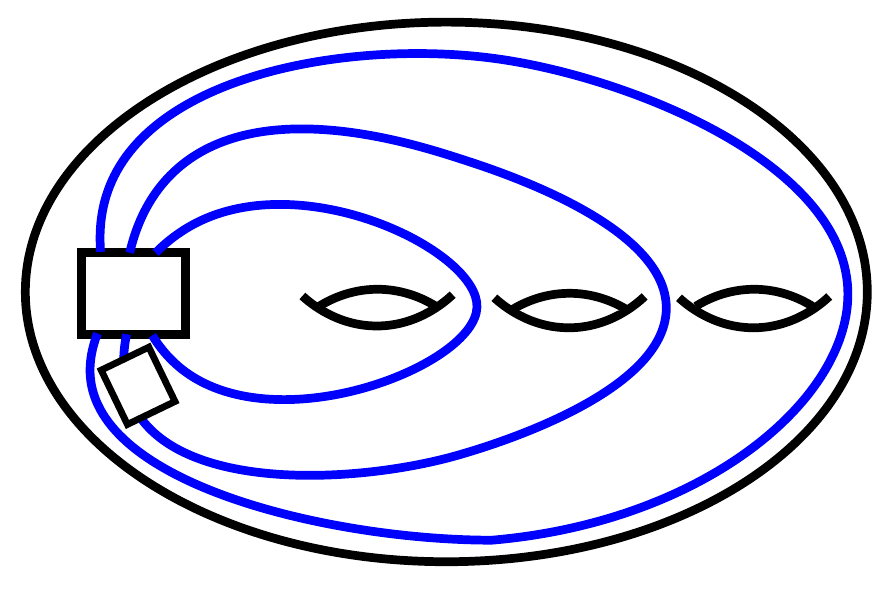}
\put(20, 47){\tiny{$F$}}
\put(21,32){\tiny{$f$}}
\end{overpic}}}$
\vspace*{2mm}
    \caption{A coupon on an admissible ribbon graph may loop around a hole in $H_g$ without changing the element of $\SM(H_g)$.}
    \label{fig:loop-coupon-around}
\end{figure}
    
    (2) For each $P\in R$, there is a linear map $p_{P}:\End(P^{\otimes g})\to \SM(H_g)$, seen in Figure \ref{fig:end-skein-projection}. As visualized in Figure \ref{fig:loop-coupon-around}, this satisfies the property that
    $$p_{P}((f_1\otimes \dots\otimes f_g)\circ F) = p_{Q}(F\circ(f_1\otimes \dots\otimes f_g))$$
    for any choice of $f_i:Q\to P$ and $F:P^{\otimes g}\to Q^{\otimes g}$. This implies that
    $$[\End(P^{\otimes g}), \End(P)^{\otimes g}]\subseteq \ker p_{P},$$
    so $p_{P}$ factors through a map $p'_{P}:
    $$\End(P^{\otimes g})/[\End(P^{\otimes g}), \End(P)^{\otimes g}]\to \SM(H_g)$. Moreover, if $P$ is a retract of $Q$ with projection $\pi_{P,Q}:Q\to P$ and inclusion $\iota_{P,Q}:P\to Q$, then, for any $F\in \End(P^{\otimes g})$,
    $$p_{P}(F) = p_{P}(\pi_{P, Q}^{\otimes g}\circ\iota_{P, Q}^{\otimes g}\circ F) = p_{Q}(\pi_{P, Q}^{\otimes g}\circ F\circ\pi_{P, Q}^{\otimes g}) = (p_{Q}\circ \tilde E^g_{P, Q})(F).$$
    Thus, by the universal property of direct limits, there is a map 
    $$p':\lim_{\substack{\longrightarrow \\P\in R}}\End(P^{\otimes g})/[\End(P^{\otimes g}), \End(P)^{\otimes g}]\to \SM(H_g).$$
    For the rest of step (2), denote by $[F]$ the image of an element $F\in\End(P^{\otimes g})$ in the limit
    $\displaystyle\lim_{\substack{\longrightarrow \\P\in R}}\End(P^{\otimes g})/[\End(P^{\otimes g}), \End(P)^{\otimes g}].$

\begin{figure}[ht]
  $$\Gamma \;=\;
  \vcenter{\hbox{\begin{overpic}[abs, unit=1mm, scale = .2]{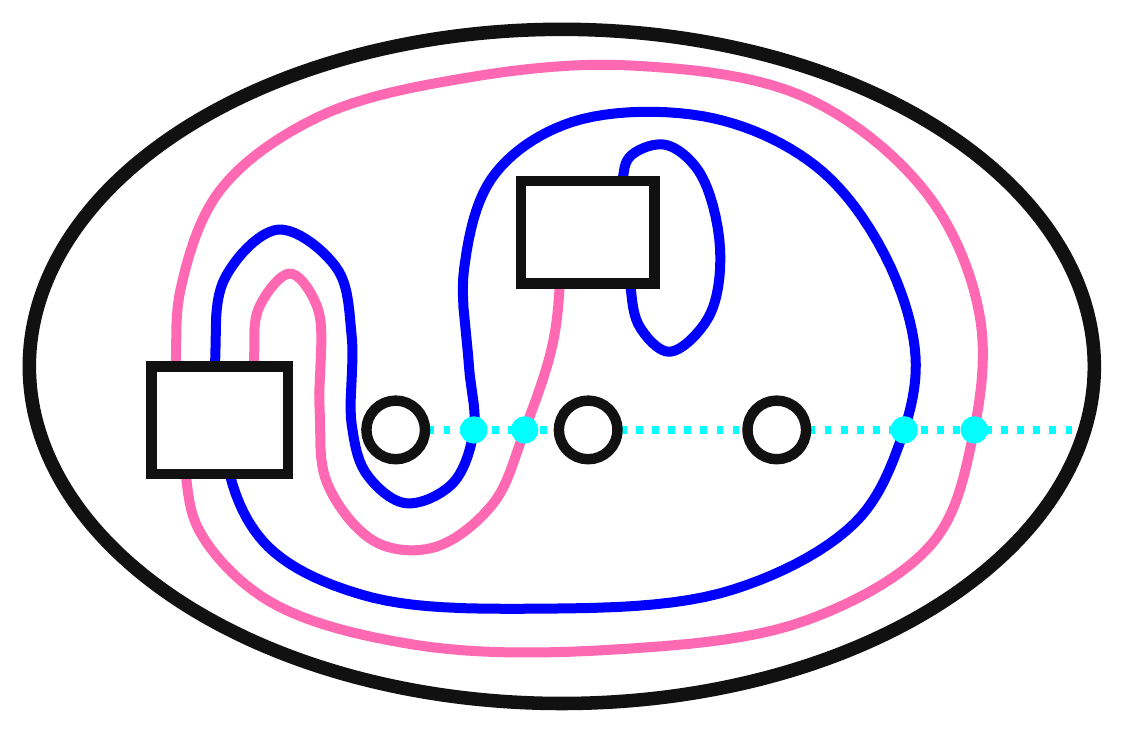}
    \put(7.5,10.6){\makebox(0,0){\footnotesize{$g_1$}}}
    \put(19.9,17.0){\makebox(0,0){\footnotesize{$g_2$}}}
    \put(4.0,15.4){\makebox(0,0){\scriptsize{$X_1$}}}
    \put(11.6,18.6){\makebox(0,0){\scriptsize{$P_1$}}}
    \put(17.9,6.6){\makebox(0,0){\scriptsize{$X_2$}}}
    \put(26.2,16.4){\makebox(0,0){\scriptsize{$P_2$}}}
  \end{overpic} }}
  \;=\;
  \vcenter{\hbox{\begin{overpic}[abs, unit=1mm, scale = .2]{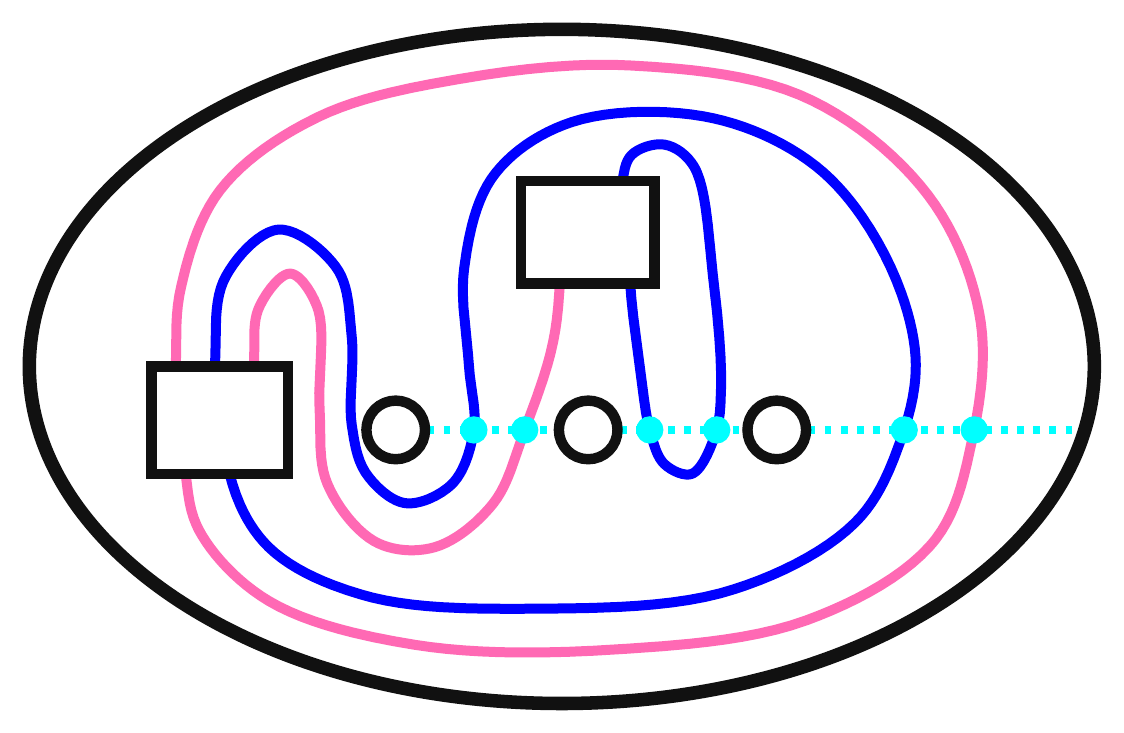}
    \put(7.5,10.6){\makebox(0,0){\footnotesize{$g_1$}}}
    \put(19.9,17.0){\makebox(0,0){\footnotesize{$g_2$}}}
    \put(4.0,15.4){\makebox(0,0){\scriptsize{$X_1$}}}
    \put(11.6,18.6){\makebox(0,0){\scriptsize{$P_1$}}}
    \put(17.9,6.6){\makebox(0,0){\scriptsize{$X_2$}}}
    \put(26.2,16.4){\makebox(0,0){\scriptsize{$P_2$}}}
  \end{overpic} }} =\;
  \vcenter{\hbox{\begin{overpic}[abs, unit=1mm, scale = .2]{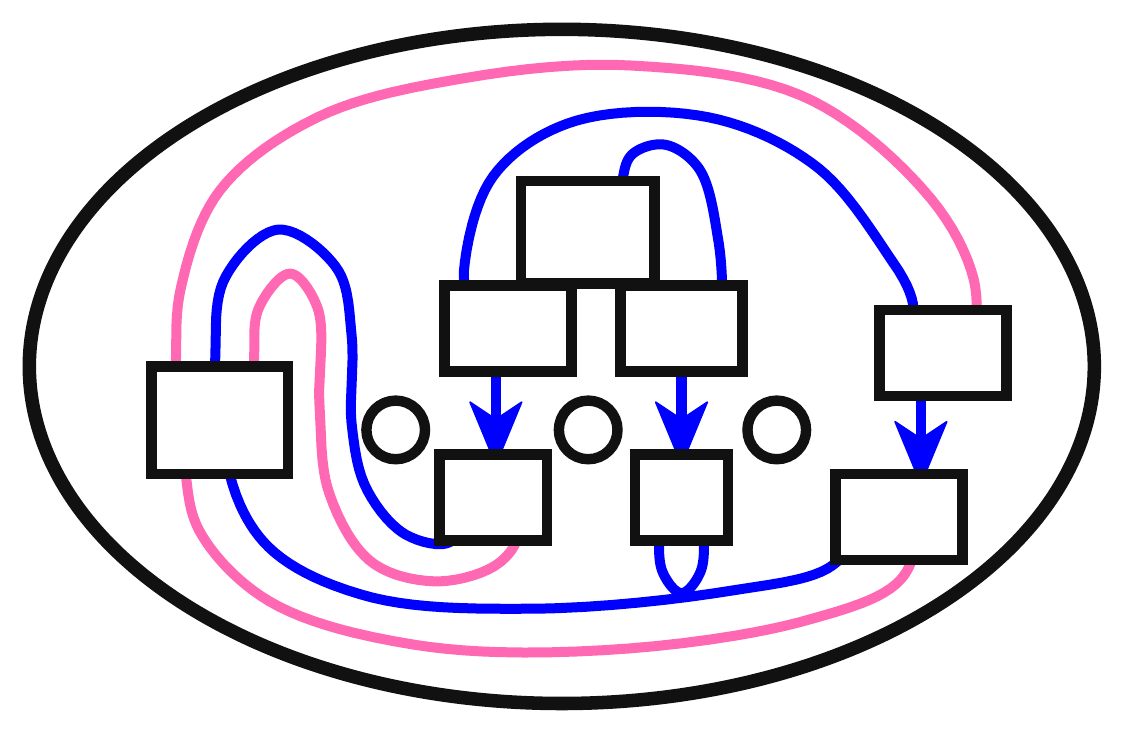}
    \put(7.5,10.6){\makebox(0,0){\footnotesize{$g_1$}}}
    \put(19.9,17.0){\makebox(0,0){\footnotesize{$g_2$}}}
    \put(17.2,13.7){\makebox(0,0){\scriptsize{$\pi_1$}}}
    \put(23.1,13.7){\makebox(0,0){\scriptsize{$\pi_2$}}}
    \put(31.9,12.9){\makebox(0,0){\scriptsize{$\pi_3$}}}
    \put(16.7,7.9){\makebox(0,0){\scriptsize{$\iota_1$}}}
    \put(23.1,7.9){\makebox(0,0){\scriptsize{$\iota_2$}}}
    \put(30.5,7.3){\makebox(0,0){\scriptsize{$\iota_3$}}}
    \put(4.0,15.4){\makebox(0,0){\scriptsize{$X_2$}}}
    \put(11.6,18.6){\makebox(0,0){\scriptsize{$P_1$}}}
    \put(20.0,6.6){\makebox(0,0){\scriptsize{$X_3$}}}
    \put(26.2,17.2){\makebox(0,0){\scriptsize{$P_2$}}}
  \end{overpic} }}$$
  $$\qquad F_\Gamma \;=\;
  \vcenter{\hbox{\begin{overpic}[abs, unit=1mm, scale = .2]{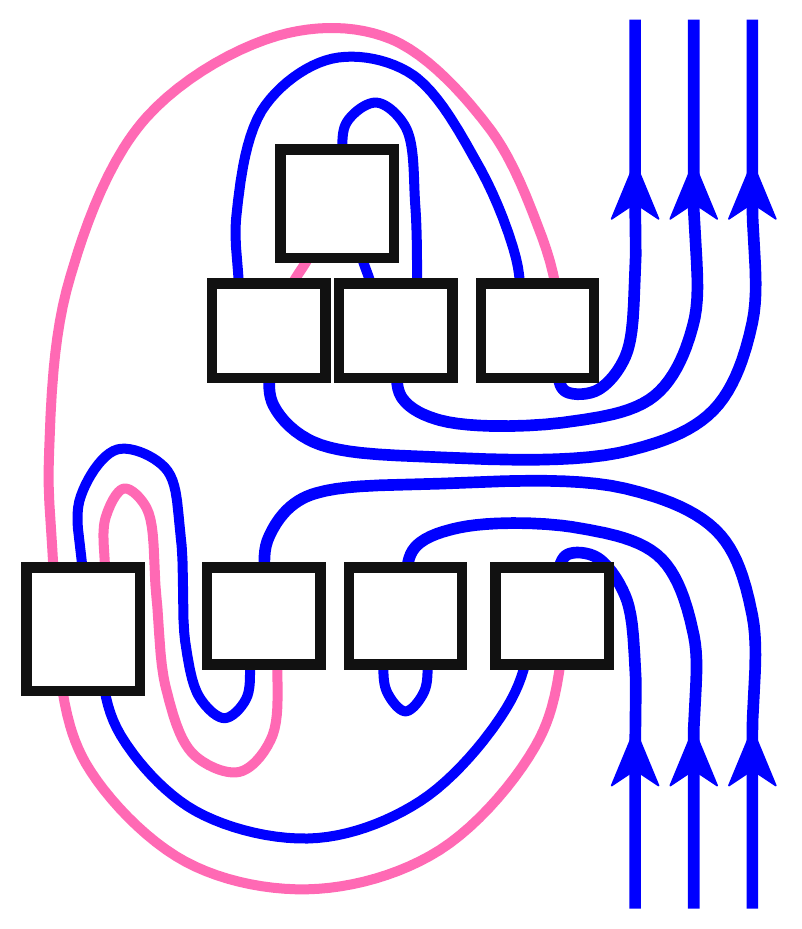}
 \put(11.4,24.5){\makebox(0,0){\scriptsize{$g_2$}}}
    \put(2.8,10.1){\makebox(0,0){\scriptsize{$g_1$}}}
    \put(9.1,20.2){\makebox(0,0){\scriptsize{$\pi_1$}}}
    \put(13.4,20.2){\makebox(0,0){\scriptsize{$\pi_2$}}}
    \put(18.2,20.2){\makebox(0,0){\scriptsize{$\pi_3$}}}
    \put(8.9,10.6){\makebox(0,0){\scriptsize{$\iota_1$}}}
    \put(13.7,10.6){\makebox(0,0){\scriptsize{$\iota_2$}}}
    \put(18.7,10.6){\makebox(0,0){\scriptsize{$\iota_3$}}}
  \end{overpic} }}\;.$$
 \caption{An example of the manipulations to take an admissible graph $\Gamma$ on a thrice-punctured disk and produce an element $F_\Gamma$ of $\End(P^{\otimes g})$ for some $P\in R$, such that its image under the map given in Figure \ref{fig:end-skein-projection} is $\Gamma$.}
    \label{fig:handlebody-skein-manipulations}
\end{figure}

    In the same fashion as the definition of the Reshetikhin-Turaev functor $\Pen_3$ and as in \cite[Corollary 3.2]{admissible}, any $\cI$-admissible ribbon graph $\Gamma$ may be mapped to an $\cI$-admissible graph in the $g$-punctured disk $D^2_g$. As before, we replace crossings and twists with coupons colored by the braiding and twist. This defines an isomorphism $\SM(H_g)\cong \mathcal S_{\mathcal I}(D^2_g)$, so we conflate an element $\sum [\Gamma]\in \SM(H_g)$ with the corresponding element in $\SM(D^2_g)$.
    
    We now show that $p'$ is surjective. Let $\Gamma$ be an admissible graph in $D^2_g$. Choose non-intersecting cuts $r_i$ such that, for $i=1,\dots, g-1$, $r_i$ is a line segment from the $i$th puncture to the $(i+1)$st puncture, and $r_g$ is a line segment from the $g$th puncture to the boundary. We may perform an isotopy on $\Gamma$ so that each cut $r_i$ intersects $\Gamma$ transversely, does not intersect any coupons of $\Gamma$, and intersects with at least one $\cI$-colored edge of $\Gamma$. For each cut $r_i$, let $P_i$ be the ordered tensor product of the objects coloring strands passing through $r_i$. Choose $P\in R$ so that $P^*$ contains $P_1\oplus\dots\oplus P_g$ as a direct summand (such a $P$ exists because tensor ideals are closed under duals), and let $\pi_i:P^*\to P_i$ and $\iota_i:P_i\to P^*$ be the projection and inclusion respectively. For each $i$, split the edges passing through $r_i$ by inserting $\pi_i$ above $r_i$ and $\iota_i$ below $r_i$. Of course, the resulting $\cC$-colored graph is equal to $\Gamma$ up to $\cI$-skein relations, so they agree as elements of $\SM(D^2_g)$. Since $D^2_g\setminus\bigcup r_i$ is contractible and homeomorphic to a square, the skein contained in $D^2_g\setminus \{r_1,\dots,r_g\}$ is equivalent in $\cI$ to a morphism $F_\Gamma \in \End(P^{\otimes g})$ such that $p'([F_{\Gamma}]) = p_{P}(F_\Gamma) = [\Gamma]$, so $p'$ is surjective because the set of admissible ribbon graphs spans the admissible skein module $\SM(H_g)$. (The process described here may be seen visually in Figure \ref{fig:handlebody-skein-manipulations}.)
    
    To complete step (2), we show that $p'$ is injective. Suppose $F\in \End(P^{\otimes g})$, $F'\in \End(Q^{\otimes g})$, and $p'([F])=p'([F'])$. Then, the ribbon graphs $\Gamma$ and $\Gamma'$ representing $p_{P}(F)$ and $p_{Q}(F')$ as in Figure \ref{fig:end-skein-projection} differ by a sequence of skein relations and isotopies. Write this sequence of admissible graphs as $\Gamma = \Gamma_0, \Gamma_1, \dots, \Gamma_{n} = \Gamma'$. Choose $Z$ so that every $[F_{\Gamma_k}]$ is equivalent to an element of $\End(Z^{\otimes g})/[\End(Z^{\otimes g}), \End(Z)^{\otimes g}]$ in the limit, and assume without loss of generality that $F_{\Gamma_k}$ is this morphism and $P = Q = Z$. Note that $F_{\Gamma} = F$ and $F_{\Gamma'} = F'$.
    
    Suppose $\Gamma_k$ and $\Gamma_{k+1}$ differ by a single skein relation. That is, there is a square $[0,1]^2\incl D^2_g$ with the following properties:
    \begin{itemize} 
        \item $\Gamma_k|_{([0,1]^2)^c\cup \partial [0,1]^2} = \Gamma_{k+1}|_{([0,1]^2)^c\cup \partial [0,1]^2}$,
        \item the boundary $\partial [0,1]^2$ intersects no coupons in $\Gamma_k$ or $\Gamma_{k+1}$,
        \item at the top and bottom of $[0,1]^2$, $\partial [0,1]^2$ intersects at least one edge labeled by objects in $\cI$ transversely,
        \item the rest of the boundary of $\partial [0,1]^2$ intersects no edges,
        \item $\Pen_2(\Gamma_k\cap [0,1]^2)=\Pen_2(\Gamma_{k+1}\cap [0,1]^2)$ as morphisms in $\cI$.
    \end{itemize}  
    Because $[0,1]^2$ is contractible, we may assume without loss of generality that $[0,1]^2$ is disjoint from the cuts $r_i$. In this case, $F_{\Gamma_k} = F_{\Gamma_{k+1}}$. 
    
    Now, suppose $\Gamma_k$ and $\Gamma_{k+1}$ differ by an isotopy of a single coupon or edge. If the isotopy does not make changes around any cut $r_i$, then $F_{\Gamma_k} = F_{\Gamma_{k+1}}$ because the relation corresponding to the isotopy can be made internally to $\End(P^{\otimes g})$. Without loss of generality, now assume the isotopy interacts with a single cut $r_i$ (if it interacts with several cuts, then we can split the isotopy into several steps, one cut at a time). Observe that $D^2_g\setminus \{r_1,\dots, r_{i-1},r_{i+1},\dots, r_g\}$ is a once-punctured disk (as in \cite{admissible}), so we obtain that $F_{\Gamma_k}$ and $F_{\Gamma_{k+1}}$ are related by a cyclic relation on the $i$th component as in Figure \ref{fig:loop-coupon-around}, which is killed in the quotient so that $[F_{\Gamma_k}] = [F_{\Gamma_{k+1}}]$. This handles all moves which can induce $p'([F]) = p'([F'])$, so $p'$ is an isomorphism: 
    $$\SM(H_g) \cong \lim_{\substack{\longrightarrow\\P\in R}} \End(P^{\otimes g}) / [\End(P^{\otimes g}), \End(P)^{\otimes g}].$$
    
    (3) Finally, we show that the limit over square-free $P$ yields the same limit. Given $P\in R$, let $P'\in R$ denote the square-free object containing the same indecomposables (without repeats). Of course, for any $P, Q\in R$, there is a map $\psi:\lim_{\text{square-free}} E^g_{P', Q'}\to \lim_{R}E^g_{P, Q}.$ 
    Given a morphism $F\in\End(P^{\otimes g})$ and $U\in R$ an indecomposable retract of $P$, if $n_U$ copies of $U$ are contained in $P$, let
    $$F_U^i = \iota_{U, P'}^{\otimes g}\circ \pi_{U,P}^{\otimes g} \circ F\circ \iota_{U,P}^{\otimes g}\circ\pi_{U, P'}^{\otimes g}.$$
    for $i=1,\dots, n_U$. Then, it is easy to see $[F]=[\sum_{U,i} F_U^i]$. We can therefore define a map $\End(P^{\otimes g})\to \lim_{\text{square-free}} E^g_{P', Q'}$ by $F\mapsto \sum_{U,i} [F_U^i]$, which is easily seen to be well-defined. This again preserves commutators, so we get a map $\lim_{R}E^g_{P, Q}\to \lim_{\text{square-free}} E^g_{P', Q'}$, and this is of course inverse to $\psi$.
\end{proof}

In the case that $\cI$ is not closed under direct sums, there is still a formula for the admissible skein module. Let $\mathcal P$ be the poset category of objects in $\cI$ ordered by $P\leq Q$ if $P$ is a retract of $Q$. The skein module $\SM(H_g)$, in this case, is isomorphic to the following quotient: 
$$\frac{\underset{P_1,\dots, P_g\in {\mathcal P}^g}{\colim}\End(P_1\otimes\dots\otimes P_g)}{[\colim\Hom(P_1\otimes\dots\otimes P_g\to Q_1\otimes\dots\otimes Q_g),\colim\Hom(Q_1\to P_1)\otimes \dots\otimes\Hom(Q_g\to P_g)]}.$$
The dual of the admissible skein module $\SM(H_g)$ may be expressed more cleanly in this case. We describe $\SM(H_g)^*$ families of trace-like functions on $\cI$. This aligns with our expectations because the skein module $\SM(H_1)$ of the solid torus satisfies $\SM(H_1)^*\cong \{\text{cyclic traces on }\cI\}$.

    \begin{remark}
        From stacking, we have an induced (non-unital) associative algebra structure on the skein module for a genus $g$ handlebody. Let $P, Q\in R$. For two morphisms $f\in \End(P^{\otimes g})$ and $g\in \End(Q^{\otimes g})$, letting $\phi:P \otimes Q\to W$ be an isomorphism for some representative $W\in R$. Further, let $\sigma: P^{\otimes g} \otimes Q^{\otimes g}\to (P \otimes Q)^{\otimes g}$ denote the map that sends each $P$ above $Q$ at each crossing. Then, letting $[\cdot ]$ denote the equivalence classes represented by a map, $[f]\cdot [g]=[\phi\circ \sigma \circ (f\otimes g) \circ \sigma^{-1} \circ \phi^{-1}]$. This defines a well-defined algebra structure since any morphism applied to a single strand can slide past $\sigma$ and the merging morphism cancels with its inverse in the skein module.
    \end{remark}

\begin{figure}[ht]
  \centering
  \providecommand{\id}{\mathrm{id}}
  $[f] \cdot [g] \;=\;
  \left[\;
  \vcenter{\hbox{\begin{overpic}[abs, unit=1mm, scale = .25]{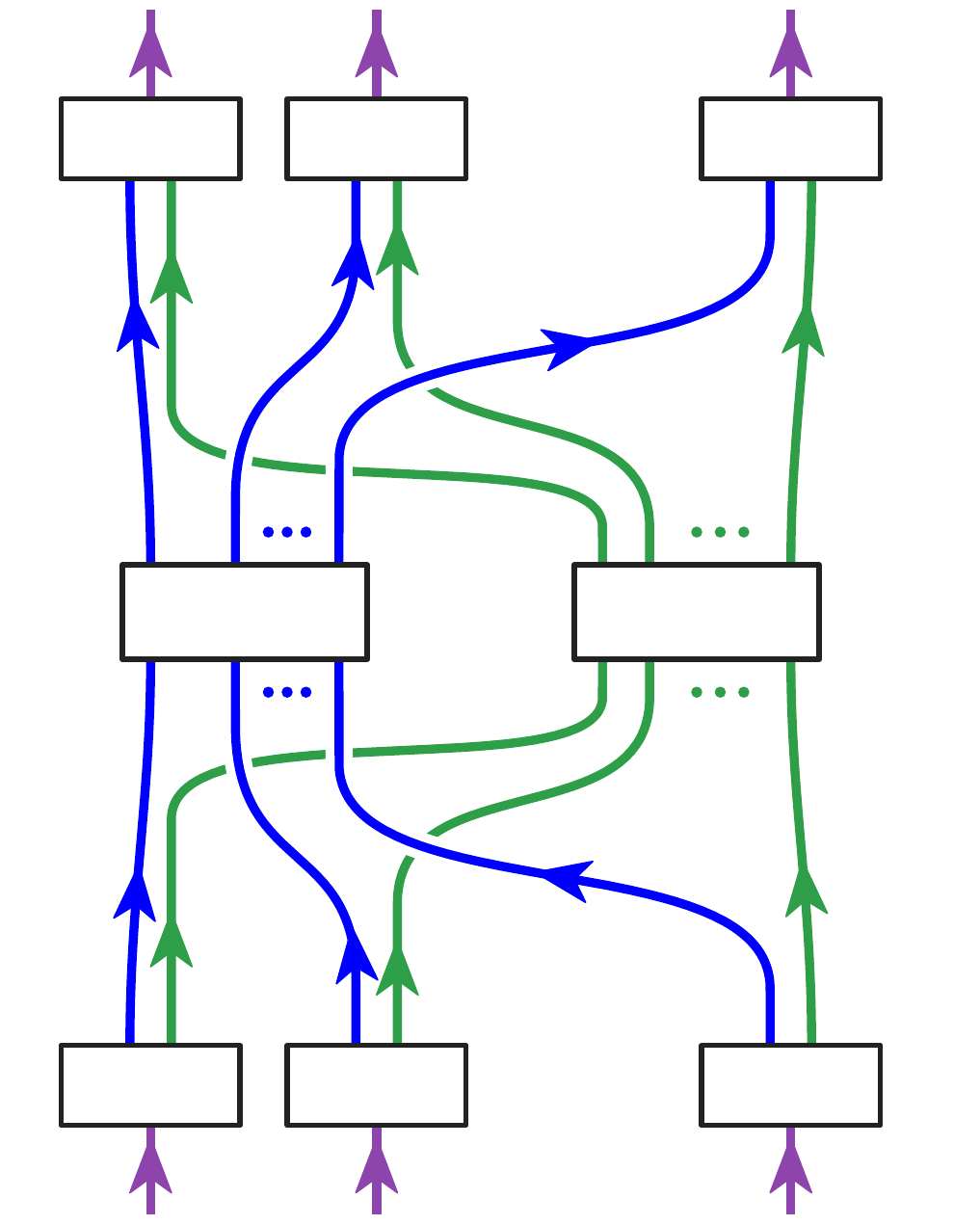}
      \put(8.7,41.5){\makebox(0,0)[l]{\scriptsize $Q$}}
      \put(8.7,12.3){\makebox(0,0)[l]{\scriptsize $Q$}}
      \put(18.6,42.7){\makebox(0,0)[l]{\scriptsize $Q$}}
      \put(18.6,11.1){\makebox(0,0)[l]{\scriptsize $Q$}}
      \put(36.5,39.2){\makebox(0,0)[l]{\scriptsize $Q$}}
      \put(36.5,14.6){\makebox(0,0)[l]{\scriptsize $Q$}}
      \put(4.8,39.4){\makebox(0,0)[r]{\scriptsize $P$}}
      \put(4.8,14.4){\makebox(0,0)[r]{\scriptsize $P$}}
      \put(14.4,42.1){\makebox(0,0)[r]{\scriptsize $P$}}
      \put(14.4,11.6){\makebox(0,0)[r]{\scriptsize $P$}}
      \put(25.3,40.3){\makebox(0,0){\scriptsize $P$}}
      \put(25.3,13.5){\makebox(0,0){\scriptsize $P$}}
      \put(7.2,51.4){\makebox(0,0)[l]{\tiny $W$}}
      \put(6.6,47.7){\makebox(0,0){\scriptsize $\phi$}}
      \put(7.4,2.4){\makebox(0,0)[l]{\tiny $W$}}
      \put(6.6,6.1){\makebox(0,0){\scriptsize $\phi^{-1}$}}
      \put(17.7,51.4){\makebox(0,0)[l]{\tiny $W$}}
      \put(16.5,47.7){\makebox(0,0){\scriptsize $\phi$}}
      \put(17.7,2.4){\makebox(0,0)[l]{\tiny $W$}}
      \put(16.5,6.1){\makebox(0,0){\scriptsize $\phi^{-1}$}}
      \put(35.9,51.4){\makebox(0,0)[l]{\tiny $W$}}
      \put(34.7,47.7){\makebox(0,0){\scriptsize $\phi$}}
      \put(35.9,2.4){\makebox(0,0)[l]{\tiny $W$}}
      \put(34.7,6.1){\makebox(0,0){\scriptsize $\phi^{-1}$}}
      \put(10.8,26.9){\makebox(0,0){\footnotesize $f$}}
      \put(30.6,26.9){\makebox(0,0){\scriptsize $g$}}
  \end{overpic} }}
  \;\right]$
  \caption{Figure illustrating the algebra structure of the handlebody skein module.} %
  \label{fig:Algebra Structure}
\end{figure}

\begin{definition}
    A $g$-semitrace $t$ on a tensor ideal $\cI\subseteq \cC$ is a family of linear functions 
    $$\{t_{V_1,\dots,V_g}:\End(V_1\otimes\dots\otimes V_g)\to \bk\mid V_1,\dots, V_g\in \cI\},$$ 
    which satisfies the following property:
    \begin{itemize}
        \item $t_{V_1,\dots,V_g}((f_1\otimes \dots\otimes f_g)\circ F) = t_{W_1,\dots,W_g}(F\circ(f_1\otimes \dots\otimes f_g))$ for any $f_i:W_i\to V_i$ and $F:V_1\otimes\dots\otimes V_g\to W_1\otimes\dots\otimes W_g$.
    \end{itemize}
\end{definition}

\begin{prop}\label{prop:semitrace-correspondence}
    Let $\cI$ be any tensor ideal of $\cC$. The admissible skein module of the genus $g$ handlebody $H_g$ satisfies
      $$\SM(H_g)^*\cong \{g\text{-semitraces on }\cI\}.$$
\end{prop}

Observe that $1$-semitraces are precisely cyclic traces. If $\cC$ is finite and $\cI$ is the projective ideal of $\cC$, we may express the admissible skein modules of the solid torus in terms of the minimal projective generator $G$. In particular, as a vector space, $\SM(S^1\times D^2)\cong \End(G)^{\ab}$, the abelianization of the projective Auslander algebra (see \cite{jubeir2025reconstructionfinitetensorcategories}). However, they do not agree as algebras unless $\cC$ is semisimple, as the admissible skein algebra is non-unital. The dual may be written as
$$\SM(S^1\times D^2)^*\cong\{\text{cyclic traces on }G\}.$$
\begin{remark}
Note: we give a full proof and do not defer to the arguments in the proof of Theorem \ref{limit-skein-module-iso}. This proof illustrates the utility of considering the duals of the admissible skein modules, rather than the modules themselves. When studying the lens spaces, we once again consider the duals. A statement in terms of limits is still possible, but obtaining the simplest form of limit as we have in Theorem \ref{limit-skein-module-iso} once again is cumbersome. While one can also express $g$-semitraces and $L(p,q)$-modified traces (see Subsection \ref{subsec:admissible-lens}) in terms of square-free objects, indexing over all objects is cleaner.
\end{remark}
\begin{proof}
First, note that at the cost of replacing crossings with coupons colored by the braiding, as in \cite[Corollary 3.2]{admissible} and the definition of the Reshetikhin-Turaev functor (Subsection \ref{subsection:icolouredgraphs}), it suffices to calculate the admissible skein module of the $g$-punctured disk $D^2_g$. Define a set of cutting segments $r_i$ to be a clockwise-oriented set of non-intersecting segments from the $i$th puncture to the boundary that intersect $\Gamma$ transversely, that does not intersect any coupons of $\Gamma$, and intersect with at least one $\cI$-colored edge of $\Gamma$. For each segment $r_i$, we may combine all colored strands intersecting that segment into a single strand colored by a single element in $\{P_i\}_{i=1}^g\subset \cI$. Since the $g$-punctured disk excluding these segments is contractible and homeomorphic to a disk, the skein contained in $D^2_g\setminus \bigcup r_i$ is equivalent in $\mcC$ to a map in $F_\Gamma \in \End(P_1\otimes \cdots \otimes P_g)$. 

Given a $g$-semitrace $t$, we define the element $T$ of $\SM(H_g)^*$ by $T(\Gamma)=t(F_\Gamma)$.
To prove that this does not depend on our choice of segment, suppose that two choices of the $i$th segment differ. Since fixing the remaining segments yields a once-punctured disk as in \cite{admissible}, then by cutting we obtain that the two are related by a cyclic relation on the $i$th component, which yield the same result since $t$ is $g$-semitracial. Since any box representing a skein relation between $\Gamma$ and $\Gamma'$ can be assumed to not intersect any segment, the two morphisms generated by $F_\Gamma$ and $F_{\Gamma'}$ are equal in $\mcC$.

Given an element $T: \SM(H_g)\to \Bbbk$, define the $g$-semitrace $t$ to be the evaluation of $\hat{\Gamma}_F$ given by the map $F$ connected by strands through each handle once. Since $T$ must be preserved under skein relations, the two skeins generated by $((\id_{V_1}\otimes \dots\otimes f\otimes \dots\otimes \id_{V_g})\circ F)$ and $ (F\circ(\id_{V_1}\otimes \dots\otimes f\otimes \dots\otimes \id_{V_g})) $ are equivalent.
\end{proof}
\begin{remark}  
    Proposition \ref{prop:semitrace-correspondence} is a corollary of Theorem \ref{limit-skein-module-iso} in the case that the tensor ideal $\cI$ is closed under direct sums, such as for the projective ideal. We provide below an alternate proof in this case.

    It suffices to show that there is a linear isomorphism between $g$-semitraces and linear functionals on the limit $S := \lim_{R} \End(P^{\otimes g}) / Q_P$, where  $Q_P = [\End(P^{\otimes g}), \End(P)^{\otimes g}].$ Given a $g$-semitrace $t$, we may define a linear functional $T$ using the universal property of limits, noting that the $g$-semitracial condition corresponds with quotienting by $Q_{P}$. 
    
    For a linear functional $T$ on $S$, we define a $g$-semitrace $t$ as follows: For $P_1,\dots, P_g$, choose $P\in R$ which contains every $P_i$ as a direct summand with projection $\pi_i:P\to P_i$ and inclusion $\iota_i:P_i\to P$. Let $\pi = \pi_1\otimes \dots \otimes \pi_g$ and $\iota = \iota_1\otimes\dots\otimes \iota_g$. As a colimit, $S$ comes with a map $C_{P}:\End(P^{\otimes g}) / Q_{P}\to S$. Define $t_{P_1,\dots, P_g}$ to be the following composition:
    $$\begin{tikzcd}
        \End(P_1\otimes\dots\otimes P_g)\arrow[r, "\iota\circ-\circ\pi"] & \End(P^{\otimes g})\arrow[r, "\mathrm{quot.}"] &\End(P^{\otimes g}) / Q_{P}\arrow[r, "c_{P}"] & S\arrow[r, "T"] & \bk
    \end{tikzcd}$$
    A priori, this depends on the choices of $\pi_i$ and $\iota_i$. However, by quotienting by $[\End(P^{\otimes g}), \End(P)^{\otimes g}]$, we are allowed the following manipulations for different choices of $\pi_i'$ and $\iota_i'$:
    $$T\circ p_{P}(\iota\circ F\circ\pi) = T\circ p_{P}(\iota\circ\pi'\circ\iota'\circ F\circ \pi'\circ\iota'\circ\pi) = T\circ p_{P}(\iota'\circ\pi\circ\iota\circ\pi'\circ\iota'\circ F\circ\pi') = T\circ p_{P}(\iota'\circ F\circ\pi').$$
    By a similar argument, we see that $t$ is indeed a $g$-semitrace. Moreover, the two constructions are clearly inverses of each other. 
\end{remark}

\begin{theorem}\label{thm:2handle}
    Let $M$ be an oriented $3$-manifold with boundary $\partial M$ and let $M'$ be obtained by attaching a $2$-handle to $M$ along a simple closed curve $K$ in $\partial M$. Then $i_*: \SM(M)\to \SM(M')$ is an epimorphism.
\end{theorem}
\begin{theorem}\label{thm:3handle}
Suppose $M'$ is obtained from $M$ by attaching a $3$-handle and let $i:M\to M'$ be the associated embedding, then $i_*$ is an isomorphism $\SM(M)\to \SM(M')$.
\end{theorem}
\begin{proof}
    The proofs of Theorems \ref{thm:2handle} and \ref{thm:3handle} follow identically to the classical skein module case; see \cite{fundamentals}.
\end{proof}
\subsection{Computation for lens spaces}\label{subsec:admissible-lens}
Let $p\geq q$ be coprime non-negative integers. Given a cyclic trace $t$ on an ideal $\cI$ of a ribbon tensor category $\cC$, we say that $t$ is \textit{$L(p,q)$-modified} if it satisfies the following two equivalent conditions for any $P\in\cI$ and $X\in\cC$:
\begin{enumerate}
    \item for any $f:P\otimes X\to P\otimes X$, the right partial trace satisfies the following:
    \begin{align*}
        t_P\left(\vcenter{\hbox{\begin{overpic}[scale = 0.04]{rightpartialtrace1}
\put(14,32){\tiny{$f$}}
\put(6,13){\tiny{$P$}}
\put(6,50){\tiny{$P$}}
\put(41,32){\tiny{$X$}}
\end{overpic} }}\right) &= t_{P\otimes X^{\otimes p}}\left(\vcenter{\hbox{\begin{overpic}[abs, unit=1mm, scale = .25]{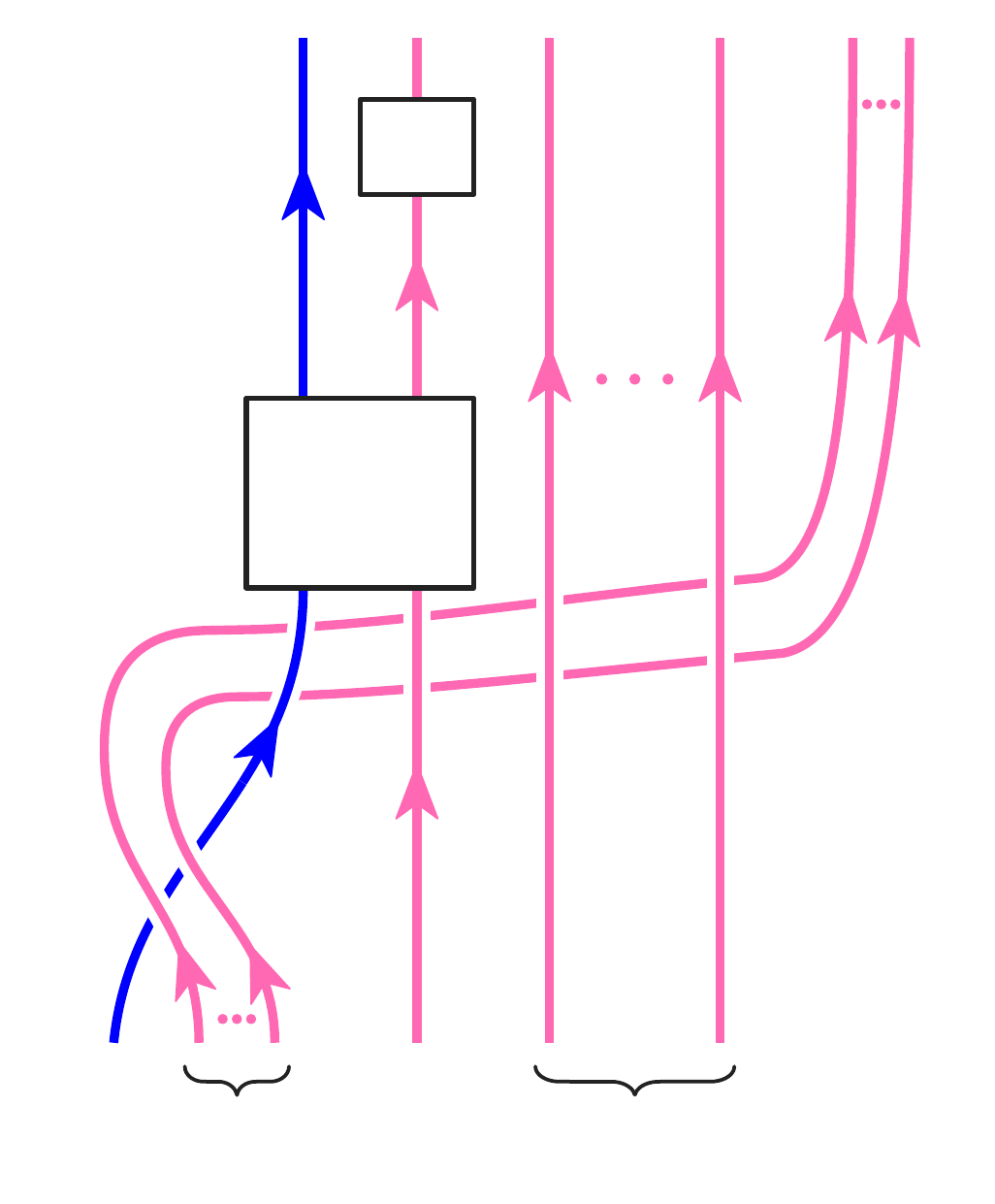}
    \put(15.7,31.0){\makebox(0,0){$f$}}
    \put(18.2,46.1){\makebox(0,0){\scriptsize{$\theta_X^{q}$}}}
    \put(2.6,10.5){\makebox(0,0){\footnotesize{$P$}}}
    \put(41.8,37.6){\makebox(0,0){\footnotesize{$X$}}}
    \put(10.3,2.6){\makebox(0,0){\footnotesize{$q$}}}
    \put(27.7,2.6){\makebox(0,0){\footnotesize{$p-q-1$}}}
  \end{overpic}}}\right),
    \end{align*}
    \item for any $f:X\otimes P\to X\otimes P$, the left partial trace satisfies the following:
    \begin{align*}
        t_P\left(\vcenter{\hbox{\reflectbox{\begin{overpic}[scale = 0.04]{rightpartialtrace1}
\put(14,32){\tiny{\reflectbox{$f$}}}
\put(6,13){\tiny{\reflectbox{$P$}}}
\put(6,50){\tiny{\reflectbox{$P$}}}
\put(41,32){\tiny{\reflectbox{$X$}}}
\end{overpic} }}}\right) &= t_{X^{\otimes p}\otimes P}\left(\vcenter{\hbox{\begin{overpic}[abs, unit=1mm, scale = .25]{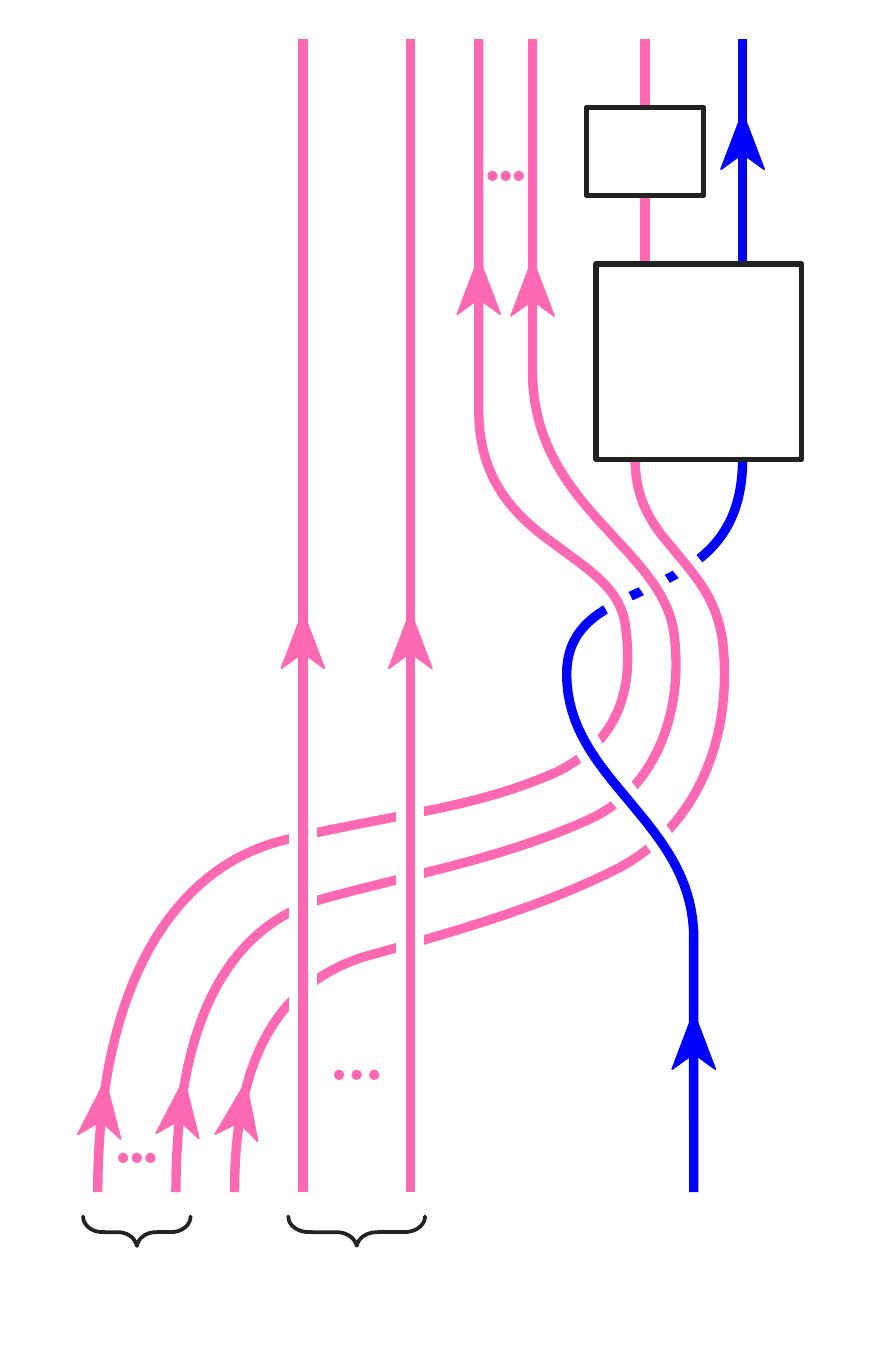}
    \put(29.6,42.2){\makebox(0,0){$f$}}
    \put(27.3,51.1){\makebox(0,0){\scriptsize{$\theta_X^{q}$}}}
    \put(31.6,13.5){\makebox(0,0){\footnotesize{$P$}}}
    \put(10.6,31.9){\makebox(0,0){\footnotesize{$X$}}}
    \put(5.8,2.6){\makebox(0,0){\footnotesize{$q-1$}}}
    \put(15.1,2.6){\makebox(0,0){\footnotesize{$p-q$}}}
  \end{overpic}}}\right).
    \end{align*}
    \end{enumerate}
    These conditions can be seen as equivalent using cyclicity and isotopy. Note that $L(1,0)$-modified traces are the usual modified traces because there is no $\theta$-term nor any braidings and a single unlinked strand which attaches to $f$. In particular, the following theorem is already known for $L(1, 0)\cong S^3$ by \cite{admissible}.
\begin{theorem}\label{thm:admissible-lens}
    Under the same assumptions as Theorem \ref{limit-skein-module-iso}, the admissible skein module of the lens space $L(p, q)$ is given by
    $$\SM(L(p, q))^*\cong \{L(p,q)\text{-modified traces}\}.$$
\end{theorem}
\begin{proof}
    Recall that there is a surjective map $\SM(M)\to \SM(N)$ when $N$ is obtained from $M$ by gluing a 2-handle $D^2\times D^1$ because any ribbon graph on the handle $D^2\times D^1$ may be isotoped to its boundary annulus $S^1\times D^1$. Moreover, $\SM(M)\cong \SM(N)$ when $N$ is obtained from $M$ by gluing a 3-handle $D^3\times D^0$. The lens space $L(p,q)$ decomposes as a gluing
    $$L(p,q) = T\cup (D^2\times D^1)\cup (D^3\times D^0),$$
    where $T$ is a solid torus. Therefore, there is a surjection $p:\SM(T)\proj \SM(L(p,q))$. In particular, there is also an inclusion $p^*:\SM(L(p,q))^*\to\SM(S^1\times D^2)^*$. Label the solid torus $T$ and the 2-handle $B$ so that $\SM(L(p,q))\cong \SM(T\cup B)$. Let $A$ denote the boundary annulus of $B$ which is glued onto $T$.
    
    A cyclic trace $t$ on $\cI$ is in $\SM(L(p,q))^*$ if $t(F_\Gamma) = t(F_{\Gamma'})$ whenever $\Gamma, \Gamma'$ are admissible ribbon graphs inside $T$ which are equivalent in $L(p,q)\supset S^1\times D^2$ by a sequence of skein relations and isotopies. By a handle slide move, we can see that $t$ must be an $L(p,q)$-modified trace (the effect of this handle-slide is visualized in Figure \ref{fig:handleslide}). The rest of this proof is to argue that handle slide moves on an admissible ribbon graph are entirely captured by this condition together with preservation under cyclic permutations.

        \begin{figure}[ht]
  \centering
  \begin{subfigure}{\textwidth}
    \centering
    $\vcenter{\hbox{\begin{overpic}[abs, unit=1mm, scale = .16]{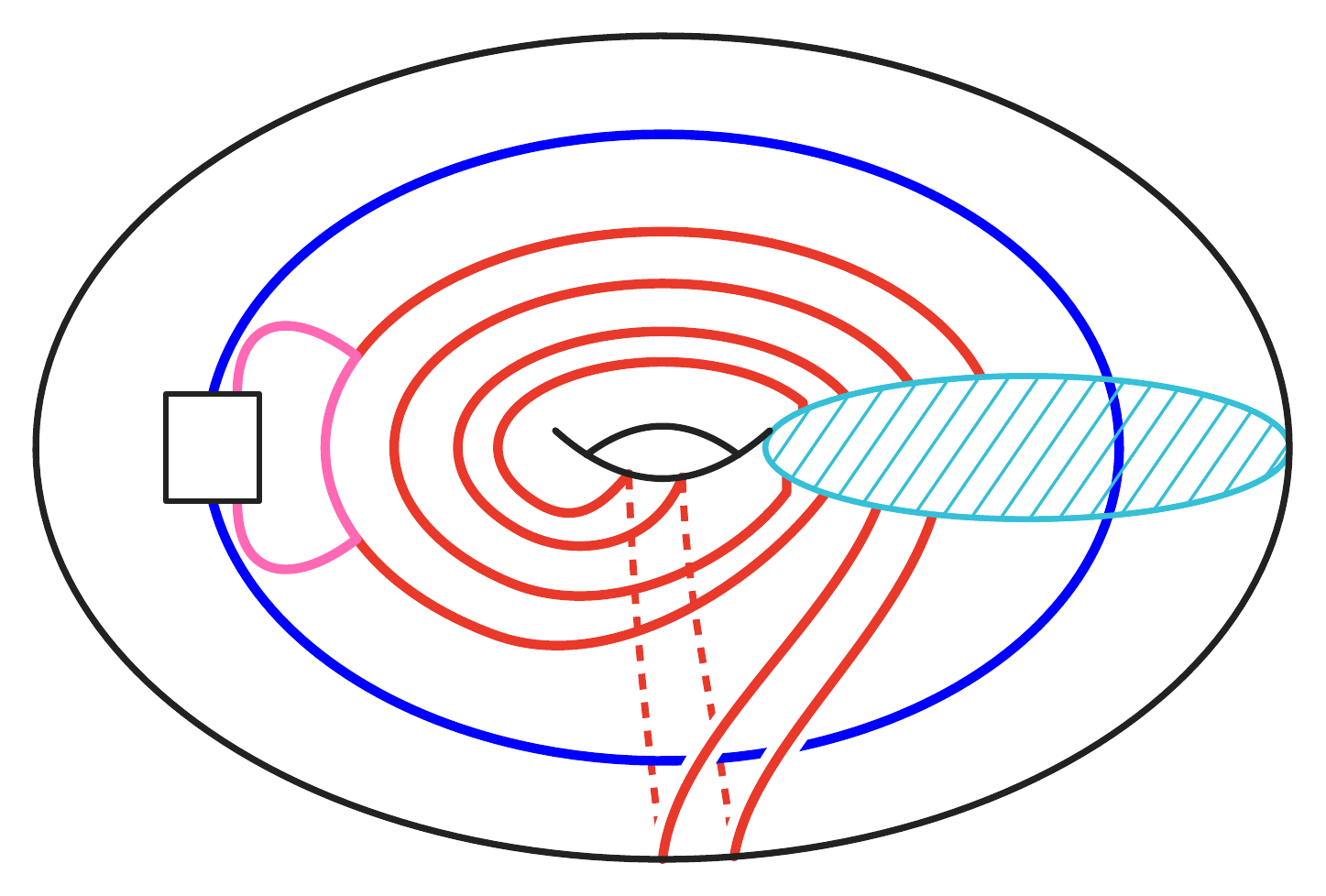}
      \put(6.3,13.2){\makebox(0,0){\footnotesize{$f$}}}
    \end{overpic} }}
    \;\vcenter{\hbox{\begin{overpic}[height=2.6cm]{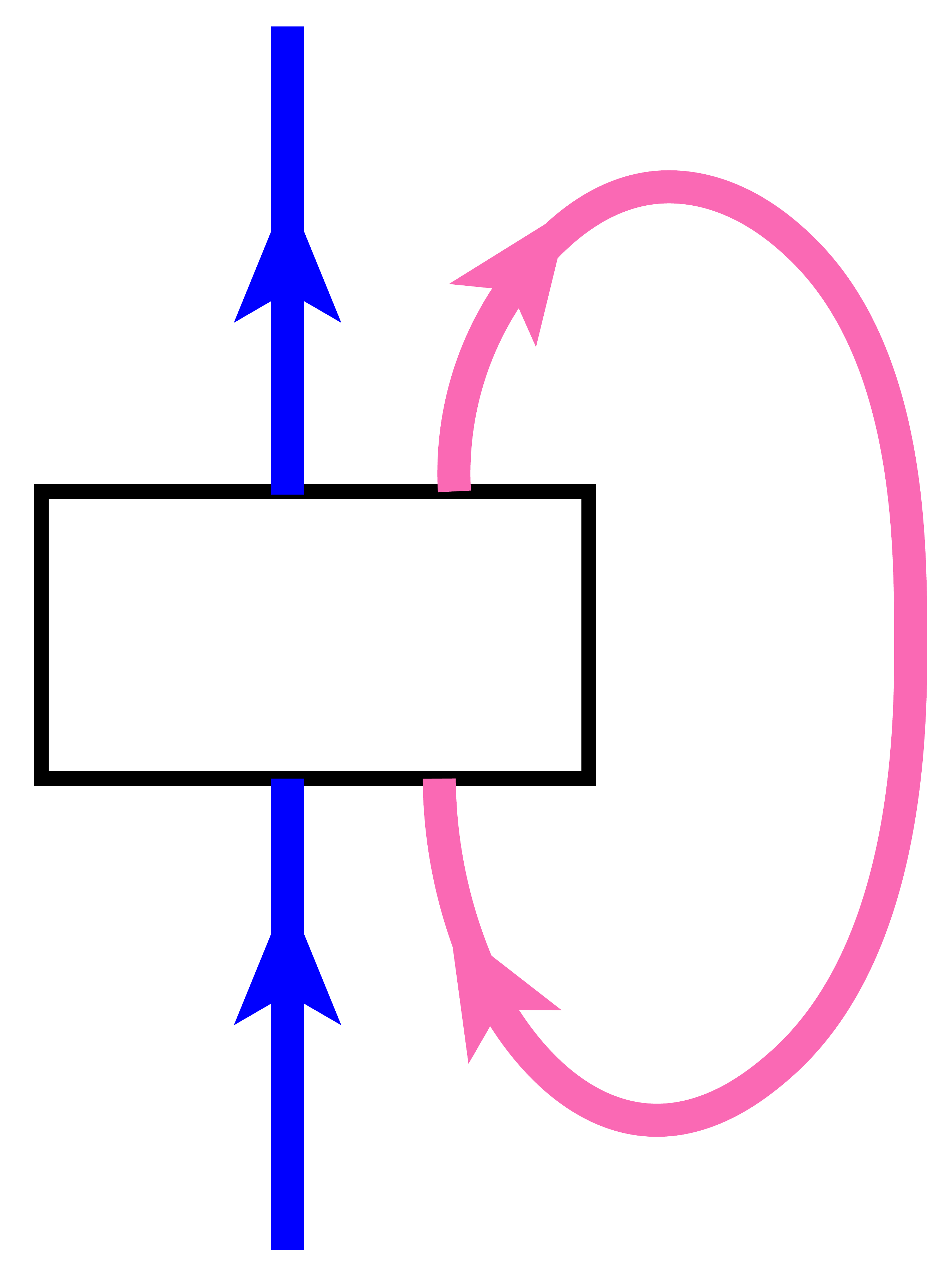}
      \put(11,65){\makebox(0,0){\footnotesize{$P$}}}
      \put(11,10){\makebox(0,0){\footnotesize{$P$}}}
      \put(17,36){\makebox(0,0){\footnotesize{$f$}}}
      \put(60,42){\makebox(0,0){\footnotesize{$X$}}}
    \end{overpic} }}$
    \caption{A skein $\Gamma$ (left) and its corresponding endomorphism $F_\Gamma$ (right) before a handle-slide.}
    \label{fig:combined-a}
  \end{subfigure}
 
  \bigskip
 
  \begin{subfigure}{\textwidth}
    \centering
    $
    \vcenter{\hbox{\begin{overpic}[abs, unit=1mm, scale = .16]{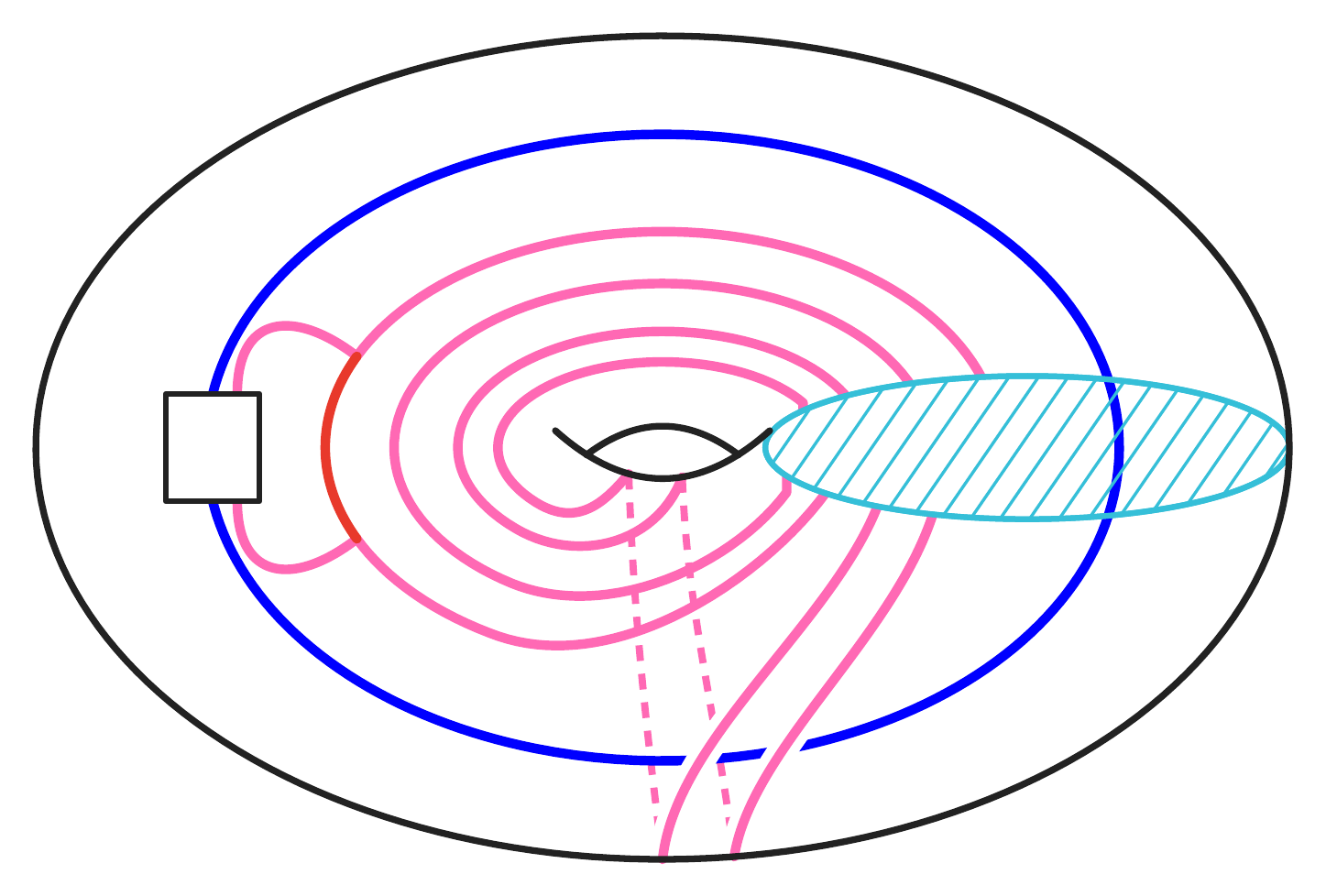}
      \put(6.3,13.2){\makebox(0,0){\footnotesize{$f$}}}
    \end{overpic} }}
    \;\vcenter{\hbox{\begin{overpic}[abs, unit=1mm, scale = .25]{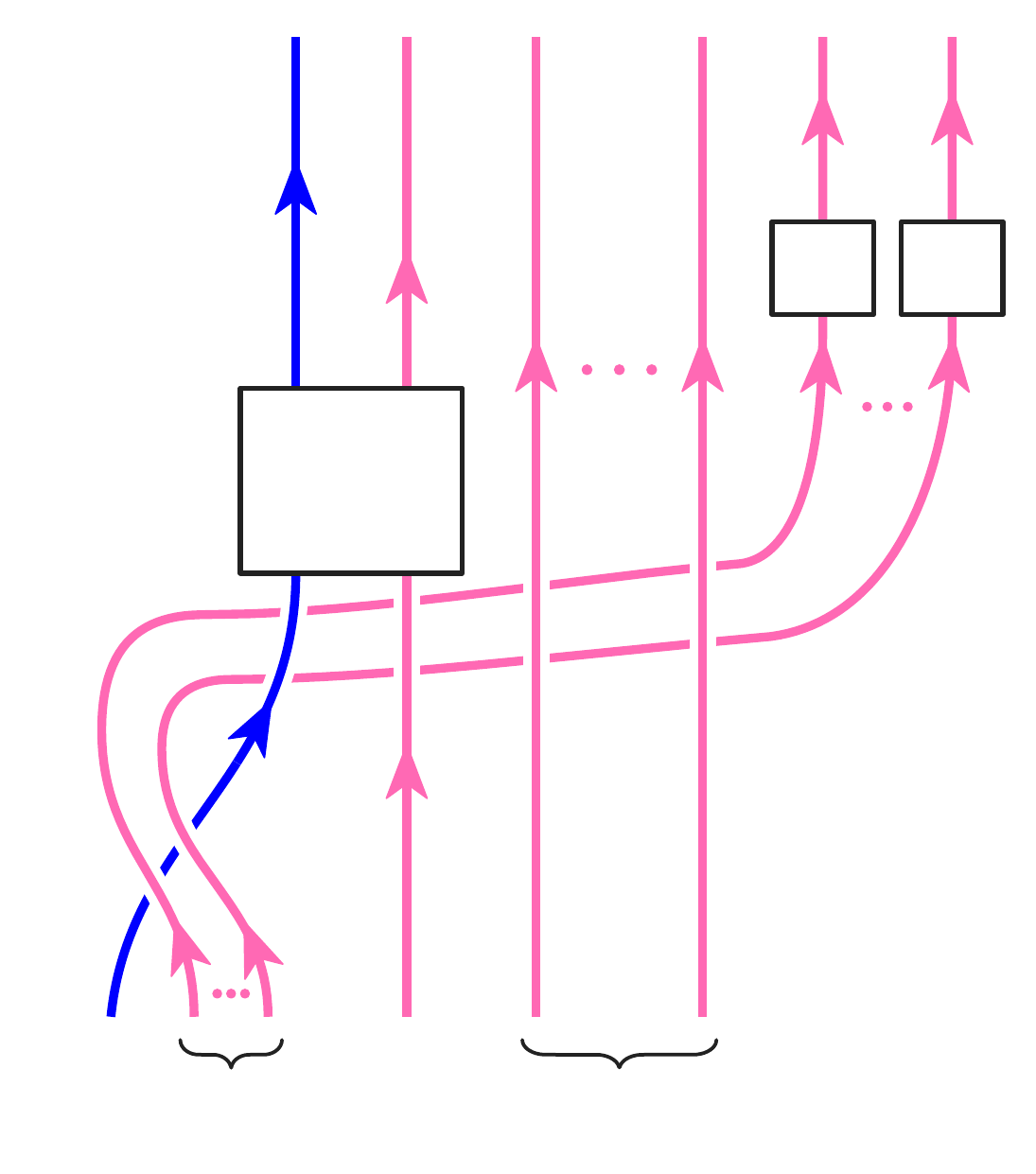}
      \put(15.7,31.0){\makebox(0,0){$f$}}
      \put(36.8,40.5){\makebox(0,0){\scriptsize{$\theta_X$}}}
      \put(42.6,40.5){\makebox(0,0){\scriptsize{$\theta_X$}}}
      \put(2.6,10.5){\makebox(0,0){\footnotesize{$P$}}}
      \put(33.6,13.0){\makebox(0,0){\footnotesize{$X$}}}
      \put(10.3,2.6){\makebox(0,0){\footnotesize{$q$}}}
      \put(27.7,2.6){\makebox(0,0){\footnotesize{$p-q-1$}}}
    \end{overpic} }}
    $
    \caption{A skein $\Gamma'$ (left) and its corresponding endomorphism $F_{\Gamma'}$ after a handle-slide.}
    \label{fig:combined-b}
  \end{subfigure}
  \caption{}
        \label{fig:handleslide}
\end{figure}

    Suppose $\Gamma, \Gamma'$ are admissible ribbon graphs inside $T$ which are equivalent in $L(p,q)\supset S^1\times D^2$ by a sequence of skein relations and isotopies. Denote this sequence $\Gamma=\Gamma_0,\Gamma_1,\dots,\Gamma_n=\Gamma'$. We may assume without loss of generality that each $\Gamma_k$ lies completely within $\text{int}(T)$ and that each $\Gamma_k$ does not have a coupon intersecting the cut made in Theorem \ref{limit-skein-module-iso}. Under these assumptions, we must show that $t(F_{\Gamma_k}) = t(F_{\Gamma_{k+1}})$ for all $k$. That is, $t$ is independent of isotopy and skein relations on the $\Gamma_k$.
    
    If the move $\Gamma_k$ to $\Gamma_{k+1}$ stays completely within $T$, then $t(F_{\Gamma_k})=t(F_{\Gamma_{k+1}})$ by the argument in Proposition \ref{prop:semitrace-correspondence}. We may assume that any skein relation happens completely within $T$, so we need only consider what happens under isotopies which pass through $B$.
    
    Suppose now that the move $\Gamma_k$ to $\Gamma_{k+1}$ is an isotopy which passes through $B$. Assuming general position, there is a segment $r$ from the inner shell to the outer shell of $B$ which does not intersect $A$ and, during isotopies between the $\Gamma_k$, intersects no coupons and there is at most one point of intersection with a ribbon at a time. Subdivide the isotopy from $\Gamma_k$ to $\Gamma_{k+1}$ so that there is at most one point of intersection with a ribbon and $r$ during the isotopy. We may assume without loss of generality that during an intersection of a ribbon with $r$, no other component moves and we may subdivide the isotopy further so that each part of the isotopy either does not intersect $r$ or has a single ribbon moving across $r$ and the rest of the ribbon graph remains fixed. Let $\Gamma_k=\hat\Gamma_{0},\dots,\hat\Gamma_{m}=\Gamma_{k+1}$ be such a subdivision. Without loss of generality, there is a small closed tubular neighborhood $N(r)$ of $r$ which intersects no $\hat\Gamma_k$ and also only intersects one ribbon at a time during isotopies. 
    
    There is an isotopy $J$ of $\id_{T\cup B}$ to an auto-homeomorphism of $T\cup B$ for which the image of $N(r)$ is $B$. Thus, we may assume without loss of generality that $N(r) = B$. Suppose the isotopy $\hat\Gamma_\ell$ to $\hat\Gamma_{\ell+1}$ does not intersect $r$, then in particular, it does not intersect $N(r)=B$ and hence is completely contained in the solid torus, so we are reduced to the case handled by Proposition \ref{prop:semitrace-correspondence}.

    Now, suppose the isotopy $\hat\Gamma_\ell$ to $\hat\Gamma_{\ell+1}$ is of a single ribbon whose central circle/line segment $s$ intersects $r$ in exactly one place at a time. We may assume that $B\cap s$ is always connected. We further assume the intersection between $s$ and $A$ is transverse and nonempty at all times $t$ except $t=0,1$ (otherwise, we may shrink $N(r)$ so that $s\not\subset N(r)$ and then shrink the time interval). The disk slice of the solid torus made during Proposition \ref{prop:semitrace-correspondence} extends to $p$ wedges $W_1,\dots, W_p$ on the 2-handle as illustrated in Figure \ref{fig:slice}. Without loss of generality, there are two fixed endpoints $x,y\in s\cap A$ of $s$ (at times $t=0,1$, we allow a non-transverse intersection where $s$ lies flat against $A$). At time $t=0$, we may assume $s\cap A\setminus \bigcup W_i$ is connected (otherwise, we can drag the ribbon along the annulus so that this is true). Moreover, we may assume the ribbon remains fixed outside of $B$ throughout the isotopy. 
    \begin{figure}[ht]
  \centering
  $\vcenter{\hbox{\begin{overpic}[abs, unit=1mm, scale = .2]{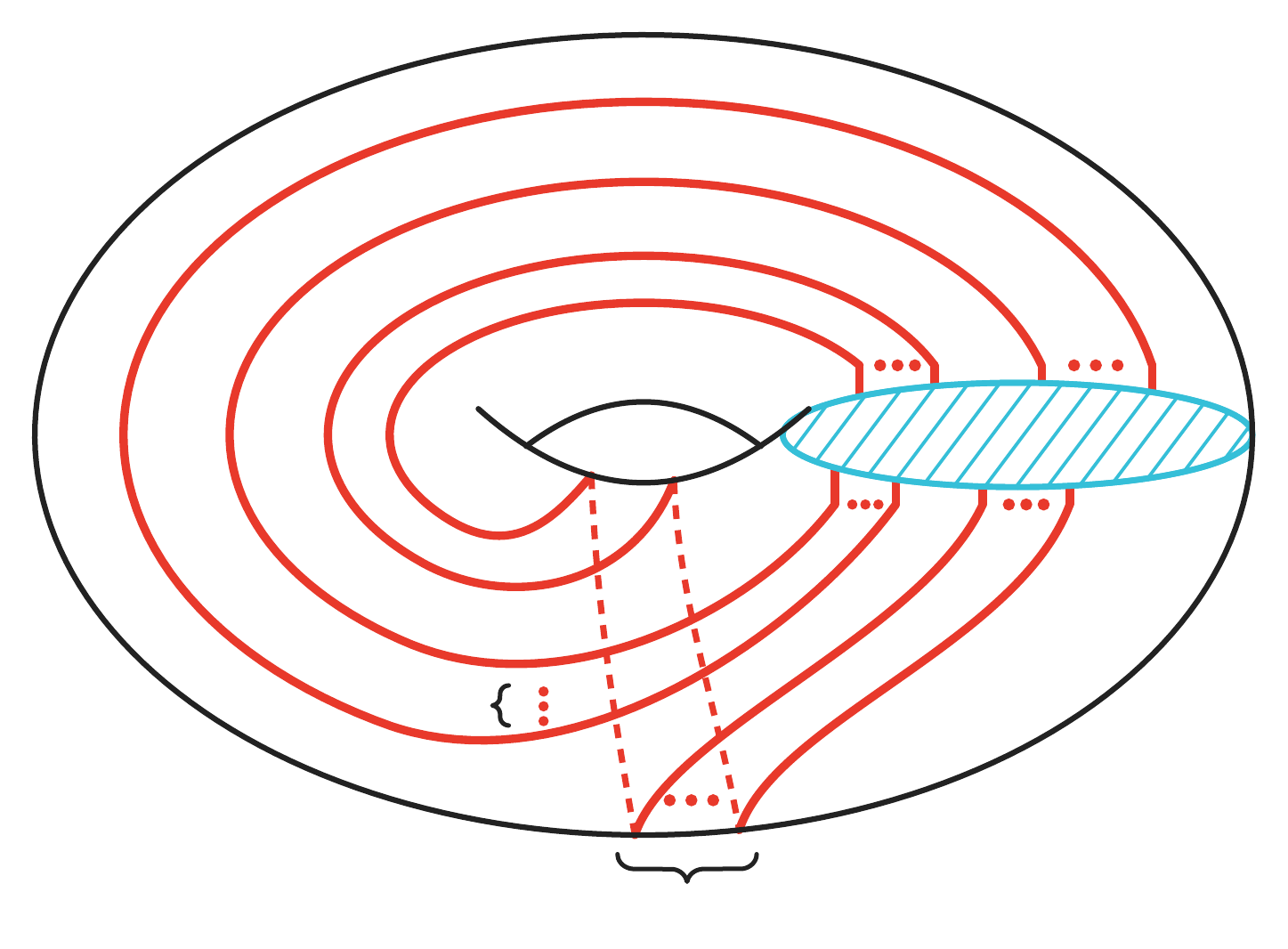}
    \put(3.5,32.5){\makebox(0,0){$T$}}
    \put(15.5,9.85){\makebox(0,0){\tiny{$p-q$}}}
    \put(26.1,1.2){\makebox(0,0){\scriptsize{$q$}}}
  \end{overpic} }}
  \qquad
  \vcenter{\hbox{\begin{overpic}[abs, unit=1mm, scale = .2]{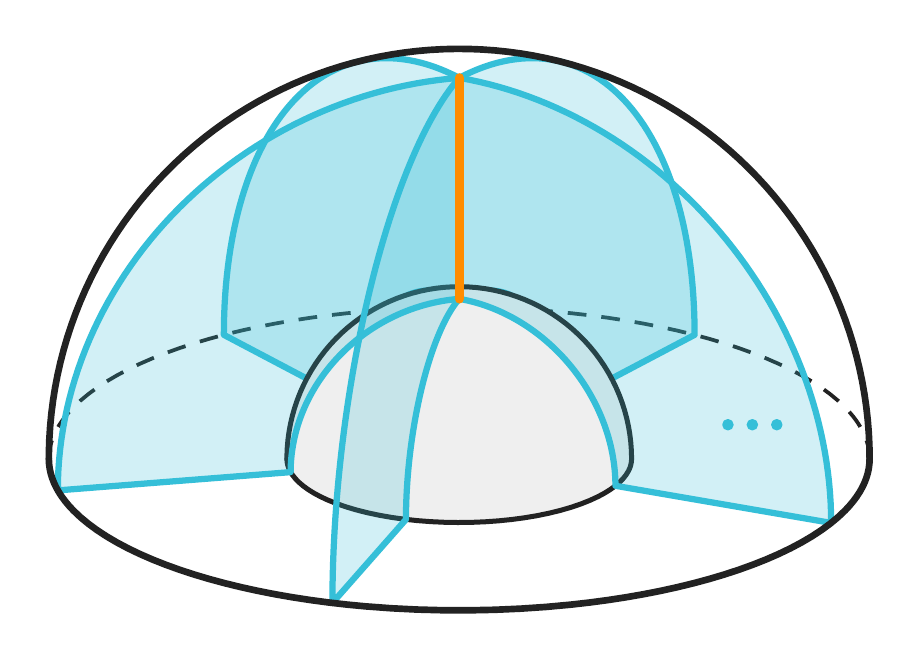}
    \put(2.5,20.0){\makebox(0,0){$B$}}
  \end{overpic} }}$
  \caption{The slice of the solid torus $T$ used to define $F_\Gamma$ extends to $p$ slices of the
    2-handle $B$, as the slice passes through the $(p,q)$-curve $p$ times.}
    \label{fig:slice}
\end{figure}
    
    At time $t=0$ and $t=1$, $s\cap A$ defines an element $[\gamma_1], [\gamma_2]\in\pi_1(A, x\to y)$ of the fundamental groupoid, which are represented by curves $\gamma_1, \gamma_2$ that contain no self-intersections. There are precisely two such elements. If $[\gamma_1]=[\gamma_2]$, then there is in fact an isotopy between $\gamma_1$ and $\gamma_2$ contained on $A\subseteq T$. Such an isotopy may be made completely lying on the solid torus, so performing this isotopy on $\hat\Gamma_{\ell+1}$ preserves $t$ and thus $t(\hat\Gamma_{\ell+1}) = t(\hat\Gamma_{\ell})$. 
    
    Finally, if $[\gamma_1]\neq[\gamma_2]$, then there is an isotopy of $\gamma_2$ contained on $A\subseteq T$ to a curve which passes through each $W_i$ exactly once. Again such an isotopy preserves $t$, so we may assume without loss of generality that $\gamma_2$ is this curve. That is, $\hat\Gamma_{\ell+1}$ and $\hat\Gamma_{\ell}$ differ by the obvious handle slide of a single ribbon. In particular, $\hat\Gamma_\ell$ and $\hat\Gamma_{\ell+1}$ can be assumed to be of the forms seen in Figure \ref{fig:handleslide} (where the blue ribbon may represent many strands and the box $f$ may represent a more complex ribbon graph). In particular, the difference between $F_{\hat\Gamma_\ell}$ and $F_{\hat\Gamma_{\ell+1}}$ aligns with the skein move preserved by the $L(p,q)$-modified trace. This handles all possible isotopies, so the result follows.
\end{proof}

One may also express $\SM(L(p,q))$ as a quotient of the colimit of $\End(X)^{\ab}$ as in Theorem \ref{limit-skein-module-iso} or the more general formula afterwards, quotienting by the skein relation described in Figure \ref{fig:handleslide}. 

\begin{remark}
     For the $3$-manifold $S^1\times S^2\cong L(0,1)$, a similar argument holds. One can show $\SM(S^1\times S^2)^*$ corresponds with cyclic traces satisfying the following condition on right partial traces: 
     $$t_P\left(\vcenter{\hbox{\begin{overpic}[scale = 0.04]{rightpartialtrace1}
\put(14,32){\tiny{$f$}}
\put(6,13){\tiny{$P$}}
\put(6,50){\tiny{$P$}}
\put(41,32){\tiny{$X$}}
\end{overpic} }}\right) = t_P\left(\vcenter{\hbox{\begin{overpic}[abs, unit=1mm, scale = .25]{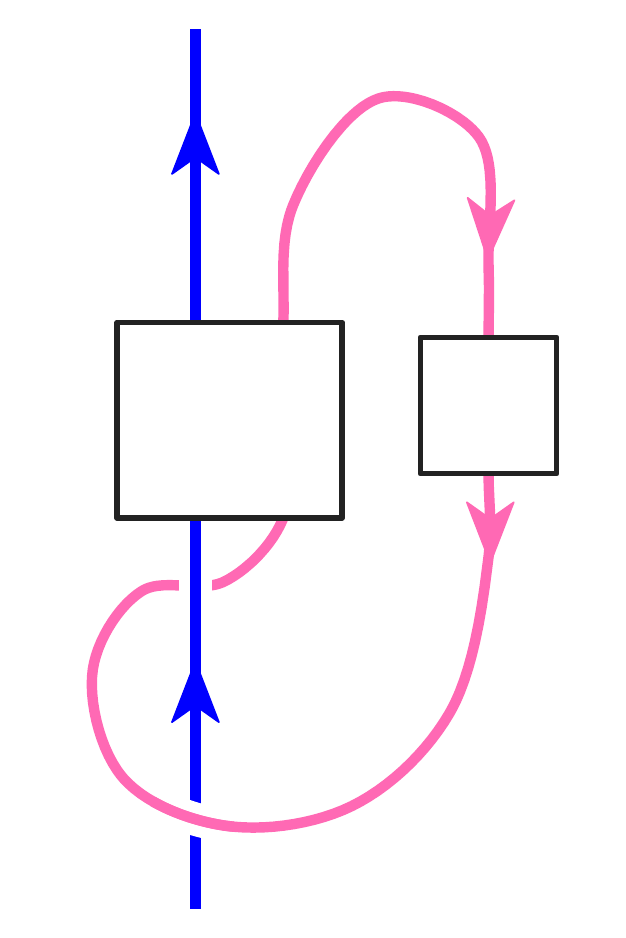}
    \put(9.7,21.9){\makebox(0,0){$f$}}
    \put(20.7,22.5){\makebox(0,0){\scriptsize{$\theta_{X^*}$}}}
    \put(6.0,33.5){\makebox(0,0){\footnotesize{$P$}}}
    \put(6.0,10.3){\makebox(0,0){\footnotesize{$P$}}}
    \put(22.5,31.4){\makebox(0,0){\footnotesize{$X$}}}
    \put(22.3,15.3){\makebox(0,0){\footnotesize{$X$}}}
  \end{overpic}}}\right).$$
\end{remark}

\section {Future Directions}\label{section:future}
By our computations, a priori, the invariants of $L(p,q)$ depend on both $p$ and $q$. However, there is reason to believe that this is not generally true. The Kauffman bracket skein module cannot distinguish the lens spaces $L(p,q)$ and $L(p, q')$. Theorem \ref{thm:K2-calculation} shows that the chromatic spherical invariant also does not necessarily distinguish the lens spaces. Thus, a natural question is if there is a general theory of which lens spaces these invariants do not distinguish. If $\cI$ is finite, $\SM(L(p,q))$ is a quotient of $\End(G)^{\ab}$ by Theorem \ref{limit-skein-module-iso}, where $G$ is a minimal generator of $\cI$. Because $\End(G)$ is finite-dimensional, the admissible skein module cannot distinguish general lens spaces in this case. Even if $\cI$ is infinite, $\SM(L(p,q))$ may be finite-dimensional. This is again exhibited in the semisimple case with the Kauffman bracket skein module $\mcS_{\mathrm{KB}}(L(p,q))$, which has dimension $\lfloor p/2\rfloor + 1$. Under which conditions on $\cI$ is the admissible skein module $\SM(L(p,q))$ finite-dimensional? In the particular case of $\cC=\Rep(U_q(\sl_2))$ at a root of unity, can we find an explicit basis for $\cS_{\Proj(\cC)}(L(p,q))$?

\bibliographystyle{abbrv}
\bibliography{references}

\end{document}